\documentclass[11pt]{amsart}

\usepackage{mystyle}

\begin{document}

\title[Thresholds for Latin squares and Steiner triple systems]
{
  Thresholds for Latin squares and Steiner triple systems: Bounds within a logarithmic factor
}

\email{\{D.Y.Kang.1, D.Kuhn, D.Osthus\}@bham.ac.uk, tom.kelly@gatech.edu, abhishekmethuku@gmail.com}
\address{School of Mathematics, University of Birmingham,
Edgbaston, Birmingham, United Kingdom}
\address{School of Mathematics, Georgia Institute of Technology, Atlanta, GA, USA}
\address{Department of Mathematics, ETH Z\"urich, Z\"urich, Switzerland}

\author[Kang]{Dong Yeap Kang}

\author[Kelly]{Tom Kelly}

\author[K\"uhn]{Daniela K\"uhn}

\author[Methuku]{Abhishek Methuku}

\author[Osthus]{Deryk Osthus}

\date{March 26, 2023}

\thanks{This project has received partial funding from the European Research
Council (ERC) under the European Union's Horizon 2020 research and innovation programme (grant agreement no.\ 786198, D.~Kang, T.~Kelly, D.~K\"uhn, A.~Methuku, and D.~Osthus).
The research leading to these results was also partially supported by the EPSRC, grant no.\ EP/S00100X/1 (A.~Methuku and D.~Osthus).}


\begin{abstract}
We prove that for $n \in \mathbb N$ and an absolute constant $C$, if $p \geq C\log^2 n / n$ and $L_{i,j} \subseteq [n]$ is a random subset of $[n]$ where each $k\in [n]$ is included in $L_{i,j}$ independently with probability $p$ for each $i, j\in [n]$, then asymptotically almost surely there is an order-$n$ Latin square in which the entry in the $i$th row and $j$th column lies in $L_{i,j}$.  The problem of determining the threshold probability for the existence of an order-$n$ Latin square was raised independently by Johansson, by Luria and Simkin, and by Casselgren and H{\"a}ggkvist; our result provides an upper bound which is tight up to a factor of $\log n$ and strengthens the bound recently obtained by Sah, Sawhney, and Simkin.  We also prove analogous results for Steiner triple systems and $1$-factorizations of complete graphs, and moreover, we show that each of these thresholds is at most the threshold for the existence of a $1$-factorization of a nearly complete regular bipartite graph.
\end{abstract}

\maketitle

\section{Introduction}

An \textit{order-$n$ Latin square} $\cL$ is an $n \times n$ matrix with entries from a set of $n$ symbols, such that each row and each column contains each symbol exactly once.
An \textit{order-$n$ Steiner triple system} $\cS$ is a pair $(V, \cT)$ where $V$ is an $n$-element ground set and $\cT \subseteq \binom{V}{3}$ is a collection of triples of $V$ such that every pair $\{u, v\} \in \binom{V}{2}$ is contained in exactly one element of $\cT$.  Latin squares and Steiner triple systems are central objects of study in combinatorics with a long and extensive line of research that can be traced back to the work of Euler in the 1700s and of Kirkman in the 1800s.
Both are fundamental in the study of combinatorial designs (an order-$n$ Steiner triple system is what is known as a \textit{$t$-$(n, k, \lambda)$ design} where $t = 2$, $k = 3$, and $\lambda = 1$\COMMENT{A \textit{$t$-$(n, k, \lambda)$ design} is a pair $(V, \cB)$ where $V$ is an $n$-element ground set and $\cB \subseteq \binom{V}{k}$ is a collection of $k$-element subsets of $V$ such that every $t$-element subset of $V$ is contained in exactly $\lambda$ elements of $\cB$}) and graph decompositions.  A \textit{triangle decomposition} $\cT$ of a graph $G$ is a collection of triangles in $G$ such that every edge $e \in E(G)$ is contained in exactly one triangle in $\cT$.  There is an obvious bijection between order-$n$ Latin squares (with a fixed symbol set) and triangle decompositions of the $n \times n \times n$ complete tripartite graph $K_{n,n,n}$ and between order-$n$ Steiner triple systems (with a fixed ground set) and triangle decompositions of the $n$-vertex complete graph $K_n$, so we refer to them interchangeably as such.  Another important class of `design-like' structures are \textit{$1$-factorizations}---that is, decompositions into perfect matchings---of complete graphs, which can be viewed as schedules for a round-robin tournament. 

In this paper, we bound the \textit{threshold probability} for the appearance of an order-$n$ Latin square or Steiner triple system in a binomial random subset of triangles in $K_{n,n,n}$ and $K_n$, respectively (see Theorems~\ref{thm:ls-threshold} and \ref{thm:sts-threshold}).  We also bound the threshold for the appearance of a $1$-factorization of $K_{2n}$ in the analogous model (see Theorem~\ref{thm:list-threshold}).

\subsection{Thresholds}

For a graph $G$ and $p \in [0, 1]$, we let $\randt(G, p)$ be a random set of triangles of $G$ where each is included independently with probability $p$.  Here we are primarily interested in the case when $G \cong K_n$ or $G \cong K_{n,n,n}$; one can view a random subset of triangles in $K_n$ as a random $3$-uniform hypergraph and a random subset of triangles in $K_{n,n,n}$ as a random $3$-uniform $3$-partite hypergraph.  For $1$-factorizations of $K_{2n}$, we also consider $\randt(G, p)$ where $G$ is the join of $K_{2n}$ and an empty graph with $2n-1$ vertices (as in Theorem~\ref{thm:red2}).
For an increasing property $\cP$ of $3$-uniform hypergraphs and a sequence of graphs $(G_n : n \in \mathbb N)$, a function $p^* = p^*(n)$ is a \textit{threshold} if
\begin{equation*}
  \Prob{\randt(G_n, p) \in \cP} \rightarrow \left\{
    \begin{array}{l l}
      0 & \text{if $p = o(p^*)$ and}\\
      1 & \text{if $p = \omega(p^*)$.}
    \end{array}
  \right.
\end{equation*}

The notion of a threshold extends naturally to other random structures, and thresholds are one of the most intensively studied topics in probabilistic combinatorics.  Among the most classical results in this area are Erd\H{o}s and R\'{e}nyi's~\cite{ER64, ER66} determination of the threshold\COMMENT{really it is only \textit{a} threshold, but it is commonly referred to as \textit{the} threshold in the literature} for the appearance of a perfect matching in a binomial random subgraph of $K_{n,n}$ and $K_{2n}$.  Moreover, determining the threshold more generally for the appearance of a perfect matching in a binomial random uniform hypergraph, known as \lq Shamir's problem\rq\ (see \cite{Erdos81}), was long viewed as one of the most important problems in the area until it was famously resolved by Johansson, Kahn, and Vu~\cite{JKV08}.  Shamir's problem itself is a special case of the problem of determining the threshold for the existence of designs, since a perfect matching in an $n$-vertex $k$-uniform hypergraph is a so-called $1$-$(n, k, 1)$ design.
This problem is of course extremely challenging; the existence of designs with arbitrary parameters was proved only recently, first by Keevash~\cite{Ke14} via randomized algebraic constructions and shortly afterwards, by Glock, K\"{u}hn, Lo, and Osthus~\cite{GKLO16} via iterative absorption.
Nevertheless, this problem was recently discussed by Frankston, Kahn, Narayanan, and Park~\cite{FKNP21}, and Simkin~\cite{Si17} boldly conjectured a threshold for $t$-$(n, t + 1, 1)$ designs.
Our bound on the threshold for Steiner triple systems is within a $\log n$ factor of the first unknown case $t = 2$ of Simkin's conjecture.  
The related problem of determining the threshold for Latin squares was first raised in 2006 by Johansson~\cite{J06}, when he gave an approximate solution to Shamir's problem.  Corresponding conjectures were later formulated by Luria and Simkin~\cite{LS19} and Casselgren and H{\"a}ggkvist~\cite{CH16}, and our results resolve these within a $\log n$ factor.

Our proof makes use of the recent breakthrough work of Frankston, Kahn, Narayanan, and Park~\cite{FKNP21}.  The Frankston--Kahn--Narayanan--Park theorem is a remarkably general result that provides an upper bound on the threshold for any property in terms of its \textit{fractional expectation-threshold}, or equivalently via linear programming duality in terms of the minimum `spreadness' of a probability distribution it supports.
(This result was subsequently generalized by Park and Pham~\cite{PP22} to a full proof of the Kahn--Kalai conjecture, but we do not make use of this here.)
The solution to Shamir's problem follows immediately as a corollary of the Frankston--Kahn--Narayanan--Park theorem, as do many other important threshold results analogously, by considering the spreadness of the uniform distribution on perfect matchings.  However, determining the spreadness of the uniform distribution for Latin squares or Steiner triple systems currently seems out of reach, so in our proof we instead construct a `well-spread' distribution on Latin squares, which still yields a bound on the threshold for Latin squares.  (Our main result is Theorem~\ref{thm:robust-ls-threshold}, a slightly stronger result which we can also use to deduce the approximate threshold for Steiner triple systems and for $1$-factorizations of complete graphs.)  In all applications of the Frankston--Kahn--Narayanan--Park theorem given in \cite{FKNP21}, bounding the spreadness is straightforward, whereas in our paper it is the main task of the proof.  This is also the case in the very recent work of Sah, Sawhney, and Simkin~\cite{SSS22} which also bounds the threshold for the appearance of spanning Latin squares and Steiner triple systems.  However, apart from this shared aspect of our proofs, our approach is quite different, and our results strengthen all of their bounds. 

\subsubsection{Latin squares}

Luria and Simkin~\cite{LS19} conjectured that $p = p(n) = \log n / n$ is the threshold for the property that $\randt(K_{n,n,n}, p)$ contains an order-$n$ Latin square, that is, for the property that $\randt(K_{n,n,n}, p)$ contains a triangle decomposition of $K_{n,n,n}$.
As evidence for their conjecture, Luria and Simkin~\cite{LS19} proved that for every $\eps > 0$, the threshold for $\randt(K_{m, n, n}, p)$ to contain an \textit{$m \times n \times n$ Latin box},\footnote{which is a triangle packing of $K_{m, n, n}$ with the maximum number of edges} where $m = \lceil(1 - \eps)n\rceil$, is $\log n / n$.
As is commonly the case in this area, the lower bound for this threshold is straightforward to prove; indeed, if $p = o(\log n / n)$, then asymptotically almost surely (which we henceforth abbreviate \textit{a.a.s.}) there is an edge of $K_{n,n,n}$ not contained in a triangle of $\randt(K_{n,n,n}, p)$ and consequently there is no triangle decomposition of $K_{n,n,n}$ in $\randt(K_{n,n,n}, p)$.  Thus, the challenge for the Luria--Simkin conjecture is to prove that if $p = \omega(\log n / n)$, then \aas $\randt(K_{n,n,n}, p)$ contains an order-$n$ Latin square.  Recently, Sah, Sawhney, and Simkin~\cite{SSS22} made a significant step in this direction by proving an upper bound of $p = \exp(C\log^{3/4} n) / n$.  Our first main result improves this bound and provides an upper bound on this threshold that is within a $\log n$ factor of best possible, as follows.

\begin{theorem}\label{thm:ls-threshold}
If $p \geq C_{\ref*{thm:ls-threshold}} \log^2 n / n$, where $C_{\ref*{thm:ls-threshold}} > 0$ is an absolute constant, then \aas $\randt(K_{n,n,n}, p)$ contains an order-$n$ Latin square.
\end{theorem}

An equivalent formulation of this result in terms of the completion of Latin squares via random lists is the following: If $p \geq C_{\ref*{thm:ls-threshold}}\log^2 n / n$ and $L_{i,j} \subseteq [n]$ is a random subset of $[n]$ chosen independently for each $i,j \in [n]$ such that each $k \in [n]$ is included in $L_{i,j}$ independently with probability $p$, then \aas there exists an order-$n$ Latin square $\cL$ such that $\cL_{i,j} \in L_{i,j}$ for each $i,j\in[n]$.  

Theorem~\ref{thm:ls-threshold} can also be understood as a result guaranteeing a proper $n$-edge-colouring of the complete bipartite graph $K_{n,n}$ from a random list assignment, as follows.  A \textit{list assignment} for the edges of a graph $G$ is a map $L$ whose domain is the edge set $E(G)$ of $G$ where $L(e)$ is a finite set for every $e \in E(G)$, and an \textit{$L$-edge-colouring} of $G$ is a proper edge-colouring of $G$ in which each edge $e$ is assigned a colour from $L(e)$.  For a graph $G$, $n \in \mathbb N$, and $p \in [0, 1]$, a \textit{random $(p, n)$-list assignment} $L$ for the edges of $G$ is a list assignment for the edges of $G$ such that $L(e)$ is a random subset of $[n]$ chosen independently for each $e \in E(G)$ such that each $k \in [n]$ is included in $L(e)$ independently with probability $p$.  Theorem~\ref{thm:ls-threshold} implies that if $p \geq C_{\ref*{thm:ls-threshold}}\log^2 n / n$ and $L$ is a random $(p, n)$-list assignment for the edges of $K_{n,n}$, then \aas there is an $L$-edge-colouring of $K_{n,n}$.
Related `palette sparsification' results have recently been employed in algorithmic applications including the design of sublinear-time algorithms for graph colouring \cite{AA20, ACK19}, but they are also of independent interest\COMMENT{we do not obtain any such applications since our proof does not yield an algorithm for finding the edge-colouring, but this is an intriguing open problem}.

Theorem~\ref{thm:ls-threshold} yields the following corollary which nearly resolves a conjecture of Casselgren and H{\"a}ggkvist~\cite{CH16}.  For a graph $G$ and $k, n \in \mathbb N$, a \textit{random $(k, n)$-list assignment} $L$ for the edges of $G$ is a list assignment for the edges of $G$ where $L(e) \subseteq \binom{[n]}{k}$ is chosen independently and uniformly at random for each $e \in E(G)$.  Casselgren and H{\"a}ggkvist~\cite{CH16} conjectured that if $k \geq C\log n$ for sufficiently large $C$ and $L$ is a random $(k, n)$-list assignment for the edges of $K_{n,n}$, then \aas there is an $L$-edge-colouring of $K_{n,n}$, and we show that this is true for $k \geq 2C_{\ref*{thm:ls-threshold}}\log^2 n$.

\begin{corollary}\label{thm:bipartite-list-threshold}
  If $k \geq 2C_{\ref*{thm:ls-threshold}}\log^2 n$ and $L$ is a random $(k, n)$-list assignment for the edges of $K_{n,n}$, then \aas there is an $L$-edge-colouring of $K_{n,n}$.
\end{corollary}

The proof of Corollary~\ref{thm:bipartite-list-threshold} assuming Theorem~\ref{thm:ls-threshold} involves standard methods---similar to those relating thresholds for the random graph $G_{n,p}$ in the binomial model to those for $G_{n,m}$ in the uniform model (see e.g.~\cite{Bo01, JLR00})---so we omit it.
\subsubsection{$1$-factorizations}\label{sect:1-factorizations}

A \textit{$1$-factorization} of a $d$-regular graph $G$ is an ordered partition $(M_1, \dots, M_d)$ of the edges of $G$ into perfect matchings, which can also be considered a proper $d$-edge-colouring of $G$.  Order-$n$ Latin squares are also in bijection with $1$-factorizations of $K_{n,n}$, and we use this fact in the proof of Theorem~\ref{thm:ls-threshold}.    We also provide upper bounds on the threshold for a random list assignment for the edges of a complete graph to admit a $1$-factorization, as follows.

\begin{theorem}\label{thm:list-threshold}
  If $k  \geq C_{\ref*{thm:list-threshold}}\log^2 n$ and $L$ is a random $(k, 2n-1)$-list assignment for the edges of $K_{2n}$, or if $p \geq C_{\ref*{thm:list-threshold}}\log^2 n / n$ and $L$ is a random $(p, 2n-1)$-list assignment for the edges of $K_{2n}$, where $C_{\ref*{thm:list-threshold}} > 0$ is an absolute constant, then \aas there is an $L$-edge-colouring of $K_{2n}$.
\end{theorem}

This result also clearly applies with $K_{2n}$ replaced with $K_{2n-1}$.  The `$2n - 1$' in Theorem~\ref{thm:list-threshold} is best possible in both cases, since a proper edge-colouring of $K_{2n}$ or $K_{2n-1}$ uses at least $2n - 1$ colours.  However, we conjecture that Theorem~\ref{thm:list-threshold} holds with $k = \omega(\log n)$ and $p =\omega(\log n / n)$, which if true, would be best possible.
In fact, we conjecture more strongly a \textit{sharp threshold} of $p = \log n / n$ and a \textit{hitting time} result -- see Conjecture~\ref{conj:1-factorization-threshold}.

\subsubsection{Steiner triple systems}

Our next result provides an upper bound on the threshold for $\randt(K_n, p)$ to contain an order-$n$ Steiner triple system.  For similar reasons as before, $p = \Omega(\log n / n)$ is a straightforward lower bound, and it is natural to conjecture that $\log n / n$ is in fact the threshold.\COMMENT{It is also necessary that $n \cong 1, 3 \mod 6$, since otherwise an order-$n$ Steiner triple system does not exist for divisibility reasons.}  Indeed, Simkin~\cite{Si17} even conjectured more generally that this is the threshold for the appearance of a $t$-$(n, t + 1, 1)$ design (also known as an $(n, t + 1, t)$-Steiner system) in a binomial random subset of $\binom{[n]}{t + 1}$.  With a similar proof as for the Latin square case\COMMENT{actually, they only wrote the proof for Steiner triple systems and said the Latin square proof is the same}, Sah, Sawhney, and Simkin~\cite{SSS22} proved an upper bound of $p = \exp(C\log^{3/4} n) / n$ for the threshold for Steiner triple systems.  
Similarly as with Theorem~\ref{thm:ls-threshold}, we improve their bound and match the conjectured threshold up to a $\log n$ factor, as follows.

\begin{theorem}\label{thm:sts-threshold}
  If $n \equiv 1,3 \mod 6$ and $p \geq C_{\ref*{thm:sts-threshold}}\log^2 n / n$, where $C_{\ref*{thm:sts-threshold}} > 0$ is an absolute constant, then \aas $\randt(K_{n}, p)$ contains an order-$n$ Steiner triple system.
\end{theorem}

\subsubsection{Main result}

As mentioned, our main result is a strengthening of Theorem~\ref{thm:ls-threshold}, which we can use to deduce Theorems~\ref{thm:list-threshold} and \ref{thm:sts-threshold}.  This result \aas guarantees an $L$-edge-colouring in a nearly complete $d$-regular bipartite graph $H \subseteq K_{n,n}$ with $d = n - o(n)$, where $L$ is a random $(p, d)$-list assignment for $H$, as follows.

\begin{theorem}\label{thm:robust-ls-threshold}
  There exist $C_{\ref*{thm:robust-ls-threshold}}, \alpha_{\ref*{thm:robust-ls-threshold}} > 0$ such that the following holds.  Let $V_1$, $V_2$, and $V_3$ be disjoint sets such that $|V_1| = |V_2| = n$ and $|V_3| = d$.  Let $H$ be a $d$-regular bipartite graph with bipartition $\{V_1, V_2\}$, and let $G$ be the graph obtained from $H$ by adding all edges $xy$ where $x \in V_1 \cup V_2$ and $y \in V_3$.  If $d \geq (1 - \alpha_{\ref*{thm:robust-ls-threshold}})n$ and $p \geq C_{\ref*{thm:robust-ls-threshold}} \log^2 n / n$, then \aas $\randt(G, p)$ contains a triangle decomposition of $G$ and equivalently, \aas there exists an $L$-edge-colouring of $H$ where $L$ is a random $(p, d)$-list assignment for the edges of $H$.
\end{theorem}

Theorem~\ref{thm:ls-threshold} follows immediately from Theorem~\ref{thm:robust-ls-threshold} with $n$ playing the role of $d$.  Next we show how we use Theorem~\ref{thm:robust-ls-threshold} to deduce Theorems~\ref{thm:list-threshold} and \ref{thm:sts-threshold}.

\subsection{Reductions}

 Together with Theorem~\ref{thm:robust-ls-threshold}, the following result implies Theorem~\ref{thm:sts-threshold}.  
 
\begin{theorem}[Steiner triple system reduction]\label{thm:sts-reduction}
  There exists $\varepsilon_{\ref*{thm:sts-reduction}} > 0$ such that for every $\eps \in (0,\varepsilon_{\ref*{thm:sts-reduction}})$, the following holds for some $C_{\ref*{thm:sts-reduction}}(\eps) > 0$. 
  Let $n \equiv 1, 3 \mod 6$, let $V_1$, $V_2$, and $V_3$ be disjoint sets such that $|V_1| = |V_2| = \lfloor n / 3\rfloor$ and $|V_3| = \lceil n / 3\rceil$, and let $W \subseteq V_3$ such that $|W| = \lfloor \eps n \rfloor$.
  Let $G$ be obtained from the complete graph on $V_1 \cup V_2 \cup V_3$ by removing every edge $xy$ where $x \in V_1 \cup V_2$ and $y \in V_3 \setminus W$.
  If $p \geq C_{\ref*{thm:sts-reduction}}(\eps) \log n / n$, then \aas there exists a collection $\cT \subseteq \randt(G, p)$ of edge-disjoint triangles and a $|V_3\setminus W|$-regular bipartite graph $H$ with bipartition $\{V_1, V_2\}$  such that $\cT$ is a triangle decomposition of $G - E(H)$.
\end{theorem}

To deduce Theorem~\ref{thm:sts-threshold} from Theorems~\ref{thm:robust-ls-threshold} and~\ref{thm:sts-reduction}, let $\eps \coloneqq \min (\varepsilon_{\ref*{thm:sts-reduction}} / 2 , \alpha_{\ref*{thm:robust-ls-threshold}} / 4)$, let $C_{\ref*{thm:sts-threshold}} \coloneqq C_{\ref*{thm:robust-ls-threshold}} + C_{\ref*{thm:sts-reduction}}(\eps)$, let $p_{\ref*{thm:sts-reduction}} \coloneqq C_{\ref*{thm:sts-reduction}} (\eps) \log^2 n / n$, and let $p_{\ref*{thm:robust-ls-threshold}} \coloneqq C_{\ref*{thm:robust-ls-threshold}} \log^2 n / n$.
For $p \geq C_{\ref*{thm:sts-threshold}} \log^2 n / n$, 
since $1 - p \leq (1 - p_{\ref*{thm:sts-reduction}})(1 - p_{\ref*{thm:robust-ls-threshold}})$, there is a coupling $\cG_{\ref*{thm:sts-reduction}} \cup \cG_{\ref*{thm:robust-ls-threshold}} \subseteq \cG \sim \cG^{(3)}(K_n , p)$,
where $\cG_{\ref*{thm:sts-reduction}} \sim \cG^{(3)}(K_n , p_{\ref*{thm:sts-reduction}})$ and $\cG_{\ref*{thm:robust-ls-threshold}} \sim \cG^{(3)}(K_n , p_{\ref*{thm:robust-ls-threshold}})$ are independent.
First, apply Theorem~\ref{thm:sts-reduction} to obtain $\cT \subseteq \cG_{\ref*{thm:sts-reduction}}$ and $H$ as described there, and let $d \coloneqq \lceil n / 3\rceil - \lfloor \eps n \rfloor \geq (1-  \alpha_{\ref*{thm:robust-ls-threshold}}) \lfloor n/3 \rfloor$.
Then apply Theorem~\ref{thm:robust-ls-threshold} with this $d$ and $H$ and with $\lfloor n / 3\rfloor$ and $V_3 \setminus W$ playing the roles of $n$ and $V_3$, respectively,\COMMENT{(note that $d \geq (1 - \alpha_{\ref*{thm:robust-ls-threshold}})\lfloor n / 3 \rfloor $)}
to obtain a triangle decomposition $\cT' \subseteq \cG_{\ref*{thm:robust-ls-threshold}}$ of the graph $G$ as defined in Theorem~\ref{thm:robust-ls-threshold}.
Now $\cT \cup \cT' \subseteq \cG_{\ref*{thm:sts-reduction}} \cup \cG_{\ref*{thm:robust-ls-threshold}} \subseteq \cG$ is a triangle decomposition of $K_n$, as required.

Similarly, together with Theorem~\ref{thm:robust-ls-threshold}, the following result immediately implies Theorem~\ref{thm:list-threshold}.

\begin{theorem}[$1$-factorization reduction]\label{thm:list-reduction}
  There exists $\eps_{\ref*{thm:list-reduction}} > 0$ such that for every $\eps \in (0, \eps_{\ref*{thm:list-reduction}})$, the following holds for some $C_{\ref*{thm:list-reduction}}(\eps) > 0$.
  Let $d \coloneqq \lceil (1 - \eps) n\rceil$.  If $p \geq C_{\ref*{thm:list-reduction}}(\eps) \log n / n$ and $L$ is a random $(p, 2n - d - 1)$-list assignment for the edges of $K_{2n}$, then \aas there exists a $d$-regular graph $H \subseteq K_{n,n}$ and an $L$-edge-colouring of $K_{2n} - E(H)$.
\end{theorem}


Moreover, since these results apply with $p = \omega(\log n / n)$, we have reduced all of the Luria--Simkin conjecture~\cite{LS19}, the Casselgren--H{\"a}ggkvist~\cite{CH16} conjecture, the Steiner triple system case of Simkin's conjecture~\cite{Si17}, and the threshold problem for $1$-factorizations of the complete graph to the problem of $L$-edge-colouring a nearly complete $d$-regular bipartite graph $H$, where $L$ is a random $(\omega(\log n / n), d)$-list assignment for $H$.

\subsection{Spreadness}\label{sect:spreadness}

Our proof of Theorem~\ref{thm:robust-ls-threshold} uses the Frankston--Kahn--Narayanan--Park theorem~\cite{FKNP21}, which reduces the problem to that of finding a probability distribution on $1$-factorizations (of the $d$-regular bipartite graph $H$) which is `well-spread'.  This notion of `spreadness' can be made much more general, but for simplicity here we only define it in the context of $1$-factorizations, as follows.

\begin{definition}[Spreadness]
  Let $G$ be a $d$-regular bipartite graph.  A probability distribution on $1$-factorizations $(M_1, \dots, M_d)$ of $G$ is \textit{$q$-spread} if 
  \begin{equation}\label{eq:spreadness}\tag{$\star$}
    \text{every $S_1, \dots, S_d \subseteq E(G)$ satisfies $\Prob{S_i \subseteq M_i~\forall i \in [d]} \leq q^{\sum_{i=1}^d |S_i|}$.}
  \end{equation}
\end{definition}

The following result is a corollary of \cite[Theorem 1.6]{FKNP21}.


\begin{theorem}[Frankston, Kahn, Narayanan, and Park~\cite{FKNP21}]\label{thm:FKNP}
   There exists $C_{\ref*{thm:FKNP}} > 0$ such that the following holds. If there exists a $q$-spread probability distribution on $1$-factorizations of a $d$-regular bipartite graph $G$ on $2n$ vertices, and if $p \geq C_{\ref*{thm:FKNP}}q \log n$, then \aas there exists an $L$-edge-colouring of $G$ where $L$ is a random $(p, d)$-list assignment for the edges of $G$.
\end{theorem}

Together with Theorem~\ref{thm:FKNP}, Theorem~\ref{thm:robust-ls-threshold} follows from the following result.
 
\begin{theorem}\label{thm:LS-spreadness}
  There exist $C_{\ref*{thm:LS-spreadness}} ,\alpha_{\ref*{thm:LS-spreadness}} > 0$ such that the following holds for every $n \in \mathbb N$.  If $G$ is a $d$-regular bipartite graph on $2n$ vertices where $d \geq (1 - \alpha_{\ref*{thm:LS-spreadness}})n$, then there exists a $(C_{\ref*{thm:LS-spreadness}} \log n / n)$-spread probability distribution on $1$-factorizations of $G$.
\end{theorem}

Theorem~\ref{thm:LS-spreadness} may also hold without the $\log n$, which would immediately imply the Luria--Simkin conjecture~\cite{LS19} and the Casselgren--H{\"a}ggkvist~\cite{CH16} conjecture, and combined with Theorems~\ref{thm:sts-reduction} and \ref{thm:list-reduction}, would imply the Steiner triple system case of Simkin's conjecture~\cite{Si17} and our conjecture on the threshold for $1$-factorizations of the complete graph.
A necessary condition for the existence of such spreadness is that the number of $1$-factorizations of $G$ is at least $\Omega(n)^{nd}$ because of the condition \eqref{eq:spreadness} in the case when $(S_1, \dots, S_d)$ is itself a $1$-factorization.  This necessary condition is satisfied since the Egorychev--Falikman theorem~\cite{Eg81, Fa81} and Br\`{e}gman's theorem~\cite{Br73} (previously known as van der Waerden's conjecture and Minc's conjecture, respectively) give bounds on the number of perfect matchings in a regular bipartite graph which imply that a $d$-regular bipartite graph with parts of size $n$ has $\left((1 + o(1))n / e^2\right)^{dn}$ $1$-factorizations.

An obvious approach to proving Theorem~\ref{thm:LS-spreadness} would be to attempt to show that the uniform distribution on $1$-factorizations of $G$ is $O(\log n / n)$-spread.  (In the case $G \cong K_{n,n}$, one could equivalently attempt to show that the uniform distribution on order-$n$ Latin squares is $O(\log n / n)$-spread).  However, the bounds from the Egorychev--Falikman theorem~\cite{Eg81, Fa81} and Br\`{e}gman's theorem~\cite{Br73} do not seem strong enough to prove the required spreadness.  For the related problem of bounding the threshold for Steiner triple systems, as discussed in \cite{FKNP21}, it is also unclear how to carry out such a direct approach (that is, by bounding the spreadness of the uniform distribution on Steiner triple systems).


An advantage of considering Latin squares over Steiner triple systems (and one which we exploit), is that they can be constructed recursively.  A well-known corollary of Hall's theorem is that regular bipartite graphs have a perfect matching, which further implies that every regular bipartite graph has a $1$-factorization.  The main goal in our proof is to first find a `well-spread' probability distribution on decompositions of $G$ into spanning regular subgraphs with $O(n\log n)$ edges.  Since each one of these subgraphs is bipartite and regular, they all have a $1$-factorization, and it turns out that choosing a $1$-factorization of each arbitrarily immediately yields a distribution on $1$-factorizations of $G$ satisfying \eqref{eq:spreadness} with $q = O(\log n / n)$.  At first glance, finding this distribution on decompositions into regular subgraphs may not appear to be any easier, but we are able to construct it using ideas inspired by the method of \textit{iterative absorption}.  We remark that Sah, Sawhney, and Simkin~\cite{SSS22} also relied on an iterative absorption approach to construct a probability distribution on Steiner triple systems; however, their application of this method is completely different from ours.  They use what one could consider a `vertex vortex' as developed in \cite{BGKLMO20, GKLO16} to prove the existence of subgraph decompositions and designs.  This approach gradually decomposes subgraphs induced on a nested sequence of vertex sets and is a natural choice for finding a triangle decomposition.  Our approach uses what we call an `edge vortex', which is more natural for finding a decomposition into spanning structures and is more similar to that of e.g.~\cite{KKO15, KO13} which first introduced the iterative absorption method.  Moreover, referring essentially to their threshold bound of $n^{-1+o(1)}$, Sah, Sawhney, and Simkin~\cite{SSS22} remark, `This is essentially the limit for a pure iterative absorption framework as in \cite{BGKLMO20, GKLO16}'.  There does not appear to be any such barrier to our approach, and it may be possible to build on our ideas to prove that $\log n / n$ is indeed the threshold for Latin squares, $1$-factorizations of complete graphs, and Steiner triple systems.

\subsection{Open problems}

Now we discuss a few open problems and conjectures.  
First, we remark that in the time since the first draft of this paper was posted on the arxiv preprint server, the thresholds $\Theta(\log n / n)$ for order-$n$ Latin squares, Steiner triple systems, and $1$-factorizations of $K_{2n}$ were proved in subsequent work by Keevash~\cite{K2022} and Jain and Pham~\cite{JP2022}, resolving 
the Luria--Simkin conjecture~\cite{LS19}, the Casselgren--H{\"a}ggkvist~\cite{CH16} conjecture, and the Steiner triple system case of Simkin's conjecture~\cite{Si17}.  Both works~\cite{JP2022, K2022} utilize our reduction results and build on our approach to show that Theorem~\ref{thm:LS-spreadness} holds without the `$\log n$' in the spreadness. 
 Whereas our proof finds a `well-spread' probability distribution on decompositions of $G$ into spanning regular subgraphs with $O(n\log n)$ edges, Keevash's~\cite{K2022} and Jain and Pham's~\cite{JP2022} proofs decompose $G$ into regular subgraphs with $O(n)$ edges.  Nevertheless, the two proofs are still quite different.  
 Keevash~\cite{K2022} developed a carefully constructed random greedy process for this decomposition, and Jain and Pham~\cite{JP2022} constructed the decomposition recursively while using a distribution based on the Lov\'{a}sz Local Lemma.

As mentioned in Section~\ref{sect:1-factorizations}, even more may still be true.  We say a threshold $p^*$ for a property $\cP$ of $3$-uniform hypergraphs and a sequence $(G_n : n \in \mathbb N)$ of graphs is \textit{sharp} if
\begin{equation*}
  \Prob{\randt(G_n, p) \in \cP} \rightarrow \left\{
    \begin{array}{l l}
      0 & \text{if $p \leq (1 - \Omega(1))p^*$ and}\\
      1 & \text{if $p \geq (1 + \Omega(1))p^*$.}
    \end{array}
  \right.
\end{equation*}
Intriguingly, as discussed in \cite{SSS22}, the results of \cite{SSS22} combined with results of Friedgut~\cite{Fr99, Fr05} imply the existence of a sharp threshold for $\randt(K_n, p)$ to contain an order-$n$ Steiner triple system, but without determining its value.  Similar arguments imply the existence of a sharp threshold for $\randt(K_{n,n,n}, p)$ to contain an order-$n$ Latin square (see \cite{SSS22}).
Sah, Sawhney, and Simkin~\cite{SSS22} also conjectured that $2\log n / n$ is the value of the sharp threshold for Steiner triple systems, since $2\log n / n$ is the sharp threshold for the property that every edge of $K_n$ is in at least one triangle of $\randt(K_n, p)$.  They even conjectured a hitting time result.  It is natural to conjecture the following analogues for Latin squares and $1$-factorizations of complete graphs as well.\COMMENT{The expected number of edges of $K_{n,n, n}$ not covered by a triangle of $\randt(K_{n,n,n}, p)$ is $3n^2 (1 - p)^n \sim 3n^2 e^{-pn}$, so the sharp threshold for Latin squares should correspond to the value of $p$ for which $3n^2 e^{-pn} = 1$.  Rearranging terms, we obtain $p = (2\log n - \log(1/3)) / n \sim 2\log n / n$.  Similarly, the sharp threshold for $1$-factorizations of $K_{2n}$ should be when $\left(\binom{2n}{2} + (2n - 1)2n\right)(1 - p)^{2n-1} \sim 6n^2 e^{-2pn} = 1$, i.e.\ when $p = (2\log n - \log(1/6)) / (2n) \sim \log n / n$.}

\begin{conjecture}\label{conj:latin-square-threshold}
  For every $\eps > 0$, the following holds.  If $p \geq (2 + \eps)\log n / n$, then \aas $\randt(K_{n,n,n}, p)$ contains an order-$n$ Latin square.
\end{conjecture}

\begin{conjecture}\label{conj:1-factorization-threshold}
  For every $\eps > 0$, the following holds.  If $p \geq (1 + \eps)\log n / n$ and $L$ is a random $(p, 2n - 1)$-list assignment for the edges of $K_{2n}$, then \aas there exists an $L$-edge-colouring of $K_{2n}$.
\end{conjecture}

We also conjecture that a hitting time version of Conjecture~\ref{conj:latin-square-threshold} holds; that is, if $T_1, \dots, T_{n^3}$ is a uniformly random ordering of the triangles of $K_{n,n,n}$, then \aas $\{T_1, \dots, T_t\}$ contains an order-$n$ Latin square for every $t \in [n^3]$ such that every edge of $K_{n,n,n}$ is contained in at least one triangle in $T_1, \dots, T_t$.
Similarly, we conjecture the following hitting time version of Conjecture~\ref{conj:1-factorization-threshold}: If $L_0, \dots, L_{(2n-1)\binom{2n}{2}}$ are list assignments for the edges of $K_{2n}$, where $L_0(e) = \varnothing$ for every edge $e$ and for $i \in [(2n-1)\binom{2n}{2}]$, $L_i$ is obtained by adding a colour $c \in [2n - 1]\setminus L_{i-1}(e)$ to $L_{i-1}(e)$ for some edge $e$ where $(c, e)$ is chosen uniformly at random among all such pairs, then \aas there exists an $L_t$-edge-colouring of $K_{2n}$ for every $t \in [(2n - 1)\binom{2n}{2}]$ such that $L_t(e) \neq \varnothing$ for every edge $e$ and $\bigcup_{e \ni v}L_t(e) = [2n - 1]$ for every vertex $v$.

Conjecture~\ref{conj:latin-square-threshold} is equivalent to the following: If $p \geq (2 + \eps)\log n / n$ and $L$ is a random $(p, n)$-list assignment for the edges of $K_{n,n}$, then \aas there is an $L$-edge-colouring of $K_{n,n}$.  It would also be interesting to strengthen this further to a `sparse random version' of Galvin's theorem~\cite{Ga95}, by showing the following: If $L'$ is a list assignment for the edges of $K_{n,n}$ satisfying $|L'(e)| \geq n$ for every edge $e$ and $L$ is a random list assignment for the edges of $K_{n,n}$ in which each $c \in L'(e)$ is included in $L(e)$ independently with probability $(2 + \eps)\log n / n$, then \aas there is an $L$-edge-colouring of $K_{n,n}$.  The analogous strengthening of Conjecture~\ref{conj:1-factorization-threshold} would also be interesting, but it is not even known whether the list chromatic index of $K_{2n}$ is $2n - 1$ for all $n \in \mathbb N$ (see \cite{HJ97, Sc14}).

The threshold for Latin squares could also be generalized by considering so-called `high-dimensional permutations' (as in \cite{LL14}) or `Latin hypercubes' (as in \cite{MW08}).  To that end, we introduce the following definitions.  A \textit{$d$-dimensional order-$n$ hypercube} is a $d$-dimensional array indexed by $[n]^d$ with entries from a set of $n$ symbols, and a \textit{line} in a hypercube is a set of $n$ entries whose coordinates match in all but one dimension.  A hypercube $\cL$ is \textit{Latin} if each line in $\cL$ contains each symbol exactly once.

\begin{conjecture}\label{conj:latin-hypercube-threshold}
  If $p = \omega(\log n / n)$ and $L_{\myVec{x}} \subseteq [n]$ is a random subset of $[n]$ in which each $k \in [n]$ is included in $L_{\myVec{x}}$ independently with probability $p$ for each $\myVec{x} \in [n]^d$, then \aas there exists a $d$-dimensional order-$n$ Latin hypercube $\cL$ such that $\cL_{\myVec{x}} \in L_{\myVec{x}}$ for each $\myVec{x} \in [n]^d$.
\end{conjecture}

The $d = 2$ case of Conjecture~\ref{conj:latin-hypercube-threshold} corresponds to the Luria--Simkin conjecture~\cite{LS19}.  Note also that the $d = 1$ case corresponds to Erd\H{o}s and R\'{e}nyi's~\cite{ER64} classic result on the threshold for a binomial random subgraph of $K_{n,n}$ to have a perfect matching.  Because of the connection between $d$-dimensional order-$n$ Latin hypercubes and $(n, d + 1, d)$-Steiner systems, Conjecture~\ref{conj:latin-hypercube-threshold} may be viewed as a `partite version' of Simkin's conjecture~\cite{Si17}.

Finally, we conjecture the threshold for the existence of designs; recall that a \textit{$t$-$(n, k, \lambda)$ design} is a pair $(V, \cB)$ where $V$ is an $n$-element ground set and $\cB \subseteq \binom{V}{k}$ is a collection of $k$-element subsets of $V$ such that every $t$-element subset of $V$ is contained in exactly $\lambda$ elements of $\cB$.  The threshold for the existence of a $t$-$(n, k, \lambda)$ design in a random subset of $\binom{V}{k}$ is obviously at least the threshold for each $t$-set in $\binom{V}{t}$ to be covered at least $\lambda$ times by $k$-sets of the random subset, and we conjecture that these thresholds are in fact the same, as follows.

\begin{conjecture}\label{conj:designs-threshold}
  For every $t, k, \lambda \in \mathbb N$ with $k > t$, the following holds.
  Let $V$ be a finite set with $|V| = n$ satisfying $\binom{k - i}{t - i} \mid \lambda\binom{n - i}{t - i}$ for every $i \in \{0\}\cup [t - 1]$.
  If $p = \omega(\log n / n^{k - t})$ and each $k$-set in $\binom{V}{k}$ is included in $\cA$ independently with probability $p$, then \aas there exists a $t$-$(n, k, \lambda)$ design $(V, \cB)$ where $\cB \subseteq \cA$.
\end{conjecture}

Conjecture~\ref{conj:designs-threshold} provides a common generalization of Shamir's problem and Simkin's conjecture~\cite{Si17}.  Indeed, the case $t = \lambda = 1$ corresponds to Shamir's problem, and the case $t = k - 1$ and $\lambda = 1$ corresponds to Simkin's conjecture.  Here again, and for Conjecture~\ref{conj:latin-hypercube-threshold} as well, we conjecture there is a sharp threshold at the obvious lower bound and that the analogous hitting time results hold as well.

All of the above conjectures concern finding decompositions of some particular graph or hypergraph via a binomial random selection of `available' subgraphs.
It is also interesting to study decompositions of random graphs or hypergraphs; however, in this case, one should ensure the random structure satisfies the trivial divisibility requirements to admit the desired decomposition.  For example, a beautiful question of Yuster~\cite{Yu07}, which is still wide open, asks for the threshold for a triangle decomposition of a random even-regular graph.  Many further open problems e.g.\ on extremal aspects of triangle decompositions and Latin square completion can be found in \cite{GKO21}.

\subsection{Outline of the paper}

We have shown that Theorem~\ref{thm:ls-threshold} follows from Theorem~\ref{thm:robust-ls-threshold}, which in turn follows from Theorem~\ref{thm:LS-spreadness}.  Moreover, Theorem~\ref{thm:list-threshold} follows from Theorems \ref{thm:robust-ls-threshold} and \ref{thm:list-reduction}, and Theorem~\ref{thm:sts-threshold} follows from Theorems \ref{thm:robust-ls-threshold} and \ref{thm:sts-reduction}.  Therefore, the rest of the paper is devoted to the proofs of Theorems \ref{thm:sts-reduction}, \ref{thm:list-reduction}, and \ref{thm:LS-spreadness}.

In Section~\ref{sect:prelim}, we introduce some notation and basic lemmas used in the subsequent sections.
We prove Theorem~\ref{thm:LS-spreadness} in Section~\ref{sect:spreadness-proof}, Theorem~\ref{thm:sts-reduction} in Section~\ref{sect:sts-reduction}, and Theorem~\ref{thm:list-reduction} in Section~\ref{sect:list-reduction}.  In Section~\ref{sect:overview}, we provide an overview of the proofs of Theorem~\ref{thm:LS-spreadness} and \ref{thm:sts-reduction}; since the proofs of Theorems \ref{thm:sts-reduction} and \ref{thm:list-reduction} are similar, we only overview Theorem~\ref{thm:sts-reduction} in Section~\ref{sect:overview}, but we also briefly overview the proof of Theorem~\ref{thm:list-reduction} in Section~\ref{sect:list-reduction}. 

\section{Overview of the proofs}\label{sect:overview}

In this section, we overview the proofs of Theorem~\ref{thm:LS-spreadness} and \ref{thm:sts-reduction}.  

\subsection{Overview of Theorem~\ref{thm:LS-spreadness}: Finding a spread distribution on $1$-factorizations} 

We begin with an overview of the proof of Theorem~\ref{thm:LS-spreadness},
in which we construct a $(C_{\ref*{thm:LS-spreadness}}\log n / n)$-spread probability distribution on $1$-factorizations of some nearly complete, regular bipartite graph $G$.  For the purposes of this overview, we assume $G \cong K_{n,n}$, since the proof of this case can be easily generalized.

As mentioned in Section~\ref{sect:spreadness}, our goal in the proof is to find a `well-spread' probability distribution on decompositions of $G$ into spanning regular subgraphs with $O(n\log n)$ edges.\COMMENT{To be precise, a set $\{H_i \subseteq G : i \in [m]\}$ is a \textit{decomposition} of $G$ if $H_1, \dots, H_m$ are pairwise edge-disjoint and $\bigcup_{i=1}^m H_i = G$.}
Such a decomposition gives a 1-factorization of $G$, since each regular subgraph can be decomposed into perfect matchings by Hall's theorem.
More precisely, to facilitate an iterative framework, we find such a decomposition $\{R_{i,j} \subseteq G : i \in [\ell], j \in [2^{\ell - i}]\}$ with $2^{\ell} - 1$ parts for some $\ell \in \mathbb N$.  We choose $\ell$ such that $2^{\ell} - 1 = n / (C\log n)$, where $C$ is a large absolute constant.
By `well-spread', we mean that 
\begin{enumerate}[align=parleft, left=0pt..0pt, label={($\star'$)}]
\item\label{eq:spreadness2}
  {\hfil every $\{S_{i,j} \subseteq E(G) : i \in [\ell], j \in [2^{\ell - i}]\}$ satisfies}
  \begin{equation*}
    \Prob{S_{i,j} \subseteq E(R_{i,j})~\forall i \in [\ell], \forall j \in [2^{\ell - i}]} \leq \left(\frac{C_{\ref*{thm:LS-spreadness}} \log n}{n}\right)^{\sum_{i,j} |S_{i,j}|}.
  \end{equation*}
\end{enumerate}

\subsubsection{An approximate decomposition}

To find such a decomposition, the first step of the proof is to consider a decomposition $\{H_{i,j} \subseteq G : i \in [\ell], j \in [2^{\ell - i}]\}$ of $G$ into $2^{\ell} - 1$ parts chosen uniformly at random.
In a sense, this decomposition is already `close' to the one we want, and the distribution clearly satisfies \ref{eq:spreadness2}.  It is of course very unlikely that every $H_{i,j}$ is regular, but a straightforward application of Chernoff's bound (Lemma~\ref{lemma:chernoff}) and a union bound shows that \aas every $H_{i,j}$ is at least \textit{nearly} regular. 
Moreover, a similar argument yields that \aas every $H_{i,j}$ satisfies a few additional quasirandomness properties (formally described in Definition~\ref{def:quasirandomness} and proved in Lemma~\ref{lemma:edge-vortex}), and these quasirandomness properties imply (via Lemma~\ref{lemma:f-factor}) that each $H_{i,j}$ contains a spanning regular subgraph with almost all of its edges.  (Since each $H_{i,j}$ is a binomial random graph with expected average degree $C\log n$, this argument strengthens Erd\H{o}s and R\'{e}nyi's~\cite{ER64} classic result that $\log n / n$ is the threshold for finding a perfect matching in a random bipartite graph.  This part of our proof uses standard probabilistic and graph-theoretic techniques.  We remark that one of the main reasons we need the $\log n$ term in Theorem~\ref{thm:LS-spreadness} is this application of Chernoff's bound.)
Therefore, by conditioning on these quasirandomness properties we can construct a distribution satisfying \ref{eq:spreadness2} on \textit{approximate} decompositions of $G$ into spanning $O(\log n)$-regular subgraphs.  Converting this approximate decomposition into a complete one is the core of our argument, where ideas from iterative absorption come into play.

The above argument for finding an approximate decomposition of $G$ into spanning $O(\log n)$-regular subgraphs relies on the fact that $G$ is complete (or at least, nearly complete), whereas the above argument for decomposing $G$ into nearly $O(\log n)$-regular subgraphs only requires that $G$ itself is nearly regular.  
As explained in the next subsection, we use this latter fact iteratively (in $\ell - 1$ steps) to find a complete decomposition $\{R_{i,j} \subseteq G : i \in [\ell], j \in [2^{\ell - i}]\}$ of $G$ into spanning $O(\log n)$-regular subgraphs.

\subsubsection{Iterative absorption}\label{sect:iterative-absorption}

As mentioned, we start by choosing a decomposition $\{H_{i,j} \subseteq G : i \in [\ell], j \in [2^{\ell - i}]\}$ of $G$ into $2^{\ell} - 1$ parts chosen uniformly at random.  With high probability, each $H_{i,j}$ satisfies several quasirandomness properties, in which case we call $\{H_{i,j} \subseteq G : i \in [\ell], j \in [2^{\ell - i}]\}$ an \textit{edge vortex} (see Definition~\ref{def:edge-vortex}).  We think of $\bigcup_{j=1}^{2^{\ell - i}}H_{i,j}$ as the `$i$th level' of the vortex.  Given an edge vortex, we (randomly) construct the decomposition into spanning $O(\log n)$-regular subgraphs, as follows.
First, as a convenient initialization, we let $R_{1, j}$ be an empty graph for each $j \in [2^{\ell - 1}]$, and for each $i \in [\ell - 1]$, we iteratively construct a set of regular graphs $\{R_{i+1, j} : j \in[2^{\ell - i - 1}]\}$.
In each step except the final one, we decompose the `leftover' $L_i \coloneqq \bigcup_{j=1}^{2^{\ell - i}}(H_{i, j} - E(R_{i, j}))$ of the previous step into $2^{\ell - i - 1}$ parts $\{L_{i, j} : j \in [2^{\ell - i - 1}]\}$ uniformly at random.  (Since $L_1 = \bigcup_{j=1}^{2^{\ell - 1}}H_{1, j}$, it is not necessary to define the $H_{1, j}$ individually, but the notation is more convenient this way.)  With high probability, $L_{i,j}$ is nearly regular for each $j \in [2^{\ell - i - 1}]$, and the quasirandomness properties of $H_{i + 1,j}$ ensure that $L_{i,j} \cup H_{i + 1, j}$ contains a spanning regular subgraph $R_{i + 1, j}$ that itself contains $L_{i,j}$ (see Lemma~\ref{lemma:non-spread-cover-down-lemma}).  In this way, we decompose the leftover of the $i$th level of the edge vortex into regular subgraphs by incorporating some additional edges from the $(i + 1)$th level (see Lemma~\ref{lemma:inductive-cover-down-lemma}).  This part of the argument can be viewed as the so-called `cover down' step in the iterative absorption framework.
To complete the decomposition, we need to decompose the union of $H_{\ell, 1}$ and the `final leftover' $L_{\ell - 1} \coloneqq \bigcup_{j=1}^{2}(H_{\ell - 1, j} - E(R_{\ell - 1, j}))$ into spanning $O(\log n)$-regular subgraphs.  A reader familiar with iterative absorption might expect $H_{\ell, 1}$ to require some special `absorbing' properties to facilitate this process.  However, in our case, $L_{\ell - 1} \cup H_{\ell, 1}$ is already $O(\log n)$-regular because it is the complement of the regular graph $\bigcup_{i=1}^{\ell - 2}\bigcup_{j=1}^{2^{\ell - i}}R_{i,j}$.  Thus, simply letting $R_{\ell, 1} \coloneqq L_{\ell - 1} \cup H_{\ell, 1}$ completes the decomposition.  The full details of this iterative argument are provided in Lemma~\ref{lemma:main-decomposition-lemma}. 
We remark that this idea of using a regular bipartite graph as an absorber is also present in the work of Ferber, Jain, and Sudakov~\cite{FJS20}.

We conclude by explaining why this decomposition satisfies \ref{eq:spreadness2}.  We only construct the decomposition in the event that $\{H_{i,j} \subseteq G : i \in [\ell], j \in [2^{\ell - i}]\}$ is an edge vortex and each $L_{i,j}$ is nearly regular, but these events hold with very high probability, so let us ignore this aspect of the proof here.  First, we sketch the argument for when $|\bigcup_{i,j}S_{i,j}| = 1$; in this case, we may assume $S_{i^*, j^*} = \{e\}$ and $S_{i, j} = \varnothing$ for all $(i, j) \neq (i^*, j^*)$.  By construction, there are only two situations in which the edge $e$ can end up in $R_{i^*,j^*}$: either $e \in E(H_{i^*, j^*})$, or $e \in E(L_{i^* - 1, j^*})$.  Both of these events occur with probability at most $O(\log n / n)$; indeed,
\begin{equation*}
  \Prob{e \in E(H_{i^*, j^*})} = \frac{1}{2^{\ell} - 1} = O\left(\frac{\log n}{n}\right)
\end{equation*}
and
\begin{equation*}
  \Prob{e \in E(L_{i^* - 1, j^*})} \leq \left.\Prob{e \in \bigcup_{j=1}^{2^{\ell - i^* + 1}}E(H_{i^* - 1, j})} \middle. \cdot \frac{1}{2^{\ell - i^*}}\right.\COMMENT{$=2 / (2^{\ell} - 1)$}= O\left(\frac{\log n}{n}\right).
\end{equation*}
Hence, $\Prob{e \in R_{i^*, j^*}} = O(\log n / n)$, as desired.
To extend this argument to any choice of $\{S_{i,j} \subseteq E(G) : i \in [\ell], j \in [2^{\ell - i}]\}$, we use a coupling argument 
-- see the proof of Theorem~\ref{thm:LS-spreadness} for details.
Via the coupling, it is straightforward to prove the decomposition satisfies \ref{eq:spreadness2}. 
Thus, as mentioned in Section~\ref{sect:spreadness}, taking arbitrary $1$-factorizations of each $R_{i,j}$ will yield a $(C_{\ref*{thm:LS-spreadness}} \log n / n)$-spread distribution on $1$-factorizations of $G$.

\subsection{Overview of Theorem~\ref{thm:sts-reduction}: Reducing to the partite case}

For simplicity, we only consider the case $|V_1|=|V_2|=|V_3|$.
In this case, we choose sets $U_i \subseteq V_i$ of size $\lfloor \eps n \rfloor$ for every $i \in [2]$, and let $U_3 \coloneqq W$. 
First, we aim to cover all of the edges in $G[V_i] - E(G[U_i])$ and most of the edges in $G[U_i]$ for every $i \in [3]$ with a set  $\cT^2$ of edge-disjoint random triangles (see Step~\ref{step:step2}). 
To that end, for every $i \in [3]$, we approximately decompose the edges in $G[V_i] - E(G[U_i])$ using the R\"odl nibble (by applying Corollary~\ref{cor:almostdecomp}) into edge-disjoint random triangles and then we `cover down' all of the leftover edges $uv$ in $G[V_i] - E(G[U_i])$ using random triangles of the form $uvw$ where $w \in U_i$. This is accomplished by applying  Lemma~\ref{lem:coverdown}. Then we cover most of the edges in $G[U_i]$ for every $i \in [3]$ with edge-disjoint random triangles (using Corollary~\ref{cor:almostdecomp} again). Let $G^2 \coloneqq G - \bigcup_{T \in \cT^2}E(T)$. As discussed, $d_{G^2}(v,V_i) = 0$ for each $v \in \bigcup_{i \in [3]} (V_i \setminus U_i)$ and $e(G^2[U_i]) = o(n^2)$ for each $i \in [3]$.

In Step~\ref{step:step3} of the proof, we will define a collection $\cT^3$ of random triangles in $G^2$, where $\cT^3$ consists of the triangles of the form $uvw$ where $uv \in E(G^2[U_i])$ for some $i \in [3]$ and $w \in U_j$ for some $j \in [3] \setminus \{ i \}$. 
The triangles in $\cT^3$ will be carefully chosen to ensure the graph $G^3 \coloneqq G^2 - \bigcup_{T \in \cT^3} E(T)$ satisfies $e(G^2[V_i]) = 0$ for each $i \in [3]$ and the following divisibility condition: $d_{G^3}(v , U_j) = d_{G^3}(v , U_k)$ for every $v \in V_i$, $i \in [3]$, and distinct $j,k \in [3]\setminus \{ i \}$; this divisibility condition ensures that the final step (Step~\ref{step:step4}) can be carried out successfully.

To explain how to define $\cT^3$ in more detail, for each $i \in [3]$ and distinct $j,k \in [3] \setminus \{ i \}$ with $j < k$, we randomly divide $E(G^2[U_i])$ into two parts $E_{i,j}$ and $E_{i,k}$ of the same size. 
For each $uv \in E_{i,j}$ (or $uv \in E_{i,k}$), we will choose some random triangle $uvw$ for some $w \in U_j$ ($w \in U_k$, respectively) and put into $\cT^3$, so $G^3 = G^2 - \bigcup_{T \in \cT^3}E(T)$ has no edges inside $U_1$, $U_2$, and $U_3$, respectively.
For all $u \in U_i$, let $E_{i,j}(u)$ be the set of the edges incident to $u$ in $E_{i,j}$, and let $E_{i,k}(u)$ be the set of the edges incident to $u$ in $E_{i,k}$. By the Chernoff's bound, with probability $1-o(1)$, ${\rm def}(u) \coloneqq |E_{i,k}(u)| - |E_{i,j}(u)|$ is at most $n^{2/3}$ for all $u \in U_i$. 
If we delete the edges of triangles $uvw \in \cT^3$ covering $E_{i,j}(u) \cup E_{i,k}(u)$ from $G^2$, it reduces the degree $d_{G^2}(u,U_j)$ and $d_{G^2}(u,U_k)$ by $|E_{i,j}(u)|$ and $|E_{i,k}(u)|$, respectively. Thus, the resulting degrees of $u$ to $U_j$ and $U_k$ may differ by at most $|{\rm def}(u)| \leq n^{2/3}$. 
Since $|{\rm def}(u)|$ is small, we can choose some random triangles $uvw$ for some $vw \in E_{j,i} \cup E_{k,i}$ carefully (see Figure~\ref{fig:step3}); after deleting the edges of such triangles, the divisibility condition $d_{G^3}(u,U_j) = d_{G^3}(u,U_k)$ is satisfied for each $u \in U_i$.

Finally, in Step~\ref{step:step4} of the proof, we cover all the edges in $G^3$ containing a vertex of $U_3$ with edge-disjoint random triangles of the form $u v w$ where $u \in U_3$, $v \in V_1$ and $w \in V_2$. Then after removing the edges of these triangles from $G^3$, the graph induced by $V_1 \cup V_2$ is the desired subgraph $H$.  Indeed, every edge of $G - E(H)$ is covered by a triangle, and since $G^3$ satisfies the divisibility condition, $H$ is $|V_3 \setminus W|$-regular, as desired. 

\section{Notation and preliminaries}\label{sect:prelim}

In this section we introduce some notation and terminology along with some tools that are used in the proofs of Theorems~\ref{thm:sts-reduction}, \ref{thm:list-reduction}, and \ref{thm:LS-spreadness}.

\subsection{Basic terminology}

By $\log x$, we mean the natural logarithm of $x$.
For $n \in \mathbb{N}$, we let $[n] \coloneqq \{k \in \mathbb{N} \: : \:1 \leq k \leq n \}$.
For $a, b, c , r \in \mathbb R$ with $b,r \geq 0$, we write $c = a \pm b$ if $a-b \leq c \leq a+b$, and we write $c = (a \pm b)r$ if $ar - br \leq c \leq ar + br$.
We say a partition of a set is \textit{equitable} of the part sizes differ by at most $1$.

We use the `$\ll$' notation to state many of our results.
We write $x \ll y$ to mean that there exists a non-decreasing function $f : (0,1] \mapsto (0,1]$ such that the subsequent result holds for all $x, y \in (0, 1]$ with $x \leq f(y)$.
We will not calculate these functions explicitly.
Hierarchies with more constants are defined in a similar way and are to be read from right to left.  

For a graph $G$, we let $G^{(3)}$ denote the set of all triangles in $G$.
For disjoint sets $X, Y \subseteq V(G)$, we let $G[X, Y]$ be the bipartite graph with vertex set $X \cup Y$ and edge set $\{xy \in E(G) : x \in X, y \in Y\}$.
Otherwise, our graph theory notation and terminology is standard. For completeness, we also mention the following that we frequently use.  Let $G$ be a graph.  A subgraph $H \subseteq G$ is \textit{spanning} if $V(H) = V(G)$.  We let $e(G) \coloneqq |E(G)|$, and for $X, Y \subseteq V(G)$, we let $e_G(X, Y) \coloneqq |\{xy \in E(G) : x \in X, y \in Y\}$.  The \textit{degree} of a vertex $v\in V(G)$, denoted $d_G(v)$, is the number of edges containing $v$, and for $U\subseteq V(G)$, we let $d_G(v, U) \coloneqq e_G(\{v\}, U)$.  We may omit the subscripts when $G$ is clear from the context.  A graph $G$ is \textit{regular} if every vertex in $G$ has the same degree and \textit{$r$-regular} if every vertex in $G$ has degree $r$.
For $a,b,c \in \mathbb{R}$ with $b,c \geq 0$, a graph $G$ is \emph{$(a \pm b)c$-regular} if every vertex of $G$ has degree in $[(a-b)c , (a+b)c]$.

We also use standard notation and terminology from probability theory.  Recall that for a graph $G$ and $p \in [0, 1]$, we let $\randt(G, p)$ be a random set of triangles of $G$ where each is included independently with probability $p$.  For a set $X$ and $p \in [0, 1]$, we also let $X(p)$ denote a random subset of $X$ in which each $x \in X$ is included with probability $p$ independently at random.  Thus, $\randt(G, p) = G^{(3)}(p)$ for every graph $G$.

\subsection{Tools}

We will often use the following standard Chernoff-type bounds (see \cite[Theorem~A.1.12]{AS16} and \cite[Corollary~2.3]{JLR00}).  Since we use Lemma~\ref{lemma:chernoff}\ref{chernoff:two-sided} below more frequently, if we only refer to Lemma~\ref{lemma:chernoff}, then we mean Lemma~\ref{lemma:chernoff}\ref{chernoff:two-sided}.

\begin{lemma}\label{lemma:chernoff}
  If $X$ is the sum of mutually independent Bernoulli random variables with $\mu \coloneqq \Expect{X}$, then the following holds.
  \begin{enumerate}[(i)]
  \item\label{chernoff:two-sided} For all $\delta \in [0, 1]$, we have
    \begin{equation*}
      \Prob{|X - \mu| \geq \delta \mu} \leq 2e^{-\delta^2 \mu / 3}.
    \end{equation*}
  \item\label{chernoff:upper-tail} For all $\beta > 1$, we have
    \begin{equation*}
      \Prob{X > \beta \mu} \leq (\beta / e)^{-\beta \mu}.
    \end{equation*}
  \end{enumerate}
\end{lemma}

We will also use the following result, which is an easy application of the max-flow min-cut theorem (see e.g.~\cite[p.~57]{Lo07}).

\begin{lemma}\label{lemma:f-factor}
  Let $G$ be a bipartite graph with bipartition $\{A, B\}$, and let $f : V(G) \rightarrow \mathbb Z_{\geq0}$.  There exists a spanning subgraph $H\subseteq G$ satisfying $d_H(v) = f(v)$ for every $v\in V(G)$ if and only if
  \begin{itemize}
  \item $\sum_{a \in A} f(a) = \sum_{b \in B}f(b)$ and
  \item $e(A', B') \geq \sum_{a \in A'}f(a) - \sum_{b \in B\setminus B'}f(b)$ for every $A' \subseteq A$ and $B' \subseteq B$.
  \end{itemize}
\end{lemma}

\section{Proof of Theorem~\ref{thm:LS-spreadness}}\label{sect:spreadness-proof}

In this section we prove Theorem~\ref{thm:LS-spreadness}, which combined with Theorem~\ref{thm:FKNP}, implies our main result, Theorem~\ref{thm:robust-ls-threshold}.

\subsection{Edge vortices}

The first step of the proof is to decompose our nearly complete bipartite graph $G$ into $n / (C\log n)$ parts uniformly at random, where $C$ is a large absolute constant.  In this subsection we prove Lemma~\ref{lemma:edge-vortex}, which states that every part of such a decomposition satisfies several quasirandomness properties, which are listed in the following definition.

\begin{definition}[Quasirandom graph]\label{def:quasirandomness}
  Let $\delta, p \in [0, 1]$, and let $H$ be a bipartite graph on $\{A, B\}$ with $|A| = |B| = n$.  We say $H$ is \textit{$(\delta, p)$-quasirandom} if the following holds for some $\delta' \in [0, \delta]$:
  \begin{enumerate}[(QR\arabic*)]
  \item\label{QR-regularity} $H$ is $(1 \pm \delta)pn$-regular,
  \item\label{QR-lower-quasirandomness} every $X \subseteq A$ and $Y \subseteq B$ with $|X|, |Y| \geq \delta' n$ satisfies $e(X, Y) \geq (1 - \delta)p|X||Y|$,
  \item\label{QR-left-upper-quasirandomness} every $X\subseteq A$ and $Y \subseteq B$ with $\delta' n > |X| > |Y| / 2$ satisfies $e(X, Y) \leq pn|X| / 3$, and
  \item\label{QR-right-upper-quasirandomness} every $X\subseteq A$ and $Y \subseteq B$ with $\delta' n > |Y| > |X| / 2$ satisfies $e(X, Y) \leq  pn|Y| / 3$.
  \end{enumerate}
  If $G$ is a bipartite graph and we write $H \subseteq G$ is $(\delta, p)$-quasirandom, then we assume $H$ is a spanning subgraph of $G$.  
\end{definition}

  Note that if $H$ is $(\delta, p)$-quasirandom, then $H$ is $(\delta', p)$-quasirandom for every $\delta' \in [\delta, 1]$.  We use this fact throughout this section without further mention.

Our iterative absorption argument described in Section~\ref{sect:iterative-absorption} requires that we work with a decomposition of $G$ in which every part is quasirandom.  Moreover, it is convenient to index the subgraphs in the decomposition in a particular way.  We refer to such a decomposition as an `edge vortex', as in the following definition.

\begin{definition}[Edge vortex]\label{def:edge-vortex}
  Let $\delta, p \in [0, 1]$.  A \textit{$(\delta, p)$-edge-vortex} of a bipartite graph $G$ is a decomposition $\{H_{i, j} \subseteq G : i \in [\ell], j \in [2^{\ell - i}]\}$ of $G$ into 
  $(\delta, p)$-quasirandom subgraphs.
\end{definition}

Now we show that a uniformly random decomposition of a nearly complete bipartite graph is an edge vortex with high probability.  The proof uses fairly standard probabilistic arguments, including the Chernoff-type bounds from Lemma~\ref{lemma:chernoff} and union bounds.

\begin{lemma}\label{lemma:edge-vortex}
  Let $1 / n_0 \ll 1 / C, \alpha \ll \delta \ll  1$, let $n \geq n_0$, and let $G$ be a spanning subgraph of $K_{n,n}$ with minimum degree at least $(1 - \alpha)n$.  Let $\ell \in \mathbb N$, and choose a decomposition $\{H_{i, j} \subseteq G : i \in [\ell], j \in [2^{\ell - i}]\}$ of $G$
  uniformly at random.
  If $p \coloneqq 1/(2^{\ell} - 1) \geq C\log n / n$, then $\{H_{i, j}\}$ is a $(\delta,  p)$-edge-vortex with probability at least $1 - n^{-\sqrt C}$.
\end{lemma}
\begin{proof}
  For every $e \in E(G)$, $i \in [\ell]$, and $j \in [2^{\ell - i}]$, let $\rvX_{e, i, j} = 1$ if $e \in E(H_{i, j})$ and $0$ otherwise, and note that for fixed $i \in [\ell]$ and $j \in [2^{\ell - i}]$, the random variables $\{\rvX_{e, i, j} : e \in E(G)\}$ are mutually independent.  Therefore, by Lemma~\ref{lemma:chernoff}\ref{chernoff:two-sided}, for every $v \in V(G)$,\COMMENT{since $\Expect{d_{H_{i,j}}(v)} = pn(1 \pm \alpha)$} we have
  \begin{equation}
    \label{eq:chernoff-degree}
    \Prob{d_{H_{i,j}}(v) \neq (1 \pm \delta)pn} \leq 2\exp\left(-\delta^2 pn / 4\right) \leq 2\exp\left(-\delta^2 C \log n / 4\right)
  \end{equation}
  and for every $X \subseteq A$ and $Y \subseteq B$ with $|X|, |Y| \geq \delta n$,\COMMENT{since $e_G(X, Y) \geq |X|(|Y| - \alpha n) \geq |X||Y|(1 - \alpha / \delta)$} we have
  \begin{equation}
    \label{eq:chernoff-lower-quasirandomness}
\Prob{e_{H_{i, j}}(X, Y) < (1 - \delta)p|X||Y|} \leq 2\exp\left(-\delta^2 p|X||Y| / 4\right) \leq 2\exp\left(-\delta^4C n\log n / 4\right).
  \end{equation}
  In addition, by Lemma~\ref{lemma:chernoff}\ref{chernoff:upper-tail}, for every $X \subseteq A$ and $Y \subseteq B$ with $\delta n > |X| > |Y| / 2$,\COMMENT{since $e_{H_{i, j}}(X, Y)$ is stochastically dominated by $\mathrm{BIN}(|X||Y|, p)$} we have
  \begin{equation}
    \label{eq:chernoff-upper-quasirandomness}
    \Prob{e_{H_{i, j}}(X, Y) > pn|X| / 3} \leq \left(\frac{n}{3e|Y|}\right)^{- pn|X| / 3} \leq \left(\frac{n}{6e|X|}\right)^{-C |X|}.
  \end{equation}
  Thus, by a union bound, we have by \eqref{eq:chernoff-degree} that
  \begin{equation}
    \label{eq:degree-union-bound}
    \Prob{d_{H_{i,j}}(v) = (1 \pm \delta)pn~\forall v} \geq 1 - n^{-C^{3/4}},
  \end{equation}
  by \eqref{eq:chernoff-lower-quasirandomness} that
  \begin{equation}
    \label{eq:lower-quasirandomness-union-bound}
    \Prob{e_{H_{i,j}}(X, Y) \geq (1 - \delta)p|X||Y|~\forall X\subseteq A, \forall Y\subseteq B, |X|, |Y| \geq \delta n} \geq 1 - n^{-C^{3/4}},
  \end{equation}
  and by \eqref{eq:chernoff-upper-quasirandomness}\COMMENT{since it suffices to check the property for all $X$ of size $k$ and $Y$ of size $2k$ for all $k \in [\lfloor \delta n\rfloor]$} that
  \begin{multline}
    \label{eq:upper-quasirandomness-union-bound}
    \Prob{e_{H_{i, j}}(X, Y) \leq pn|X| / 3~\forall X \subseteq A, \forall Y \subseteq B, \delta n > |X| > |Y| / 2} \\ 
    > 1 - \sum_{k=1}^{\lfloor \delta n\rfloor} 2k\binom{n}{k}\binom{n}{2k}\left(\frac{n}{6ek}\right)^{-C k} > 1 - n^{-C^{3/4}}.
  \end{multline}
  \COMMENT{Using the bound $\binom{n}{k} \leq (en / k)^k$, we have
  \begin{equation*}
    \sum_{k=1}^{\lfloor \delta n\rfloor} 2k \binom{n}{k}\binom{n}{2k}\left(\frac{n}{6ek}\right)^{-C k} \leq 
    \COMMENT{$2k \sum_{k=1}^{\lfloor \delta n\rfloor}\binom{n}{2k}^2\left(\frac{n}{6ek}\right)^{-C k} \leq 2k\sum_{k=1}^{\lfloor \delta n\rfloor}\left(\frac{en}{2k}\right)^{4k}\left(\frac{n}{6ek}\right)^{-C k} = $} 
    2k\sum_{k=1}^{\lfloor \delta n\rfloor} \left(\frac{1}{3e^2}\right)^{-Ck}\left(\frac{en}{2k}\right)^{-k(C - 4)} \leq 
    \COMMENT{2k$\sum_{k=1}^{\lfloor \delta n\rfloor} \left(\sqrt\delta\right)^{-9CK/10}\left(\frac{n}{k}\right)^{-9CK/10}\leq $}
    2k\sum_{k=1}^{\lfloor \delta n\rfloor} \left(\frac{\sqrt \delta n}{k}\right)^{-9Ck / 10}.
  \end{equation*}
  Since $(\sqrt \delta n / k)^{-9Ck/10} \leq (\sqrt \delta n)^{-9C / 10}$
  for all $k \in [\lfloor \delta n \rfloor]$, we have
  \begin{equation*}
    \sum_{k=1}^{\lfloor \delta n\rfloor} \left(\frac{\sqrt \delta n}{k}\right)^{-9Ck / 10} \leq \delta n (\sqrt \delta n)^{-9C / 10}\leq n^{-C^{3/4}}. 
  \end{equation*}
  Combining the previous three inequalities, we have
  \begin{equation*}
    \Prob{e_{H_{i, j}}(X, Y) \leq pn|X| / 3~\forall X \subseteq A, \forall Y \subseteq B, \delta n > |X| > |Y| / 2} > 1 - n^{-C^{3/4}}.
  \end{equation*}}
  Essentially the same argument yields
  \begin{equation}
    \label{eq:upper-quasirandomness-union-bound2}
    \Prob{e_{H_{i, j}}(X, Y) \leq  pn|Y| / 3~\forall X \subseteq A, \forall Y \subseteq B, \delta n > |Y| > |X| / 2} > 1 - n^{-C^{3/4}}.
  \end{equation}
  Combining \eqref{eq:degree-union-bound}--\eqref{eq:upper-quasirandomness-union-bound2} with a union bound over $i \in [\ell]$ and $j \in [2^{\ell - i}]$, we have that \ref{QR-regularity}--\lastQRcondition holds simultaneously for every $H_{i,j}$ with probability at least $1 - n^{-\sqrt C}$.  In particular, in this event every $H_{i,j}$ is $(\delta, p)$-quasirandom, so $\{H_{i,j}\}$ is a $(\delta, p)$-edge-vortex, as desired.
\end{proof}

\subsection{Covering down}

In this subsection, we show (in Lemma~\ref{lemma:main-decomposition-lemma}) how, given an edge vortex of a regular graph $G$, we can randomly decompose $G$ into regular subgraphs via an iterative-absorption-type argument.  A key component of this argument is the following lemma, which uses Lemma~\ref{lemma:f-factor} to show how edges of a quasirandom graph (in the sense of Definition~\ref{def:quasirandomness}) can be added to a nearly regular graph to make it regular.

\begin{lemma}\label{lemma:non-spread-cover-down-lemma}
  Let $1 / n_0 \ll \delta \ll 1$, let $n \geq n_0$, let $p \in [0,1]$, and let $\gamma \geq 0$. Let $H \subseteq K_{n,n}$ be $(\delta, p)$-quasirandom.
  If $L$ is a $(\gamma \pm \delta)pn$-regular spanning subgraph of $K_{n,n} - E(H)$,
  then there exists $R \subseteq H \cup L$ such that $R \supseteq L$ is spanning and regular.
\end{lemma}
\begin{proof}
  We find $R \supseteq L$ to be $(\gamma + 3\delta)pn$-regular.  To that end, let $f(v) \coloneqq (\gamma + 3\delta)pn - d_L(v)$ for every $v \in V(H)$.  We claim that there is a spanning subgraph $H' \subseteq H$  satisfying $d_{H'}(v) = f(v)$ for every $v \in V(H)$, in which case $R \coloneqq H' \cup L$ is $(\gamma + 3\delta)pn$-regular, as desired.  Let $\{A, B\}$ be the bipartition of $H$, and note that $\sum_{a\in A} f(a) = \sum_{b\in B} f(b) = (\gamma + 3\delta)pn^2 - |E(L)|$.  Hence, by Lemma~\ref{lemma:f-factor}, it suffices to show that 
  \begin{equation}
    \label{eq:hall-type-condition}
    \text{$e(A', B') \geq \sum_{A \in A'} f(a) - \sum_{b \in B\setminus B'}f(b)$ for every $A' \subseteq A$ and $B' \subseteq B$.}
  \end{equation}

  To that end, let $A' \subseteq A$ and $B' \subseteq B$, let $A'' \coloneqq A \setminus A'$ and $B'' \coloneqq B \setminus B'$, and note that\COMMENT{since $|A'| - |B''| = (n - |A''|) - (n - |B'|) = |B'| - |A''|$ and $\sum_{a\in A'}d_L(a) - \sum_{b\in B''}d_L(b) = \sum_{a\in A'}d_L(a) - \sum_{a \in A}d_L(a) + \sum_{b\in B}d_L(b) - \sum_{b\in B''}d_L(b) = \sum_{b\in B'}d_L(b) - \sum_{a\in A''}d_L(a)$}
  \begin{align*}
    \sum_{A \in A'} f(a) - \sum_{b \in B\setminus B'}f(b)
    &= pn(\gamma + 3\delta)(|A'| - |B''|) - \sum_{a \in A'}d_L(a) + \sum_{b \in B''}d_L(b)\\
    &= pn(\gamma + 3\delta)(|B'| - |A''|) - \sum_{b \in B'}d_L(b) + \sum_{a \in A''}d_L(a).
  \end{align*}
  Since $L$ is $(\gamma \pm \delta)pn$-regular, we have
  \begin{multline*}
    pn(\gamma + 3\delta)(|A'| - |B''|) - \sum_{a \in A'}d_L(a) + \sum_{b \in B''}d_L(b) \\
    \leq pn3\delta(|A'| - |B''|) + \delta pn(|A'| + |B''|) \leq pn\delta (4 |A'| - 2|B''|)
  \end{multline*}
  and similarly
  \begin{equation*}
    pn(\gamma + 3\delta)(|B'| - |A''|) - \sum_{b \in B'}d_L(b) + \sum_{a \in A''}d_L(a)
    \COMMENT{$\leq pn3\delta(|B'| - |A''|) - \delta pn(|A''| + |B'|)$}
    \leq pn\delta(4|B'| - 2|A''|).
  \end{equation*}
  Therefore, to show \eqref{eq:hall-type-condition}, it suffices to show that one of
  \begin{equation}
    \label{eq:s-is-min}
    e(A', B') \geq pn\delta(4|A'| - 2|B''|)
  \end{equation}
  and
  \begin{equation}
    \label{eq:t-is-min}
    e(A', B') \geq pn\delta(4|B'| - 2|A''|)
  \end{equation}
  holds.  We consider a few cases.

  First, if $|B''| \geq 2|A'|$, then \eqref{eq:s-is-min} clearly holds\COMMENT{because the left side is at least zero and the right side is at most zero}, and similarly, if $|A''| \geq 2|B'|$, then \eqref{eq:t-is-min} holds.  Thus, we may assume $|B''| < 2|A'|$ and $|A''| < 2|B'|$.
  
  Next, suppose $|A'| < \delta' n$, where $H$ satisfies \ref{QR-regularity}--\lastQRcondition for $\delta' \in [0, \delta]$.  Since $H$ has minimum degree at least $(1 - \delta) pn$, we have
  \begin{equation*}
    e(A', B') \geq (1 - \delta)pn |A'| - e(A', B''),
  \end{equation*}
  and by \ref{QR-left-upper-quasirandomness}, we have
  \begin{equation*}
    e(A', B'') \leq pn |A'| / 3.
  \end{equation*}
  Combining the inequalities above, we have $e(A', B') \geq (2/3 - \delta)pn |A'| \geq 4 pn \delta |A'|$, and \eqref{eq:s-is-min} holds, as desired.
  Similarly, if $|B'| < \delta' n$, then we can use essentially the same argument, using that $e(A', B') \geq (1 - \delta)pn|B'| - e(A'', B')$ and \ref{QR-right-upper-quasirandomness} to show that \eqref{eq:t-is-min} holds.\COMMENT{Since $H$ has minimum degree at least$(1 - \delta) pn$, we have
  \begin{equation*}
    e(A', B') \geq (1 - \delta)pn|B''| - e(A'', B'),
  \end{equation*}
    and by \ref{QR-right-upper-quasirandomness}, we have
  \begin{equation*}
    e(A'', B') \leq pn |B'| / 3.
  \end{equation*}
  Combining the inequalities above, we have $e(A', B') \geq (2/3 - \delta)pn |B'| \geq 4 pn \delta |B'|$, and \eqref{eq:t-is-min} holds, as desired.
}

Finally, suppose $|A'|, |B'| \geq \delta' n$.  By \ref{QR-lower-quasirandomness},
\begin{equation}
  \label{MFMC-lower-quasirandomess}
  e(A', B') \geq (1 - \delta)p|A'||B'| .
\end{equation}
If $|B'| \geq n / 3 \geq 4 \delta n / (1 - \delta)$, then \eqref{MFMC-lower-quasirandomess} implies $e(A', B') \geq 4pn\delta |A'|$, and \eqref{eq:s-is-min} holds, as desired.  Otherwise, since $|B''| < 2|A'|$, we have $|B''| > 2n / 3$ and $|A'| \geq n / 3$, in which case \eqref{MFMC-lower-quasirandomess} implies $e(A', B') \geq 4pn\delta |B'|$, and \eqref{eq:t-is-min} holds, as desired.  
Therefore \eqref{eq:hall-type-condition} holds, and the proof is complete.
%
\end{proof}

Next we prove what can be viewed as a `cover down' lemma, which we apply iteratively to prove Lemma~\ref{lemma:main-decomposition-lemma}.  This lemma `bootstraps' Lemma~\ref{lemma:non-spread-cover-down-lemma} to show how a nearly regular graph can be decomposed randomly, in a sufficiently `well-spread' way, into regular graphs, by incorporating some edges from some quasirandom graphs.  We obtain the `spreadness' by first decomposing the nearly regular graph randomly, before `regularizing' each piece via Lemma~\ref{lemma:non-spread-cover-down-lemma}.

\begin{lemma}\label{lemma:inductive-cover-down-lemma}
  Let $1 / n_0 \ll 1 / C \ll \delta \ll 1$, let $n \geq n_0$, let $p \in [0,1]$, and let $\gamma \geq 0$.  
  Let $H_1, \dots, H_m \subseteq K_{n, n}$ be pairwise edge-disjoint $(\delta,  p)$-quasirandom graphs.
  Let $L$ be a $(\gamma \pm \delta)mpn$-regular spanning subgraph of $K_{n, n} - E\left(\bigcup_{i=1}^m H_i \right)$, 
  and choose a decomposition $\{L_i \subseteq L : i \in [m]\}$ of $L$ 
  uniformly at random.
  If $p \geq C \log n / n$, then with probability at least $1 - n^{-\sqrt C}$, there exist pairwise edge-disjoint $R_i \subseteq H_i \cup L_i$ for each $i \in [m]$ such that $R_i \supseteq L_i$ is spanning and regular. 
\end{lemma}

\begin{proof}
  For every edge $e \in E(L)$ and $i \in [m]$, let $\rvX_{e, i} = 1$ if $e \in E(L_{i})$ and $0$ otherwise, and note that for fixed $i \in [m]$, the random variables $\{\rvX_{e, i} : e \in E(L)\}$ are mutually independent.  Therefore, by Lemma~\ref{lemma:chernoff}\ref{chernoff:two-sided}, for every vertex $v \in V(L)$,
  \COMMENT{using that $d_{L_i}(v)$ is a random variable stochastically dominated by $\mathrm{BIN}((\gamma + \delta)mnp, 1/m)$ and which stochastically dominates $\mathrm{BIN}(\max\{(\gamma - \delta)mnp, 0\}, 1/m)$}
  \COMMENT{$\Prob{d_{L_i}(v) \leq (\gamma + 2\delta)pn} \leq 2\exp(-((2\delta / (\gamma + \delta)^2)(\gamma + \delta)np / 3)) \leq \exp(-\delta^4 np) / 2$ and assuming that $\gamma > 2\delta$, $\Prob{d_{L_i}(v) \leq (\gamma - 2\delta)pn} \leq 2\exp(-((2\delta / (\gamma - \delta)^2)(\gamma - \delta)np / 3)) \leq \exp(-\delta^4 np) / 2$}
  \begin{equation*}
    \Prob{d_{L_i}(v) \neq (\gamma \pm 2\delta)pn} \leq \exp(-\delta^4 np) \leq \exp(-\delta^4 C \log n).
  \end{equation*}
  Hence, by a union bound over $i \in [m]$ and $v \in V(L)$, we have that $L_i$ is $(\gamma \pm 2\delta)pn$-regular for every $i \in [m]$ with probability at least $1 - n^{-\sqrt C}$.  In this event, by Lemma~\ref{lemma:non-spread-cover-down-lemma} applied with $H_i$, $L_i$, and $2\delta$ playing the role of $H$, $L$, and $\delta$, respectively for each $i \in [m]$, there exists $R_i \subseteq H_i \cup L_i$ that is spanning and regular with $R_i \supseteq L_i$, as desired.
\end{proof}

Now we show how, given an edge vortex of a regular graph $G$, we can iterate Lemma~\ref{lemma:inductive-cover-down-lemma} to randomly decompose $G$ into regular subgraphs in a `well-spread' way.

\begin{lemma}\label{lemma:main-decomposition-lemma}
  Let $1 / n_0 \ll 1 / C \ll \delta \ll 1$, let $n \geq n_0$, and let $p \geq C \log n / n$.
  Let $\{H_{i, j} \subseteq G : i \in [\ell], j \in [2^{\ell - i}]\}$ be a $(\delta, p)$-edge-vortex of a regular bipartite spanning subgraph $G$ of $K_{n,n}$.
  If $\randomBit_e \in [2^{\ell - i - 1}]$ is chosen independently and uniformly at random for each $i \in [\ell - 1]$ and $e \in \bigcup_{j=1}^{2^{\ell - i}}E(H_{i, j})$, then with probability at least $1 - n^{-\sqrt C}\log_2 n$ there exists a decomposition $\{R_{i, j} \subseteq G : i \in [\ell], j\in [2^{\ell - i}]\}$ of $G$ such that
  \begin{enumerate}[(\ref*{lemma:main-decomposition-lemma}.1)]
  \item\label{item:decomposition-regularity} $R_{i, j}$ is regular for each $i \in [\ell]$ and $j \in [2^{\ell - i}]$ and
  \item\label{item:decomposition-cover-down} every $e \in E(H_{i,j})$ for $i \in [\ell - 1]$ and $j \in [2^{\ell - i}]$ satisfies $e \in E(R_{i,j}) \cup E(R_{i+1, \randomBit_e})$.
  \end{enumerate}
\end{lemma}
\begin{proof}
  We prove the following by induction: For every $\ell' \in [\ell - 1]$, with probability at least $1 - (\ell' - 1) n^{-\sqrt C}$ there exists a decomposition $\{L_{\ell'}\} \cup \{R_{i, j} \subseteq G : i \in [\ell'], j \in [2^{\ell - i}]\}$ of $\bigcup_{i=1}^{\ell'}\bigcup_{j=1}^{2^{\ell - i}} H_{i, j}$ such that
  \begin{enumerate}[(a)]
  \item\label{inductive-decomposition:regularity} $R_{i,j}$ is regular for each $i \in [\ell']$ and $j \in [2^{\ell - i}]$,
  \item\label{inductive-decomposition:cover-down} every $e \in E(H_{i,j})$ for $i \in [\ell' - 1]$ and $j \in [2^{\ell - i}]$ satisfies $e \in E(R_{i,j}) \cup E(R_{i+1, \randomBit_e})$, and
  \item\label{inductive-decomposition:leftover} every $e \in E(H_{\ell', j})$ for $j \in [2^{\ell - \ell'}]$ satisfies $e \in E(R_{\ell', j}) \cup L_{\ell'}$.
  \end{enumerate}
  Having proved this, the lemma follows from the case $\ell' = \ell - 1$, by letting $R_{\ell, 1}\coloneqq L_{\ell - 1} \cup H_{\ell, 1}$. \COMMENT{Indeed, $R_{i, j}$ is regular for each $i \in [\ell - 1]$ and $j \in [2^{\ell - i}]$ by \ref{inductive-decomposition:regularity}, and since $\bigcup_{i=1}^{\ell - 1}\bigcup_{j=1}^{2^{\ell - i}}R_{i,j} = G - E(R_{\ell, 1})$ and $G$ are also regular, we have that $R_{\ell, 1}$ is regular as well.  Moreover, every $e \in E(H_{i,j})$ satisfies $e \in E(R_{i, j}) \cup E(R_{i + 1, \randomBit_e})$ for $i \in [\ell - 2]$ and $j \in [2^{\ell - i}]$ by \ref{inductive-decomposition:cover-down} and for $i = \ell - 1$ and $j \in [2]$ by \ref{inductive-decomposition:leftover}.}

  To prove the base case $\ell' = 1$, we let $R_{1, j} \subseteq H_{1, j}$ be spanning subgraphs with no edges for each $j \in [2^{\ell - 1}]$.  Now  $\{L_1\} \cup \{R_{1, 1}, \dots, R_{1, 2^{\ell-1}}\}$, where $L_1 \coloneqq \bigcup_{j=1}^{2^{\ell}-1}H_{1, j}$, satisfies \ref{inductive-decomposition:regularity}--\ref{inductive-decomposition:leftover}, as desired.

  Now we assume that for $\ell' \in [\ell - 2]$, as long as some event $\cB_{\ell'}$ satisfying $\Prob{\cB_{\ell'}} \leq (\ell' - 1)n^{-\sqrt C}$ does not hold, there exists a decomposition $\{L_{\ell'}\} \cup \{R_{i, j} \subseteq G : i \in [\ell'], j \in [2^{\ell - i}]\}$ of $\bigcup_{i=1}^{\ell'}\bigcup_{j=1}^{2^{\ell - i}} H_{i, j}$ satisfying \ref{inductive-decomposition:regularity}--\ref{inductive-decomposition:leftover} for $\ell' \in [\ell - 2]$, and we show that as long as some event $\cB_{\ell'+1}\supseteq \cB_{\ell'}$ satisfying $\Prob{\cB_{\ell' + 1}} \leq \ell'n^{-\sqrt C}$ does not hold, there exists a decomposition $\{L_{\ell' + 1}\} \cup \{R_{i, j} \subseteq G : i \in [\ell' + 1], j \in [2^{\ell - i}]\}$ of $\bigcup_{i=1}^{\ell' + 1}\bigcup_{j=1}^{2^{\ell - i}} H_{i, j}$ satisfying \ref{inductive-decomposition:regularity}--\ref{inductive-decomposition:leftover}.
  
  Let us assume that $\cB_{\ell'}$ does not hold, and let $m \coloneqq 2^{\ell - \ell' - 1}$.
  We will apply Lemma~\ref{lemma:inductive-cover-down-lemma} to $H_{\ell' + 1, 1}, \dots, H_{\ell' + 1, m}$, and $L_{\ell'}$ with 
  \begin{align*}
    \gamma \coloneqq \max\left\{0,\:\: \frac{d - \sum_{i=1}^{\ell'}\sum_{j=1}^{2^{\ell - i}}d_{i,j} - \sum_{i=\ell' + 1}^{\ell}pn 2^{\ell - i}}{mpn} \right\},    
  \end{align*}
  where $R_{i,j}$ is $d_{i,j}$-regular for each $i \in [\ell']$ and $j \in [2^{\ell - i}]$ and $G$ is $d$-regular, so we need to show that $L_{\ell'}$ is nearly $\gamma mpn$-regular. 
  By \ref{inductive-decomposition:regularity}, every vertex $v \in V(G)$ satisfies
  \begin{equation*}
    d_{L_{\ell'}}(v) = \COMMENT{$d - \sum_{i=1}^{\ell'}\sum_{j=1}^{2^{\ell-i}}d_{i,j} - \sum_{i=\ell' + 1}^{\ell}\sum_{j=1}^{2^{\ell - i}}d_{H_{i,j}}(v) = $} d - \sum_{i=1}^{\ell'}\sum_{j=1}^{2^{\ell-i}}d_{i,j} -  \sum_{i=\ell' + 1}^{\ell}(1 \pm \delta)pn 2^{\ell - i}.
  \end{equation*}
  Note that
  \begin{equation*}
    \sum_{i=\ell' + 1}^{\ell}pn 2^{\ell - i} = pn\sum_{i=0}^{\ell - \ell' - 1}2^i = pn\left(2^{\ell - \ell'} - 1\right).
  \end{equation*}
  In particular, since $\delta pn (2^{\ell - \ell'} - 1) < 2\delta mpn$, the two equalities above imply that 
  $L_{\ell'}$ is $(\gamma \pm 2\delta)mpn$-regular, as required.
  

  Decompose $L_{\ell'}$ into $\{L_{\ell', j} \subseteq L_{\ell'} : j \in [m]\}$ where $E(L_{\ell', j}) \coloneqq \{e \in E(L_{\ell'}) : \randomBit_e = j\}$, and note that $\{L_{\ell', j}\}$ is a uniformly random decomposition of $L_{\ell'}$.  
  Hence, 
  since $L_{\ell'}$ is $(\gamma \pm 2\delta)mpn$-regular, 
  by Lemma~\ref{lemma:inductive-cover-down-lemma}, as long as some event $\cB'_{\ell'}$ satisfying $\Prob{\cB'_{\ell'}} \leq n^{-\sqrt C}$ does not hold, there exist pairwise edge-disjoint $R_{\ell' + 1, j} \subseteq H_{\ell'+1, j} \cup L_{\ell', j}$ for each $j \in [2^{\ell - \ell' - 1}]$ such that $R_{\ell' + 1, j}$ is spanning and regular 
  and $R_{\ell' + 1, j} \supseteq L_{\ell', j}$ for each $j \in [2^{\ell - \ell' - 1}]$.  
  
  Now as long as $\cB_{\ell' + 1} \coloneqq \cB_{\ell'} \cup \cB'_{\ell'}$ does not hold, $\{L_{\ell' + 1}\} \cup \{R_{i, j} \subseteq G : i \in [\ell' + 1], j \in [2^{\ell - i}]\}$, where $L_{\ell'+1}\coloneqq \bigcup_{j=1}^{2^{\ell - \ell' - 1}}\left(H_{\ell'+1, j} - E(R_{\ell'+1, j})\right)$, satisfies \ref{inductive-decomposition:regularity}--\ref{inductive-decomposition:leftover}, as desired.
\end{proof}

\subsection{The proof}

Finally, we combine Lemmas~\ref{lemma:edge-vortex} and \ref{lemma:main-decomposition-lemma} to prove Theorem~\ref{thm:LS-spreadness}.  In the proof, we use Lemma~\ref{lemma:edge-vortex} to construct an edge vortex in $G$ and then use Lemma~\ref{lemma:main-decomposition-lemma} to randomly decompose $G$ into regular subgraphs.  Crucially, \ref{item:decomposition-cover-down}, combined with the fact that the edge vortex comes from a uniformly random decomposition, ensures that the distribution of the ultimate decomposition of $G$ into $1$-factorizations is $O(\log n / n)$-spread.

\begin{proof}[Proof of Theorem~\ref{thm:LS-spreadness}]
  Let $1 / n_0 \ll 1 / C, \alpha \ll \delta \ll 1$, let $\alpha_{\ref*{thm:LS-spreadness}} \coloneqq \alpha$, let $C_{\ref*{thm:LS-spreadness}} \coloneqq \max\{8C, n_0 / \log n_0 \}$, and let $G \subseteq K_{n,n}$ be a $d$-regular graph where $d \geq (1 - \alpha)n$.
  For $n < n_0$, any probability distribution on $1$-factorizations of $G$ is $1$-spread and $C_{\ref*{thm:LS-spreadness}} \log n / n \geq 1$,
  so it suffices to prove that for $n \geq n_0$, there is an $(8C \log n / n)$-spread probability distribution on $1$-factorizations of $G$.
  To that end, let $p \in [C\log n / n, 2C\log n / n]$ satisfy $p = 2^{-\ell}$ for $\ell \in \mathbb N$.

  Let $\randomBit_{e, 0} \in [2^\ell - 1]$ be chosen independently and uniformly at random for each $e \in E(G)$, and for each $i \in [\ell]$ and $j \in [2^{\ell - i}]$, let $H_{i,j}$ be the spanning subgraph of $G$ where $E(H_{i,j}) \coloneqq \{e \in E(G) : \randomBit_{e,0} = 2^{\ell - i} + j - 1\}$.
  Note that $\{H_{i, j} : i \in [\ell], j \in [2^{\ell - i}]\}$ is a decomposition of $G$ chosen uniformly at random, and let $\cB_1$ be the event that $\{H_{i, j} : i \in [\ell], j \in [2^{\ell - i}]\}$ is not a $(\delta,  p)$-edge-vortex.  By Lemma~\ref{lemma:edge-vortex}, $\Prob{\cB_1} \leq n^{-\sqrt C}$. 

  Now let $\randomBit_{e, i} \in [2^{\ell - i - 1}]$ for each $i \in [\ell - 1]$ be chosen independently and uniformly at random for each $e \in E(G)$.
  For each outcome not in $\cB_1$, we can apply Lemma~\ref{lemma:main-decomposition-lemma} with $\randomBit_{e, i}$ playing the role of $\randomBit_e$ for each $e \in \bigcup_{j=1}^{2^{\ell - i}}E(H_{i, j})$.
  Therefore, there is an event $\cB_2$ such that $\Prob{\cB_2} \leq n^{-\sqrt C}\log_2 n$ and for every outcome not in $\cB_1$ or $\cB_2$, there is a decomposition $\{R_{i, j} \subseteq G : i \in [\ell], j\in [2^{\ell - i}]\}$ of $G$ such that
  \begin{enumerate}[(a)]
  \item\label{item:regularity} $R_{i, j}$ is $d_{i,j}$-regular for each $i \in [\ell]$ and $j \in [2^{\ell - i}]$ for some $d_{i,j}\in\mathbb N$ and
  \item\label{item:cover-down} every $e \in E(H_{i,j})$ for $i \in [\ell - 1]$ and $j \in [2^{\ell - i}]$ satisfies $e \in E(R_{i,j}) \cup E(R_{i+1, \randomBit_{e, i}})$.
  \end{enumerate}
  Hence, by \ref{item:cover-down}, for every $i \in [\ell]$ and $j \in [2^{\ell - i}]$, if $e \in E(R_{i, j})$, then either $\randomBit_{e, 0} = 2^{\ell - i} + j - 1$ or $\randomBit_{e, 0} \in [2^{\ell - i + 1}, 2^{\ell - i + 2} - 1]$ and $\randomBit_{e, i-1} = j$.
  Since
  \begin{equation*}
    \Prob{\randomBit_{e, 0} = 2^{\ell - i} + j - 1} = \frac{1}{2^\ell - 1}
  \end{equation*}
  and
  \begin{equation*}
    \Prob{(\randomBit_{e, 0} \in [2^{\ell - i + 1}, 2^{\ell - i + 2} - 1]) \cap (\randomBit_{e, i-1} = j)} = \frac{2^{\ell - i + 1} - 1}{2^\ell - 1}\cdot\frac{1}{2^{\ell - i}} \leq \frac{2}{2^{\ell} - 1},
  \end{equation*}
  and since 
  \begin{equation*}
      \left.\left(\frac{1 + 2}{2^{\ell} - 1}\right)\middle/\Prob{\overline{\cB_1} \cap \overline{\cB_2}}\right. \leq\frac{4}{2^{\ell}} = 4p,
  \end{equation*}\COMMENT{$\left(3 / (2^{\ell} - 1)\right) / (1 - n^{-\sqrt C} - n^{-\sqrt{C}}\log_2 n) \leq 4 / 2^{\ell}$}
  we can couple the distributions of $\{\randomBit_{e, i} : i \in \{0\}\cup[\ell - 1]\}$ conditional on $\overline{\cB_1} \cap \overline{\cB_2}$ with the distributions of $\rvX_{e, i, j}$ so that the event $e \in E(R_{i, j})$ is stochastically dominated by $\rvX_{e, i, j}$, where the $\rvX_{e, i, j}$ are i.i.d.\ Bernoulli random variables satisfying $\Prob{\rvX_{e, i, j} = 1} = 4p$ for each $e \in E(G)$, $i \in [\ell]$, and $j \in [2^{\ell - i}]$.

  By \ref{item:regularity}, since each $R_{i, j}$ is bipartite, there is a decomposition $\{M_{i, j, k} \subseteq R_{i, j} : k \in [d_{i,j}]\}$ of $R_{i, j}$ into perfect matchings, and $(M_{i, j, k} : i \in [\ell], j\in[2^{\ell - i}], k\in [d_{i,j}])$ is a random $1$-factorization of $G$ which we claim is $(8C\log n / n)$-spread.  Indeed, given a set $\{S_{i,j,k} \subseteq E(G) : i \in [\ell], j\in[2^{\ell - i}], k\in [d_{i,j}]\}$ of pairwise edge-disjoint matchings in $G$, we have
  \begin{multline*}
    \Prob{S_{i, j, k} \subseteq M_{i, j, k}~\forall i \in [\ell], \forall j\in[2^{\ell - i}], \forall k\in [d_{i,j}]}\\
    \begin{aligned}
        &\leq \Prob{S_{i, j, k} \subseteq E(R_{i, j})~\forall i \in [\ell], \forall j\in[2^{\ell - i}], \forall k\in [d_{i,j}]}\\    
        &\leq \prod_{i, j, k}\prod_{e\in S_{i,j,k}} \Prob{\rvX_{e, i, j} = 1}\\
        &= (4p)^{\sum_{i,j,k} |S_{i, j, k}|} \leq \left(\frac{8C \log n}{n}\right)^{\sum_{i,j,k} |S_{i, j, k}|},
    \end{aligned}
  \end{multline*}
  as desired.
\end{proof}

\section{Steiner triple system reduction: Proof of Theorem~\ref{thm:sts-reduction}}\label{sect:sts-reduction}

In this section we prove Theorem~\ref{thm:sts-reduction}. Rather than proving Theorem~\ref{thm:sts-reduction}, we find it convenient to prove the following almost identical result, where we replace the `base graph' $G$ (as defined in Theorem~\ref{thm:sts-reduction}) by the complete graph $K_n$.
Since \ref{deg_red1} below ensures that the set $\cT$ of triangles obtained by Theorem~\ref{thm:red1} all lie in $G$ (as defined in Theorem~\ref{thm:sts-reduction}), it is not necessary to consider the outcome of the `exposure' of the triangles not spanned by $G$ at this point, so Theorem~\ref{thm:sts-reduction} follows from Theorem~\ref{thm:red1}.  Alternatively, one can of course obtain Theorem~\ref{thm:sts-threshold} from Theorems~\ref{thm:robust-ls-threshold} and \ref{thm:red1} directly, by exposing each triangle twice (once for Theorem~\ref{thm:red1} and once for Theorem~\ref{thm:robust-ls-threshold}) and adjusting the constant $C_{\ref*{thm:sts-threshold}}$ accordingly.

\begin{theorem}\label{thm:red1}
Let\COMMENT{This reduction is used to show the following. There exists $C \in \mathbb{R}$ such that if $n \equiv 1,3\pmod 6$ and $p \ge C \log n/n$, then with high
probability, $\randt(K_n,p)$ contains a Steiner triple system on $n$ vertices.} $1/n_0 \ll 1/C \ll \eps \ll 1$, and let $n \ge n_0$ satisfy $n \equiv 1,3 \mod 6$. Let $\{V_1 , V_2 , V_3\}$ be an equitable partition of $[n]$ with $|V_1| \leq |V_2| \leq |V_3|$, and let $G$ be a complete graph on vertices $V_1 \cup V_2 \cup V_3$. 
Let $W_3 \subseteq V_3$ such that $|W_3| = \lfloor \eps n \rfloor$. 
If $p \geq C \log n/n$, then with probability at least $1-1/n$, there exists a collection $\cT \subseteq \randt(G , p)$ of edge-disjoint triangles such that $H \coloneqq G - \bigcup_{T \in \cT} E(T)$ is a tripartite graph with parts $V_1 , V_2 , V_3$ satisfying the following properties.
\begin{enumerate}[(\ref*{thm:red1}:a), topsep = 6pt]
    \item \label{deg_red1} Every $u \in W_3$ satisfies $d_H(u) = 0$ and every $u \in V_3 \setminus W_3$ satisfies $N_H(u)=V_1 \cup V_2$.
    
    \item \label{reg_red1} $H[V_1 \cup V_2]$ is $(|V_3| - |W_3|)$-regular.
\end{enumerate}
\end{theorem}


\subsection{More preliminaries}
\subsubsection{Finding many edge-disjoint matchings in a random subgraph of a dense graph}



Since every $r$-regular bipartite graph contains $r$ edge-disjoint perfect matchings, by Lemma~\ref{lemma:f-factor},\COMMENT{with $f(v) = r$ for all $v \in V(G)$} every bipartite graph $G$ with bipartition $\{ X,Y \}$ such that $|X| = |Y| = n$ has $r$ edge-disjoint perfect matchings if and only if for every $X' \subseteq X$ and $Y' \subseteq Y$, we have $e(X', Y') \geq r(|X'| + |Y'|-n)$. 
Using this fact and Lemma~\ref{lemma:chernoff}, we obtain the following lemma.  The proof is straightforward, so we omit it.

\begin{lemma}\label{lem:manymatchings}
Let $1/n_0 \ll \eps , 1/C \ll 1$, and let $n \ge n_0$. Let $G$ be a bipartite graph with bipartition $\{X, Y\}$ such that $|X|=|Y|=n$ and $\delta(G) \geq (1-\eps)n$. 
If $p \geq C \log n / n$, then with probability at least $1 - n^{-C^{1/2}}$, there exist at least $C \log n/2$ edge-disjoint perfect matchings of $G$ in $E(p)$, where $E \coloneqq E(G)$.
\end{lemma}
\COMMENT{\begin{proof}For $1 \leq x,y \leq n$, the inequality $xy/n \geq x+y-n$ holds. 
Thus, it suffices to show that with probability $1 - n^{-C^{1/2}}$, $e_{G_p}(X',Y') \geq \frac{C \log n}{2n}|X'||Y'|$ for all subsets $X' \subseteq X$ and $Y' \subseteq Y'$ such that $|X'| \geq n-|Y'|$ and $|Y'| \geq n-|X'|$.
If $0 < |X'| < 3 \eps n$, for any $Y'$ with $|Y'| \geq n - |X'| > (1-3\eps)n$ (where there are at most $\sum_{x=0}^{|X'|} \binom{n}{x} \leq (1+n)^{|X'|}$ such choices of $Y'$), we have $d_G (x , Y') > |Y'| - \eps n > (1-4\eps)n$ for each $x \in X'$, thus $e_G(X',Y') > (1-4\eps)n|X'|$. 
Thus $\mathbb{E}[e_{G_p}(X',Y')] > (1-4\eps)C \log n |X'|$. By Lemma~\ref{lemma:chernoff}, with probability at least $1 - \exp(-|X'| \frac{C \log n}{100})$, $e_{G_p}(X',Y') > \frac{C \log n}{2}|X'| > \frac{C \log n}{2n} |X'||Y'|$.
Similarly, without loss of generality, if $0 < |Y'| < 3 \eps n$, for any $X'$ with $|X'| \geq n - |Y'| > (1-3\eps)n$ (where there are at most $\sum_{x=0}^{|Y'|} \binom{n}{x} \leq (1+n)^{|Y'|}$ such choices of $X'$) with probability at least $1 - \exp(-|Y'|\frac{C \log n}{100})$, $e_{G_p}(X',Y') > \frac{C \log n}{2} |Y'| > \frac{C \log n}{2 n} |X'||Y'|$.
Thus, we may assume $|X'|, |Y'| \ge 3\eps n$. So we have $d_G (x , Y') \ge |Y'| - \eps n \ge 2|Y'|/3$ for every $x \in X'$, thus $e_G(X',Y') \ge 2|X'||Y'|/3$. Thus $\mathbb{E}e_{G_p}(X',Y') \ge \frac{2C \log n}{3n}|X'||Y'| \geq 6 \eps^2 C n \log n$. By Lemma~\ref{lemma:chernoff}, with probability at least $1 - \exp(-n \log^{1/2} n)$, $e_{G_p}(X',Y') > \frac{C \log n}{2n }|X'||Y'|$ as desired. By a union bound, with probability at least
\begin{equation*}
    1 - 2\sum_{x=1}^{3\eps n} n^x (1+n)^{x} e^{-x C \log n / 100} - 4^n e^{-n \log^{1/2} n} \geq 1 - n^{-C^{1/2}}, 
\end{equation*}
the lemma follows.
\end{proof}}

\subsubsection{Equitable proper edge-colouring}
A proper $k$-edge-colouring of a graph $G$ is \emph{equitable} if the set of colour classes is an equitable partition of $E(G)$, i.e., each colour class has either $\lfloor e(G)/k\rfloor$ or $\lceil e(G)/k \rceil$ edges. We will use the following lemma proved in~\cite[Theorem 1]{mcdiarmid1972}.

\begin{lemma}[McDiarmid~\cite{mcdiarmid1972}]\label{lem:equitable}
Let $G$ be a graph, and let $k$ be an integer such that $G$ admits a proper $k$-edge-colouring. Then there exists an equitable proper $k$-edge-colouring of $G$.
\end{lemma}

In particular, by Vizing's theorem~\cite{vizing1965}, for any integer $k$, a graph $G$ with maximum degree at most $k$ admits a proper $(k+1)$-edge-colouring, so it also admits an equitable proper $(k+1)$-edge-colouring by Lemma~\ref{lem:equitable}.



\subsubsection{Pseudorandom matchings and approximate decomposition of typical graphs with random triangles}

\begin{lemma}[Pseudorandom matchings]\label{prop:pseudomat}
Let $1/n_0 \ll 1/C \ll \delta \ll \gamma \ll 1$, let $n \ge n_0$, and let $D \geq C \log n$.  Let $\cH$ be an $n$-vertex $3$-uniform linear hypergraph such that $d_\cH(v) = (1 \pm \delta)D$ for every $v \in V(\cH)$. Let $\cF \subseteq 2^{V(\cH)}$ be a family of subsets of $V(\cH)$ where $|\cF| \leq n^{100}$ and every set $S \in \cF$ satisfies $|S| \geq n^{1/3}$. 
Then there exists a matching $M$ of $\cH$ such that $4\gamma|S|/5 \leq |S \setminus V(M)| \leq \gamma |S|$ for every $S \in \cF$.
\end{lemma}

Related results were proved by Alon and Yuster~\cite{AY05} and Ehard, Glock, and Joos~\cite{EGJ20}, but they do not work for our purposes as they require a more restrictive assumption on the value of $D$.  
The proof of Lemma~\ref{prop:pseudomat} proceeds by the following lemma which is based on the R\"odl nibble (see e.g. \cite{AS16, rodl1985}). 
We include the proofs for completeness.

\begin{lemma}[Nibble lemma]\label{lem:nibble}
Let $1/n_0 \ll 1/C \ll \delta \ll \eps \ll 1$, let $n \ge n_0$, and let $D \geq C \log n$. Let $\cH$ be an $n$-vertex 3-uniform linear hypergraph such that $d_\cH(v) = (1 \pm \delta)D$ for every $v \in V(\cH)$. Let $\cF \subseteq 2^{V(\cH)}$ be a family of subsets of $V(\cH)$ where $|\cF| \leq n^{200}$ and each set $S \in \cF$ satisfies $|S| \geq n^{1/4}$. 
Then there exist a subset $B \subseteq \cH$ and a matching $M \subseteq B$ satisfying the following.
\begin{enumerate}[(\ref*{lem:nibble}:a), topsep = 6pt]
    \item \label{nibble:SminusB} $|S \setminus V(B)| = (1 \pm 5\delta) e^{-\eps} |S|$ for each $S \in \cF$, 
    \item \label{nibble:SintM} $|S \cap V(M)| = (1 \pm 5\delta) \eps e^{-3\eps} |S|$ for each $S \in \cF$, and
    \item \label{nibble:degHminudB} $d_{\cH - V(B)}(v) = (1 \pm 5\delta) e^{-2\eps} D$ for each $v \in V(\cH) \setminus V(B)$.
\end{enumerate}
\end{lemma}

\begin{proof}[Proof of Lemma~\ref{prop:pseudomat} using Lemma~\ref{lem:nibble}]
Let $\eps$ satisfy $\eps \ll \gamma$, and choose $t$ close to $2\eps^{-1} \log (\eps^{-1})$ so that $e^{-\eps t/2} = (1 \pm 0.01) \eps$. 
Letting $\delta$ small enough, we have
\begin{align}\label{eqn:delta}
    1 - \eps < (1 - 5^t \delta)^t \leq \prod_{i=1}^{t} (1 - 5^i \delta) \leq \prod_{i=1}^{t} (1 + 5^i \delta) \leq (1 + 5^t \delta)^t < 1 + \eps.
\end{align}

Let $\cF$ be a collection of subsets of $V(\cH)$ such that $|S| \geq n^{1/3}$ for each $S \in \cF$ and $|\cF| \leq n^{100}$. 
To that end, applying Lemma~\ref{lem:nibble} repeatedly, we will define subhypergraphs $\cH_1 , \dots , \cH_t$ of $\cH$ inductively. 
For $0 \leq k \leq t$, let $\delta_k \coloneqq 5^k \delta$, let $\cH_0 \coloneqq \cH$, and let $\cF_0 \coloneqq \{ V(\cH) \} \cup \cF$. 
For each $S \in \cF_0$, let $S_k \coloneqq S \cap V(\cH_k)$, and let $\cF_k \coloneqq \{ S_k : S \in \cF_0 \}$. 

For $i \in [t]$, if $|S_{i-1}| \geq |V(\cH_{i-1})|^{1/4}$ for each $S_{i-1} \in \cF_{i-1}$, then we can apply Lemma~\ref{lem:nibble} where $\cH_{i-1}$, $\cF_{i-1}$, and $\delta_{i-1}$ playing the roles of $\cH$, $\cF$, and $\delta$, and obtain a subset $B_i \subseteq \cH_{i-1}$ and a matching $M_i$ of $\cH_{i-1}$.
Let $\cH_i \coloneqq \cH_{i-1} - V(B_i)$, thus $S_i = S \cap V(\cH_i) = S_{i-1} \setminus V(B_i)$ for each $S \in \cF_0$. Consequently, we have $|V(\cH_i)| = (1 \pm \delta_i) e^{-\eps} |V(\cH_{i-1})|$, $|S_i| = (1 \pm \delta_i) e^{-\eps} |S_{i-1}|$, and $|S_{i-1} \cap V(M_i)| = (1 \pm \delta_i) \eps e^{-3 \eps} |S_{i-1}|$ for each $S \in \cF_0$. 

We now establish the bounds on $|S_i|$ and $|S_{i-1} \cap V(M_i)|$ for $S \in \cF_0$ and $i \in [t]$, which also shows that $|S_i| \geq |V(\cH_i)|^{1/4}$ in order to apply Lemma~\ref{lem:nibble}.
By the induction,
\begin{align}\label{eqn:size_si}
    |S_i| = (1 \pm \delta_i)e^{-\eps}|S_{i-1}| = \prod_{k=1}^{i} ((1 \pm 5^k \delta) e^{-\eps}) |S| \overset{\eqref{eqn:delta}}{=} (1 \pm \eps) e^{-\eps i} |S|,   
\end{align}
thus $|S_t| = (1 \pm \eps) e^{-\eps t} |S| = (1 \pm 0.05) \eps^2 |S|$, since $e^{-\eps t/2} = (1 \pm 0.01)\eps$ and $\eps$ is sufficiently small.
Since $V(\cH) \in \cF_0$, we have $|V(\cH_i)| \geq |V(\cH_t)| = (1 \pm 0.05) \eps^2 |V(\cH)|$. Thus,
\begin{align*}
    |S_i| \geq |S_t| = (1 \pm 0.05) \eps^2 |S| \geq 0.95 \eps^2 n^{1/3} > n^{1/4} \geq |V(\cH_i)|^{1/4},
\end{align*}
as desired. Moreover,
\begin{align*}
    |S_{i-1} \cap V(M_i)| &= (1 \pm \delta_i)\eps e^{-3\eps}|S_{i-1}| \overset{\eqref{eqn:size_si}}{=} (1 \pm 5^i \delta)\eps e^{-3\eps} \cdot (1 \pm \eps) e^{-\eps (i-1)} |S| \\ 
    &\overset{\eqref{eqn:delta}}{=} (1 \pm 3\eps)\eps e^{-3\eps} e^{-\eps (i-1)} |S|.
\end{align*}
Let $M^* \coloneqq \bigcup_{i=1}^{t} M_i$, which is a matching in $\cH$. Then
\begin{align*}
    |S \setminus V(M^*)| &= |S| - \sum_{i=1}^{t} |S_{i-1} \cap V(M_i)| = |S| -  \sum_{i=1}^{t} (1 \pm 3 \eps) \eps e^{-3 \eps} e^{-\eps (i-1)} |S|\\ 
    &\leq |S| - (1 - 3\eps)\eps e^{-3 \eps} \frac{1 - e^{-\eps t}}{1 - e^{-\eps}}|S| \\
    &\leq |S| - (1 - 3\eps)(1 - 4\eps) (1 - \eps)|S| < 50\eps|S|,
\end{align*}
since $\frac{\eps e^{-3 \eps}}{1 - e^{-\eps}} = 1 - 3\eps + O(\eps^2)$ and $e^{-\eps t} < \eps$ by the definition of $t$. 
Let $M$ be a matching obtained from $M^*$ by deleting each edge of $M^*$ with probability $\gamma$ independently at random. For each $S \in \cF$,
$\mathbb{E} |S \setminus V(M)| = |S \setminus V(M^*)| + \gamma |S \cap V(M^*)| = (\gamma \pm 50 \eps)|S|$.
Thus, by the Chernoff's bound (Lemma~\ref{lemma:chernoff}) and the union bound, we have $|S \setminus V(M)| = (\gamma \pm \eps^{1/2})|S|$ for all $S \in \cF$ with probability at least 0.99.
\end{proof}

The proof of Lemma~\ref{lem:nibble} is a modification of the proof of~\cite[Lemma 4.7.2]{AS16} which uses the second moment method to estimate probabilities, while our approach uses a Martingale inequality (see Lemma~\ref{lem:martingale}) from~\cite{AKS1997}. This allows us to obtain better estimates when the degree is a function in $n$.

Suppose our random variable $Z$ is determined by asking mutually independent ``Yes/No" queries. The choice of the next query can depend on the answers to the previous queries. Thus a \emph{strategy} for determining $Z$ can be represented in the form of a decision tree. \emph{A line of questioning} is a path from the root to a leaf of this tree that determines $Z$. Suppose the $i$th query in this line of questioning has probability $p_i$ to be answered ``Yes". Then the \emph{variance} of the $i$th query is $p_i(1-p_i)\zeta_i^2 \leq p_i\zeta_i^2$, where $\zeta_i$ is the \emph{effect} of the $i$th query, i.e., changing the answer of the $i$th query (but keeping all the answers of other queries the same) will affect the random variable $Z$ by at most $\zeta_i$.  Let $\zeta \coloneqq \max_{i} \zeta_i$. The \emph{total variance} of a line of questioning is the sum of variances of the queries in it, i.e $\sum_{i} p_i (1 - p_i)\zeta_i^2 \leq \sum_{i} p_i \zeta_i^2$.

\begin{lemma}[Martingale Inequality~\cite{AKS1997}]\label{lem:martingale}
There exists an absolute constant $\eta > 0$ such that if there is a strategy for determining $Z$ for which the total variance of every line of questioning is at most $\sigma^2$, then
\begin{equation*}
    \mathbb{P}(|Z - \mathbb{E}Z| > \lambda \sigma) \leq 2 \exp(-\lambda^2 / 4)
\end{equation*}
for any $\lambda \in [0, 2 \sigma \eta / \zeta]$.
\end{lemma}

Now we are ready to prove Lemma~\ref{lem:nibble}.
\begin{proof}[Proof of Lemma~\ref{lem:nibble}]
Let $B$ be a random subset of $\cH$ such that each edge $e \in \cH$ is chosen to be in $B$ with probability $\eps/D$ independently at random, and $M \subseteq B$ be the set of isolated edges in $B$. Then for every $v \in V(\cH)$ and $e \in \cH$,
\begin{align*}
    \mathbb{P}(e \in M) &= \frac{\eps}{D} \left ( 1 - \frac{\eps}{D} \right )^{3D(1 \pm \delta)} = (1 \pm 2\delta) \frac{\eps e^{-3\eps}}{D},\\
    \mathbb{P}(v \in V(M)) &= (1 \pm \delta)D \cdot (1 \pm 2\delta) \frac{\eps e^{-3\eps}}{D} = (1 \pm 4\delta) \eps e^{-3\eps},\\
    \mathbb{P}(v \notin V(B)) &= \left ( 1 - \frac{\eps}{D} \right )^{d_\cH(v)} = (1 \pm 2\delta) e^{-\eps}.
\end{align*}

Let $\cH' \coloneqq \cH - V(B)$. For each $v \in V(\cH)$, let $d_{\cH'}^* (v)$ be the number of edges $e \in \cH$ incident to $v$ such that $e \setminus \{ v \} \subseteq V(\cH') = V(\cH) \setminus V(B)$. Then it is clear that $d_{\cH'}(v) = d_{\cH'}^*(v)$ for each $v \in V(\cH')$. 
For every $v \in V(\cH)$ and $S \in \cF$,
\begin{align*}
    \mathbb{E}d^*_{\cH'}(v) &= (1 \pm \delta)D \cdot \left ( 1 - \frac{\eps}{D} \right )^{2D(1 \pm \delta)} = (1 \pm 4\delta) e^{-2\eps} D,\\
    \mathbb{E}|S \cap V(M)| &= (1 \pm 4\delta) \eps e^{-3 \eps} |S|,\\
    \mathbb{E}|S \setminus V(B)| &= (1 \pm 2\delta) e^{-\eps}|S|.
\end{align*}

Both $|S \cap V(M)|$ and $|S \setminus V(B)|$ are determined by some queries of whether an edge is in $B$. First, we ask whether every edge $e$ incident to a vertex in $S$ is in $B$, and we call such queries as \emph{Type 1} queries. Type 1 queries will determine $|S \setminus V(B)|$, and there are at most $2|S|D$ Type 1 queries.

In order to determine $|S \cap V(M)|$, we need to ask some more queries. After asking Type~1 queries for every vertex $w \in S$, let us call an edge $e_w$ incident to $w$ \emph{solitary} if it is the only edge incident to $w$ which is in $B$. Note that an edge can be in $M$ only if it is solitary. To see which solitary edges are in $M$, we ask whether every edge intersecting each solitary edge is in $B$, and we call such queries as \emph{Type 2} queries. Since there are at most $|S|$ solitary edges, the number of Type 2 queries is at most $2|S|(1 \pm \delta)D \leq 3|S|D$.

Thus, both $|S \setminus V(B)|$ and $|S \cap V(M)|$ are determined by asking at most $5|S|D$ queries. For each edge $e$, changing whether $e \in B$ affects $|S \cap V(M)|$ by at most nine since $e$ can only intersect at most three pairwise disjoint edges, and also affects $|S \setminus V(B)|$ by at most three. Thus, the variance $\sigma$ of each query is at most $\eps/D \cdot 9^2$, and the total variance of the queries will be $\sigma^2 \coloneqq 5|S|D \cdot \eps/D \cdot 9^2 \leq 405 \eps |S|$, thus $\sigma \leq 21 \sqrt{\eps |S|}$. By the Martingale inequality (Lemma~\ref{lem:martingale}), since $|S| \geq n^{1/4}$, we have $|S \cap V(M)| = (1 \pm 5\delta) \eps e^{-3\eps} |S|$ and $|S \setminus V(B)| = (1 \pm 3\delta)e^{-\eps}|S|$ with probability at least $1 - e^{-n^{1/100}}$, for every $S \in \cF$. This proves ~\ref{nibble:SminusB} and ~\ref{nibble:SintM}.

For each $v \in V(\cH)$, let $E(v)$ be the set of edges in $\cH$ incident to $v$. Then $d_{\cH'}^*(v)$ is determined by the events that $e \setminus \{v \} \subseteq V(\cH) \setminus V(B)$ for each $e \in E(v)$, where each of them is determined by asking at most $2(1 \pm \delta)D \leq 3D$ queries of whether an edge is in $B$. In total, $d_{\cH'}^*(v)$ is determined by at most $d_\cH(v) \cdot 3D \leq 6D^2$ queries of whether an edge is in $B$. If we change the result of one of the queries, then it affects $d_{\cH'}^*(v)$ by at most three, since each edge can intersect at most three edges in $E(v)$ by the linearity of $\cH$. Again by the Martingale inequality (Lemma~\ref{lem:martingale}),\COMMENT{$\sigma^2 = 6 D^2 \frac{\eps}{D} 9 = 54 \eps D$, so $\sigma \le 8 \sqrt{\eps D}$. So we can apply the Martingale inequality with $\lambda = \delta \eta \sqrt{\eps D}/16$.} with probability at least $1 - e^{-\eps \eta^2 \delta^2 D / 128} \geq 1 - n^{-C^{2/3}}$,
\begin{equation}\label{eqn:deg_concentration}
d_{\cH'}^*(v) = \mathbb{E}d_{\cH'}^*(v) \pm \eps \delta \eta D / 2 = (1 \pm 5\delta) e^{-2 \eps} D
\end{equation}
where $\eta > 0$ is some absolute constant. Thus by a union bound, with probability at least $1 - n^{-C^{1/2}}$, ~\eqref{eqn:deg_concentration} holds for every $v \in V(\cH)$. This proves~\ref{nibble:degHminudB}, as desired. 
\end{proof}


An $n$-vertex graph $G$ is $(\eta , k , p)$-\emph{typical} if $|\bigcap_{s \in S} N_G(s)| = (1 \pm \eta)p^{|S|}n$ for every set $S \subseteq V(G)$ with $|S| \leq k$.
A collection $\cF$ of triangles in a graph $G$ is $(\eta , p)$-\emph{regular} if every edge of $G$ is in $(1 \pm \eta)p^2 n$ triangles in $\cF$. 
The following lemma is from \cite[Lemma 4.2]{BGKLMO20}, which in turn is based on an idea from \cite{barber2017fractional}.

\begin{lemma}\label{lem:regulark3}
Let $1/n_0 \ll \eta , p$, let $n \ge n_0$, and let $\eta \leq p^7 / 20$. Every $(\eta , 4 , p)$-typical $n$-vertex graph $G$ contains a $(n^{-1/3} , p/2)$-regular collection $\cT$ of triangles.
\end{lemma}

Now Lemma~\ref{lem:regulark3} together with Lemma~\ref{prop:pseudomat} implies the following corollary, which has the advantage that the degree parameter $\gamma$ of the leftover is independent of the typicality parameter $\eta$ of $G$.

\begin{corollary}\label{cor:almostdecomp}
Let $1/n_0 \ll 1/C \ll \gamma , \eta \ll p \leq 1$, and let $n \ge n_0$. For every $(\eta , 4 , p)$-typical $n$-vertex graph $G$, with probability at least $1 - n^{-C^{1/2}}$, there exists $\cT \subseteq \randt(G ,  C \log n/n)$ such that
\begin{itemize}
    \item all triangles in $\cT$ are edge-disjoint, and 
    \item for every vertex $v \in V(G)$, $3\gamma pn / 4 \leq d_{G'}(v) \leq \gamma pn$, where $G' \coloneqq G -  \bigcup_{K \in \cT} E(K)$.
\end{itemize}
\end{corollary}
\COMMENT{\begin{proof}Let us consider another parameter $\delta$ such that $1/C \ll \delta \ll \gamma$.
Applying Lemma~\ref{lem:regulark3} to $G$, we obtain a $(n^{-1/3} , p/2)$-regular collection $\cT_0$ of triangles.
Now let us consider a random 3-uniform hypergraph $\cH_{\rm aux}$ such that $V(\cH_{\rm aux}) = E(G)$ and $E(\cH_{\rm aux}) = \cT_0 \cap \randt(G , C \log n / n)$. 
For each $e \in E(G)$, since $\mathbb{E}d_{\cH_{\rm aux}}(e) = (1 \pm n^{-1/3})(p/2)^2 C \log n$, by Lemma~\ref{lemma:chernoff}, with probability at least $1 - n^{-C^{1/2}}$, we have $d_{\cH_{\rm aux}}(e) = (1 \pm \delta) (p/2)^2 C \log n$, so $\cH_{\rm aux}$ is almost regular.
For each vertex $v \in V(G)$, let $E_v$ be the set of edges incident to $v$ in $G$, and $\cF \coloneqq \{ E_v \: : \: v \in V(G) \}$. Note that since $G$ is $(\eta, 4, p)$-typical, for any $v \in V(G)$, $|E_v| = (1 \pm \eta)pn$. Moreover, $|V(\cH _{\rm aux})| = (1 \pm \eta)pn^2 / 2$.
Thus we can apply Lemma~\ref{prop:pseudomat} to $\cH_{\rm aux}$ and $\cF \subseteq 2^{V(\cH_{\rm aux})}$ with $99 \gamma/100$ playing the role of $\gamma$ to obtain the desired set $\cT \subseteq \cT_1$ of edge-disjoint triangles in $\cH_{\rm aux}$. 
\end{proof}}

\subsection{Covering down edges inside the parts}\label{sec:sec3}
We need the following ``randomized version" of~\cite[Lemma 3.8]{BGKLMO20} to prove Theorem~\ref{thm:red1}. More precisely, we use it to cover all of the edges inside each part (of the equitable partition of $[n]$) which do not lie inside a small subset of vertices.

\begin{lemma}\label{lem:coverdown}
Let $1/n_0 \ll 1/C \ll \eps_1 \ll \eps_2 \ll \eps_3 \ll 1$, and let $n \ge n_0$. Let $G$ be an $n$-vertex graph, let $U \subseteq V(G)$ be a subset with $|U| = \lfloor \eps_3 n \rfloor$, and suppose that $\delta(G) \geq (1-\eps_1)n$ holds and $d_G(v)$ is even for each $v \in V(G) \setminus U$. 
With probability at least $1 - n^{-C^{1/6}}$, there exists a set $\cT \subseteq \randt(G , C \log n/n)$ of edge-disjoint triangles such that the subgraph $G^* \coloneqq G - \bigcup_{K \in \cT} E(K)$ satisfies the following.
\begin{enumerate}[(\ref*{lem:coverdown}:a), topsep = 6pt]
    \item \label{emptyVminusUcdown} For every vertex $v \in V(G) \setminus U$, we have $d_{G^*}(v) = 0$.
    \item \label{denseU} For every vertex $v \in U$, we have $d_{G^*}(v) \geq |U| - 2 \eps_2 n$.
\end{enumerate}
\end{lemma}

The proof of Lemma~\ref{lem:coverdown} is similar to the proof of~\cite[Lemma 3.8]{BGKLMO20}.
The main ingredient of the proof of~\cite[Lemma 3.8]{BGKLMO20} is~\cite[Lemma 3.10]{BGKLMO20}. Here we need a randomized version of~\cite[Lemma 3.10]{BGKLMO20}, Lemma~\ref{lem:coverdown_reservoir}, whose proof is very similar to the proof of~\cite[Lemma 3.10]{BGKLMO20} -- the main difference is that we apply Lemma~\ref{lem:manymatchings} instead of finding edge-disjoint matchings greedily. 

\begin{proof}[Proof of Lemma~\ref{lem:coverdown}]
For each $w \in V(G) \setminus U$, let $U_w \subseteq N_G(w) \cap U$ be a set obtained by choosing each $v \in N_G(w) \cap U$ independently at random with probability $\eps_2$. By Lemma~\ref{lemma:chernoff}, there exists a family $\{ U_w \}_{w \in V(G) \setminus U}$ of subsets such that
\begin{itemize}
    \item[(i)] $3 \eps_2 \eps_3 n / 4 \leq |U_w| \leq 5 \eps_2 \eps_3 n / 4$ for each $w \in V(G) \setminus U$, 
    \item[(ii)] $3 \eps_2^2 \eps_3 n / 4 \leq |U_{w_1} \cap U_{w_2}| \leq 5 \eps_2^2 \eps_3 n/ 4$ for distinct $w_1 , w_2 \in V(G) \setminus U$, and
    \item[(iii)] each vertex $u \in U$ is contained in at most $\eps_2 n$ sets in $\{ U_w \}_{w \in V(G) \setminus U}$.\COMMENT{$u$ is adjacent to $\eps_2|V(G) \setminus U| = \eps_2(1-\eps_3)n$ in expectation, so Lemma~\ref{lemma:chernoff} gives an upper bound of $\eps_2n$.}
\end{itemize}

Let $R \coloneqq \{wv : w \in V(G) \setminus U \text{ and } v \in U_w\}$ be a reservoir of edges.
For every $i \in [3]$, let $\cS^i$ be the set of triangles in $G^{(3)}$ which contain exactly $i$ vertices in $V(G) \setminus U$. Then $\cS^1, \cS^2, \cS^3$ are disjoint subsets of $G^{(3)}$. 
We now define subsets of $\cS^1$ and $\cS^2$ as follows. For every vertex $v \in V(G) \setminus U$, let $\cS^1_{v}$ be the set of triangles in $\cS^1$ containing $v$. Then $\{ \cS^1_{v} : v \in V(G) \setminus U \}$ is a partition of $\cS^1$.
Let $\cS^2_{\mathrm{res}} \coloneqq \{ \{u,v,w\}: uv \in E(G[V(G) \setminus U]), w \in U, uw, vw \in R\}$, and let $\cS^2_{\mathrm{nres}}  \coloneqq \{ \{u,v,w\}: uv \in E(G[V(G) \setminus U]), w \in U, uw, vw \in E(G) \setminus R \}$. Clearly, $\cS^2_{\mathrm{res}} \sqcup \cS^2_{\mathrm{nres}} \subseteq \cS^2$. 
Moreover, for every edge $uv \in E(G[V(G) \setminus U])$, let $\cS^2_{uv, \mathrm{res}}$ be the set of triangles in $\cS^2_{\mathrm{res}}$ containing $uv$. Then $ \{ \cS^2_{uv, \mathrm{res}} : uv \in E(G[V(G) \setminus U]) \}$ is a collection of disjoint subsets of $\cS^2_{\mathrm{res}}$.

To construct our desired collection of edge-disjoint triangles $\cT \subseteq \randt(G , \frac{C \log n}{n})$, we first expose the triangles in $\cS^3(\frac{C \log n}{n})\cup \cS^2_{\mathrm{nres}} (\frac{C \log n}{n})$ to cover most of the edges in $G - R- E(G[U])$ using a set $\cT'$ of edge-disjoint triangles, and then we expose the triangles in $\cS^2_{uv, \mathrm{res}} (\frac{C \log n}{n})$ for every leftover edge $uv \in E(G[V(G) \setminus U])$ to cover all remaining edges in $G[V(G) \setminus U]$ using a set $\cT''$ of edge-disjoint triangles, and finally, we expose the triangles in $\cS^1_v (\frac{C \log n}{n})$ for all $v \in V(G) \setminus U$ to cover all remaining edges in $G - G[U]$ using a set $\cT_{\rm rem}$ of edge-disjoint triangles.

Let $G' \coloneqq G - R - E(G[U])$. Then for each $u \in V(G) \setminus U$, we have $d_{G'}(u) = d_G (u) - |U_u| \geq (1 - \eps_1 - \eps_2)n$,\COMMENT{we used (i) here, and that $|U_u| \le 5 \eps_2 \eps_3 n/4 \le \eps_2 n$} and for each $v \in U$, we have\COMMENT{For the last inequality below, namely $d_G(v) - |U| - |\{ w \in V(G) \setminus U \: : \: v \in U_w \}| 
\geq (1 - \eps_1 - \eps_2 - \eps_3)n$, we used $d_G(v) \ge (1-\eps_1)n$, $|\{ w \in V(G) \setminus U \: : \: v \in U_w \}| \le \eps_2n$ by (iii) and that $|U| \le \eps_3 n$.} 
\begin{equation*}
    d_{G'}(v) \geq d_G(v) - |U| - |\{ w \in V(G) \setminus U \: : \: v \in U_w \}| \overset{\mathrm{(iii)}}{\geq} (1 - \eps_1 - \eps_3 - \eps_2)n.
\end{equation*}

Thus, exposing triangles in $\cS^3(\frac{C \log n}{n})\cup \cS^2_{\mathrm{nres}} (\frac{C \log n}{n})$, by applying Corollary~\ref{cor:almostdecomp} to $G'$ with $\eps_1$, $1$ and $4(\eps_1 + \eps_2 + \eps_3)$ playing the roles of $\gamma$, $p$ and $\eta$, respectively, with probability at least $1 - n^{-C^{1/2}}$, there exists $\cT' \subseteq {G'}^{(3)} \cap ( \cS^3(\frac{C \log n}{n})\cup \cS^2_{\mathrm{nres}} (\frac{C \log n}{n}))$ such that for $G'' \coloneqq G' - \bigcup_{K \in \cT'} E(K)$,
\begin{itemize}
    \item[(a)] all triangles in $\cT'$ are edge-disjoint, and 
    \item[(b)] $\eps_1 n / 2 \leq \delta(G'') \leq \Delta(G'') \leq \eps_1 n$.
\end{itemize}

Now we will expose triangles in $\cS^2_{\mathrm{res}} (\frac{C \log n}{n})$ and show that with probability at least $1 - n^{-C^{1/2}}$ there exists a map $\psi : E(G''[V(G) \setminus U]) \to U$ such that for each $vw \in E(G''[V(G) \setminus U])$, $\psi(vw) \in U_v \cap U_w$, $\{v,w,\psi(vw)\} \in \cS^2_{\mathrm{res}}(\frac{C \log n}{n})$, and $\psi(e) \ne \psi(e')$ if $e \cap e' \ne \varnothing$.

To that end, we consider edges $vw \in E(G''[V(G) \setminus U])$, one by one, and construct $\psi$ by assigning $\psi(vw)$ to an element of $U_v \cap U_w$ greedily as follows. 
Note that there are at most $2\Delta(G'') \leq 2 \eps_1 n$ edges which intersect $vw$, so there are at most $2 \eps_1 n$ vertices in $U_v \cap U_w$ already assigned to these edges. 
Since $|U_v \cap U_w| \geq 3\eps_2^2 \eps_3 n / 4$ by (ii), we have at least $3 \eps_2^2 \eps_3 n / 4 - 2\eps_1 n > \eps_2^2 \eps_3 n / 2$ candidates $u \in U_v \cap U_w$ which can be assigned to $\psi(vw)$, and for a given $vw \in E(G''[V(G) \setminus U])$ the probability that all triangles $\{v,w,u\}$ where $u$ is a candidate for $\psi(vw)$ are not in $\cS^2_{vw, \mathrm{res}}(\frac{C \log n}{n})$ is
$(1 - \frac{C \log n}{n})^{\eps_2^2 \eps_3 n / 2} \leq n^{- C \eps_2^2 \eps_3 / 2}$,
so by a union bound, with probability at least $1 - |E(G'')| n^{-C \eps_2^2 \eps_3 / 2} \geq 1 - n^{-C^{1/2}}$, the desired map $\psi$ exists. Note that $\cT'' \coloneqq \{ \{ v,w,\psi(vw) \} \: : \: vw \in E(G''[V(G) \setminus U])\}$ is a collection of edge-disjoint triangles. Consider
\begin{equation*}
    G_{\rm rem} \coloneqq (G'' \cup R) - \bigcup_{K \in \cT''} E(K).
\end{equation*}

Note that all edges in $G_{\rm rem}$ have one endpoint in $V(G) \setminus U$ and the other endpoint in $U$.\COMMENT{$G_{\rm rem}$ consists of leftover among crossing edges and the remaining edges of the reservoir $R$} Let $V_w \coloneqq N_{G_{\rm rem}}(w)$ for each $w \in V(G) \setminus U$. Since $w$ has even degree in $G$ and $G_{\rm rem}$ can be obtained from $G$ by deleting edge-disjoint triangles, the vertex $w$ also has even degree in $G_{\rm rem}$, so $|V_w|$ is even.
Note that $|V_w| \geq |U_w| - d_{G''}(w) \geq \eps_2 \eps_3 n / 2 \geq \eps_2 |U|/4$ by (i) and (b). Moreover, for distinct $w_1, w_2 \in V(G) \setminus U$, $|V_{w_1} \cap V_{w_2}| \leq |U_{w_1} \cap U_{w_2}| + d_{G''}(w_1) + d_{G''}(w_2) \leq 3 \eps_2^2 \eps_3 n / 2 \leq 3 \eps_2^2 |U|$ by (ii) and (b).
For every $w \in V(G) \setminus U$, we partition $V_w$ into two equal-sized sets $V_w'$ and $V_w''$.
Let $G_w$ be the random bipartite graph with bipartition $\{V_w',V_w''\}$ such that for every $u \in V_w'$ and $v \in V_w''$, $uv \in E(G_{w})$ if and only if $\{u,v,w \} \in \cS^1_w(\frac{C \log n}{n})$.

For every $u \in U$, the vertex $u$ lies in at most $2\eps_2 \eps_3^{-1} |U|$ sets in $\{ V_w \}_{w \in V(G) \setminus U}$, since
\begin{equation}\label{eqn:deg_u_rem}
    d_{G_{\rm rem}}(u) \leq d_{R}(u) + d_{G''}(u) \overset{\mathrm{(iii)}}{\leq} \eps_2 n + d_{G''}(u) \overset{(b)}{\leq} (\eps_1 + \eps_2)n \leq 2 \eps_2 \eps_3^{-1} |U|.
\end{equation}
Applying Lemma~\ref{lem:coverdown_reservoir} with $|U|$, $\eps_2 \eps_3^{-1}$, $|V(G) \setminus U|$, and $\{ G_{\rm rem}[V_w' , V_w''] \}_{w \in V(G)\setminus U}$ playing the roles of $n$, $\rho$, $m$, and $\{H_i\}_{i \in [m]}$\COMMENT{Note that $\eps_2 \gg (\eps_2 \eps_3^{-1})^{4/3}$ and $\eps_2^2 \ll (\eps_2\eps_3^{-1})^2$ so for $w \in V(G) \setminus U$,
$|V_w| \ge \eps_2 |U|/4 \gg (\eps_2 \eps_3^{-1})^{4/3} |U|$ and for distinct $w_1, w_2 \in V(G) \setminus U$, $|V_{w_1} \cap V_{w_2}| \leq 3 \eps_2^2 |U| \ll (\eps_2\eps_3^{-1})^2 |U|$. Moreover, every $u \in U$ lies in at most $2 (\eps_2 \eps_3^{-1}) |U|$ sets in $\{ V_w \}_{w \in V(G) \setminus U}$. Thus the conditions of Lemma~\ref{lem:coverdown_reservoir} are satisfied for $\rho = \eps_2\eps_3^{-1}$.}, with probability at least $1 - n^{-C^{1/5}}$, for each $w \in V(G) \setminus U$, we obtain a matching $M_w$ of $G_w$ such that $\{ M_w \}_{w \in V(G) \setminus U}$ is a set of edge-disjoint matchings.
Let
$\cT_{\rm rem} \coloneqq \{ \{u,v,w \} \: : \: w \in V(G) \setminus U, \: uv \in M_w \}$.
Since the matchings $M_w$ for $w \in V(G) \setminus U$ are edge-disjoint, all triangles in $\cT_{\rm rem}$ are edge-disjoint, and $E(G_{\rm rem}) \subseteq \bigcup_{K \in \cT_{\rm rem}} E(K)$. 
Moreover, by~\eqref{eqn:deg_u_rem}, every vertex $u \in U$ lies in at most $2 \eps_2 n$ sets in $\{ V_w \}_{w \in V(G) \setminus U}$, so $u$ lies in at most $2 \eps_2 n$ triangles in $\cT_{\rm rem}$. 
Thus, if $\cT \coloneqq \cT' \cup \cT'' \cup \cT_{\rm rem}$ then $G^* = G - \bigcup_{K \in \cT} T(K)$ satisfies ~\ref{denseU}. Note that the triangles in $\cT$ cover all edges of $G - G[U]$, so $G^*$ satisfies ~\ref{emptyVminusUcdown}. Moreover, $\cT$ exists in $G$ with probability at least $1 - (n^{-C^{1/2}} + n^{-C^{1/2}} + n^{-C^{1/5}}) \geq 1 - n^{-C^{1/6}}$, completing the proof of Lemma~\ref{lem:coverdown}.
\end{proof}

\subsection{Proof of Theorem~\ref{thm:red1}}\label{subsec:main}


\begin{proof}[Proof of Theorem~\ref{thm:red1}]

Choose a new constant $\eps_2$ such that $1/C \ll \eps_2 \ll \eps$.

\setcounter{pf-step}{0}


We first define several subsets of $V(G)$ and disjoint subsets $\cS^1, \cS^2, \cS^3, \cS^4$ of $G^{(3)}$. 
For every $i \in [4]$, in Step $i$, our plan is to expose triangles in $\cS^i (p)$ and find a set $\cT^i \subseteq \cS^i(p)$ with probability at least $1 - n^{-C^{1/5}}$ so that, in the end, $\cT \coloneqq \bigcup_{i=1}^{4}\cT^i$ will be the desired set of edge-disjoint triangles.

If $n = 6t+1$ for some integer $t$, then since $\{V_1 , V_2 , V_3 \}$ is an equitable partition of $V(G)$ we have $|V_1| = |V_2| = 2t$ and $|V_3| = 2t+1$. In this case, let $\eps_3 n \coloneqq \eps n - 1$\COMMENT{We used $\eps_3$ instead of $\eps_1$, to be consistent with $\eps_3$ in Lemma~\ref{lem:coverdown}.}, so that $W_3 \subseteq V_3$ has size $\lfloor \eps_3 n \rfloor + 1$.
Let $v^* \in W_3$ be any vertex, let $U_3 \coloneqq W_3 \setminus \{ v^* \}$, and let $\cS^1 \coloneqq \{ \{u,v,v^* \} : uv \in E(G[V_1]) \cup E(G[V_2]) \}$.

Otherwise if $n = 6t+3$ for some integer $t$, then since $\{ V_1 , V_2 , V_3 \}$ is an equitable partition of $V(G)$, we have $|V_1| = |V_2| = |V_3| = 2t+1$. In this case, let $\eps_3 \coloneqq \eps$, and let $U_3 = W_3 \subseteq V_3$, so that $|W_3| = |U_3| = \lfloor \eps_3 n \rfloor$. Let $\cS^1 \coloneqq \varnothing$.

For both cases ($n = 6t+1$ or $n = 6t+3$), let $U_1 \subseteq V_1$, $U_2 \subseteq V_2$ be arbitrary subsets of size $\lfloor \eps_3 n \rfloor$. 
Let $\cS^2$ be the set of triangles in $\bigcup_{i=1}^{3}G[V_i]$, and let
\begin{align*}
    \cS^3 &\coloneqq \bigcup_{i=1}^{3} \bigcup_{j \in [3] \setminus \{ i \}} \{ \{u,v,w\} : u,v \in U_j,\:w \in U_i \},\\
    \cS^4 &\coloneqq \{ \{u,v,w\} : u \in V_1 , v \in V_2, w \in U_3 \}.
\end{align*}

Note that $\cS^1, \cS^2, \cS^3, \cS^4$ are disjoint subsets of $G^{(3)}$.

\begin{pf-step}\label{step:step1}
Equalising the vertex class sizes.
\end{pf-step}

If $n = 6t+1$ for some integer $t$, 
exposing triangles in $\cS^1(p)$, 
by Lemma~\ref{lem:manymatchings}, with probability at least $1 - n^{-C^{1/3}}$, there exist perfect matchings $M_1 , M_2$ of $G[V_1]$ and $G[V_2]$ respectively such that for each $uv \in M_1 \cup M_2$, $\{u,v,v^*\} \in \cS^1(p)$.
Let $\cT^1 \coloneqq \{ \{u,v,v^*\} \: : \: uv \in M_1 \cup M_2 \}$ and let $G^1 \coloneqq G - \bigcup_{T \in \cT^1}E(T)$. Then $G^1[V_1]$, $G^1[V_2]$ are $(2t-2)$-regular, and $G^1[V_3]$ is $2t$-regular. Note that since $e(G^1[V_3]) = t(2t+1)$ and $e(G^1[V_1]) = e(G^1[V_2]) = t(2t-2)$, we have
\begin{equation}\label{eqn:same_rem}
    e(G^1[V_1]) \equiv e(G^1[V_2]) \equiv e(G^1[V_3]) \mod 3.
\end{equation}

Otherwise if $n = 6t+3$ for some integer $t$, 
let $\cT^1 \coloneqq \varnothing$, and let $G^1 \coloneqq G$. Then $G^1[V_1]$, $G^1[V_2]$, and $G^1[V_3]$ are $2t$-regular.

\begin{pf-step}\label{step:step2}
Covering most of the edges in $G^1[V_1], G^1[V_2]$ and $G^1[V_3]$ with edge-disjoint triangles.
\end{pf-step}

For each $i \in [3]$, we apply Lemma~\ref{lem:coverdown} to $G^1[V_i]$ with $U_i$ playing the role of $U$,\COMMENT{Note that $G^1[V_i]$ is almost complete (with even degrees), so we can pick $\eps_1$ arbitrarily small in Lemma~\ref{lem:coverdown}.} to obtain a subgraph $H_i'$ of $G^1[V_i]$ (by removing edge-disjoint triangles) such that $d_{H_i'}(v)=0$ for all $v \in V_i \setminus U_i$, and $d_{H_i'}(v) \ge |U_i| - 2\eps_2 n \ge |U_i| - 4 \eps_2 \eps_3^{-1}|U_i|$ for all $v \in U_i$.\COMMENT{actually, we get $|U_i|-2 \eps_2 (n/3)$ from Lemma~\ref{lem:coverdown} but that is at least $|U_i|-2 \eps_2 n$.}

We then apply Corollary~\ref{cor:almostdecomp} to $H_i'[U_i]$ with $1, \eps_2^2$ and $16\eps_2 \eps_3^{-1}$ playing the role of $p, \gamma$ and $\eta$ respectively. 
Thus, exposing triangles in $\cS^2(p)$, with probability at least $1 - n^{-C^{1/5}}$, we obtain a set of edge-disjoint triangles 
$\cT^2 \subseteq \cS^2(p)$ such that $G^2 \coloneqq G^1 - \bigcup_{T \in \cT^2} E(T)$ satisfies the following. 
\begin{enumerate}[(\ref*{step:step2}:a), topsep = 6pt]
    \item \label{degVioutsideUi} For every $i \in [3]$ and $v \in V_i \setminus U_i$, $d_{G^2[V_i]}(v) = 0$.
    
    \item \label{degVisdegU} For every $i \in [3]$ and $u \in U_i$, $d_{G^2[V_i]}(u) = d_{G^2[U_i]}(u)$ is even.
\end{enumerate}

Indeed, since $d_{G^1[V_i]}(u)$ is even for each $i \in [3]$ and $u \in V_i$, and $G^2[V_i]$ is obtained from $G^1[V_i]$ by removing edge-disjoint triangles from $G^1[V_i]$, (\ref{step:step2}:b) holds. 


In addition, after applying Corollary~\ref{cor:almostdecomp} to $H_i'[U_i]$ for each $i \in [3]$, we also obtain $\Delta(G^2[U_i]) \leq 2 \eps_2^2 \eps_3 n$. 
Thus $e(G^2[U_i]) \leq \eps_2^2 \eps_3^2 n^2$, and the number of triangles in $\cT^2$ contained in $U_i$ is $(1 \pm 0.001) \eps_3^2 n^2 / 6$. Since~\eqref{eqn:same_rem} holds and $e(G^2[U_i]) \leq \eps_2^2 \eps_3^2 n^2$, by removing $(1 \pm 3\eps_2) \eps_2 \eps_3^2 n^2 / 6$ such triangles induced by $U_i$ from $\cT^2$ randomly for each $i \in [3]$, with probability at least 0.9, we can ensure that $e(G^2[U_1]) = e(G^2[U_2]) = e(G^2[U_3])$ and they are all even. 
Moreover, since we removed $(1 \pm 0.1) \eps_2$-fraction of the triangles contained in $U_i$ from $\cT^2$ at random, this increases the degree of all vertices in $G^2[U_i]$ by $(2 \pm 1/2) \eps_2 |U_i|$ with probability at least 0.9. Hence, we deduce the following.
    
\begin{enumerate}[label=(\ref*{step:step2}:\alph*), start=3, topsep = 6pt]
    \item \label{minmaxdegUi} For every $i \in [3]$, 
    $\eps_2 \eps_3 n / 3 \leq  \delta(G^2[U_i]) \leq \Delta(G^2[U_i]) \leq 3 \eps_2 \eps_3 n$. 
    In particular, $\eps_2 \eps_3^2 n^2 / 10 \leq e(G^2[U_i]) \leq 3\eps_2 \eps_3^2 n^2 / 2$. 
    
    \item \label{eveninUi} $e(G^2[U_1]) = e(G^2[U_2]) = e(G^2[U_3])$ and it is even.\COMMENT{This is necessary to ensure~\eqref{eqn:ineq_defi} in the next step.}
\end{enumerate}


\begin{pf-step}\label{step:step3}
Covering the remaining edges in $G^2[U_1]$, $G^2[U_2]$ and $G^2[U_3]$.
\end{pf-step}

In this step our objective is to find a set $\cT^3$ of edge-disjoint triangles satisfying the following with probability at least $1 - n^{-C^{1/5}}$ by exposing triangles in $\cS^3(p)$.

\begin{enumerate}[(\ref*{step:step3}:a), topsep = 6pt]
\item Every $T \in \cT^3$ lies in $\cS^3(p)$, and is of the form $\{u,v,w \}$, where $w \in U_i$ and $uv \in E(G^2[U_s])$ for some distinct $i,s \in [3]$. 
We call $w$ the \emph{centre vertex} of the triangle $T$ and call the triangle $T$ an \emph{$s$-type} triangle.

\item \label{coverremedgesUi} Every edge in $\bigcup_{i=1}^{3}E(G^2[U_i])$ is contained in one of the triangles in $\cT^3$, and the subgraph $G^3 \coloneqq G^2 - \bigcup_{T \in \cT^3}E(T)$ satisfies the \emph{divisibility condition}, i.e., for any $i \in [3]$, distinct $j_1, j_2 \in [3] \setminus \{i \}$, and $u \in V_i$,
    $d_{G^3}(u, V_{j_1}) = d_{G^3}(u, V_{j_2})$,
or equivalently, $d_{G^3}(u, U_{j_1}) = d_{G^3}(u, U_{j_2})$ for each $u \in U_i$.
\end{enumerate}

For each $i \in [3]$, in the rest of this step we always denote $j,k \in [3] \setminus \{ i \}$ to be the indices satisfying $j < k$. 
Now consider a random partition of $E(G^2[U_i])$ into two parts $E_{i,j}$ and $E_{i,k}$ of the same size (which is possible since $e(G^2[U_i])$ is even by ~\ref{eveninUi}); for each $uv \in E_{i,j}$ (or $uv \in E_{i,k}$), the edge $uv$ will be covered by a triangle of the form $\{ u,v,w \}$ for some $w \in U_j$ ($w \in U_k$, respectively).

For each $u \in U_i$, let us define $\defi(u)$ as the number of edges in $E_{i,k}$ incident to $u$ minus the number of edges in $E_{i,j}$ incident to $u$. Since $d_{G^2[U_i]}(u)$ is even by ~\ref{degVisdegU}, $\defi(u)$ is also even. Also note that
\begin{equation}\label{eqn:ineq_defi}
    \sum_{u \in U_i} \defi(u) = 2(|E_{i,k}| - |E_{i,j}|) = 0.
\end{equation}

Since we randomly partitioned $E(G^2[U_i])$ into sets $E_{i,j}, E_{i,k}$ of the same size, with nonzero probability we have $|\defi(u)| \leq n^{2/3}$ for every $u \in U_i$.
To ensure (\ref{step:step3}:b), we will construct $\cT^3$ to satisfy
\begin{equation}\label{cond:def}
    t_j(u,\cT^3) = t_k(u,\cT^3) + \frac{\defi(u)}{2}
\end{equation}
for each $u \in U_i$, where $t_j(u,\cT^3)$ (or $t_k(u,\cT^3)$) denotes the number of $j$-type ($k$-type, respectively) triangles in $\cT^3$ with centre vertex $u$. Note that ~\eqref{cond:def} indeed implies ~\ref{coverremedgesUi}.


Let $q \coloneqq \lfloor 10 \eps_2 \eps_3 n \rfloor$. By (\ref{step:step2}:c), Vizing's theorem, and Lemma~\ref{lem:equitable}, there exists an integer $\ell$ such that for all $i \in [3]$ and $s \in [3] \setminus \{i \}$, the edges in $E_{i,s}$ can be properly coloured with $q$ colours such that each colour class has size either $\ell$ or $\ell + 1$;\COMMENT{By (\ref{step:step2}:c) and Vizing's theorem, the chromatic number of the graph formed by the edges in $E_{i,s}$ is at most $3 \eps_2 \eps_3 n + 1$, and this is less than $q$, so Lemma~\ref{lem:equitable} applies.} let $\cC_{i,s}$ be the set of these colour classes.
Note that the average size of colour classes is $|E_{i,s}|/q = e(G^2[U_i]) / (2q)$, so by (\ref{step:step2}:c),
\begin{equation}\label{eqn:size_ell}
    \ell = \frac{e(G^2[U_i])}{2q} \pm 1 \in \left [\frac{\eps_3 n}{250} \: , \: \frac{\eps_3 n}{5} \right].
\end{equation}

By (\ref{step:step2}:d), the number of colour classes of size $\ell$ (size $\ell+1$) in $\cC_{j,i}$ is equal to the number of colour classes of size $\ell$ (size $\ell+1$, respectively) in $\cC_{k,i}$.
Thus, there exists a bijection between the elements of $\cC_{j,i}$ and $\cC_{k,i}$ such that each $M^j \in \cC_{j,i}$ is paired with a unique $N^k \in \cC_{k,i}$ with $|M^j| = |N^k|$.
Let $(M^j_1 , N^k_1) , \dots , (M^j_q, N^k_q)$ be the enumeration of such pairs. 

We will construct $\cT^3$ as the disjoint union of $\cT^3_1, \cT^3_2, \cT^3_3$, where the triangles in $\cT^3_i$ cover all edges in $E_{j,i} \cup E_{k,i} = \bigcup_{r=1}^{q} (M_r^j \cup N_r^k)$. Let us briefly outline our plan for doing so. We sequentially construct $\cT^3_1, \cT^3_2, \cT^3_3$.
In turn, for every $i \in [3]$, we will construct $\cT^3_i$ as the disjoint union of $\cT^3_{i,1} , \dots , \cT^3_{i,q}$, where the triangles in $\cT^3_{i,r}$ cover all edges in $M_r^j \cup N_r^k$ for each $r \in [q]$.
To construct $\cT^3_{i,r}$ for each $r \in [q]$, let $h_r \coloneqq |M_r^j| = |N_r^k| \in \{\ell, \ell+1 \}$, and let
$\cS^3_{i,r} \coloneqq \{ \{u,v,w \} : uv \in M_r^j \cup N_r^k,\:w \in U_i \}$. Note that $\bigsqcup_{i \in [3]} \bigsqcup_{r \in [q]} \cS^3_{i,r} \subseteq \cS^3$.


We will choose sets $C_r, C'_r \subseteq U_i$ of size $h_r$ and  find $\cT^3_{i,r}$ (based on $C_r$ and $C'_r$) satisfying the following property with probability at least $1 - n^{-C^{1/4}}$ by exposing triangles in $\cS_{i,r}^3 (p)$.

\begin{figure}
	\centering
	\includegraphics[scale=0.5]{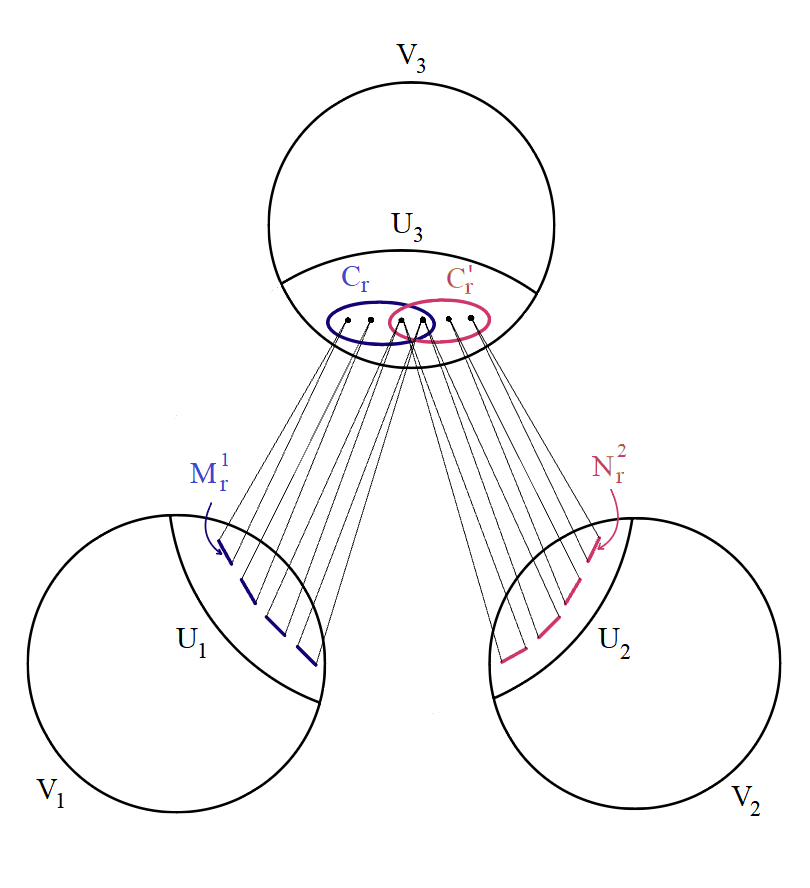}
        \captionsetup{justification=centering}
	\caption{The figure shows how to find $\cT_{3,r}^3$ for $r \in [q]$.}
        \label{fig:step3}
\end{figure}

\begin{enumerate}[label=($\mathrm{P}_r$)]
    \item\label{para:step3} For every $uv \in M_r^j$ there exists a unique $w_{uv} \in C_r$ and for every $u'v' \in N_r^k$, there exists a unique $w'_{u'v'} \in C_r'$, such that the set $\cT^3_{i,r}$ consists of $2h_r$ edge-disjoint triangles of the form
    \begin{equation*}
        \{ \{u,v, w_{uv}\}, \{u',v', w'_{u'v'}\} : uv \in M_r^j,\:u'v' \in N_r^k \} \subseteq \cS^3_{i,r}(p),
    \end{equation*}
    and they are edge-disjoint from the triangles in $\bigcup_{s=1}^{i-1} \cT^3_s \cup \bigcup_{s=1}^{r-1} \cT^3_{i,s}$.
\end{enumerate}

Note that $M_r^j \cup N_r^k \subseteq \bigcup_{T \in \cT^3_{i,r}}E(T)$ for $r \in [q]$, 
so $E_{j,i} \cup E_{k,i} \subseteq \bigcup_{r=1}^{q} \bigcup_{T \in \cT^3_{i,r}} E(T) = \bigcup_{T \in \cT^3_i} E(T)$. Thus $\bigcup_{i=1}^3 E(G^2[U_i]) \subseteq \bigcup_{T \in \cT^3}E(T)$. As discussed later, the choice of sets $C_r, C'_r \subseteq U_i$ for $r \in [q]$ will be made carefully in order to construct $\cT^3$ satisfying the divisibility condition (\ref{step:step3}:b). 
In particular, we need to construct $\cT^3$ so that it satisfies~\eqref{cond:def} for each $u \in U_i$. To that end, for every $r \in \{0\} \cup [q]$, we define a function $w_{i,r} : U_i \to \mathbb{Z}$ which keeps track of how much the degree of a vertex in  $U_i$ needs be compensated for the divisibility condition to be satisfied after constructing $\bigcup_{s=1}^{r}\cT^3_{i,s}$.  More precisely, for every $u \in U_i$ and $r \in [q]$, let $w_{i,0}(u) \coloneqq \defi(u)/2$ and let 
\begin{equation*}
    w_{i,r}(u) \coloneqq \frac{\defi(u)}{2} - t_j (u,\bigcup_{s=1}^{r}\cT^3_{i,s}) + t_k(u,\bigcup_{s=1}^{r}\cT^3_{i,s}),
\end{equation*}
where $t_j(u,\bigcup_{s=1}^{r}\cT^3_{i,s})$ (or $t_k(u,\bigcup_{s=1}^{r}\cT^3_{i,s})$) denotes the number of $j$-type ($k$-type, respectively) triangles in $\bigcup_{s=1}^{r}\cT^3_{i,s}$ with centre vertex $u$.


Note that~\eqref{cond:def} holds if and only if $w_{i,q}(u) = 0$ for each $u \in U_i$.
Let $\norm{w_{i,r}} \coloneqq \sum_{u \in U_i} |w_{i,r}(u)|$ for $r \in \{0\} \cup [q]$. Note that $\norm{w_{i,0}} \leq n^{5/3}/2$. We will inductively construct $\mathcal T_i^3 = \bigsqcup_{r=1}^q \cT^3_{i,r}$ while making sure that for every $0 \leq r \leq q$, we have \begin{equation}\label{eqn:ineq_ws}
    \sum_{u \in U_i}w_{i,r}(u) = 0\:\:\:\text{and}\:\:\:\norm{w_{i,r}} \leq \norm{w_{i,r-1}},
\end{equation}
where $\norm{w_{i,-1}} \coloneqq n^{5/3}/2$. Observe that \eqref{eqn:ineq_ws} holds for $r=0$ by~\eqref{eqn:ineq_defi}.

Fix $s \in [q]$.
Suppose we have already constructed $\cT^3_r$ for $r \in [i-1]$ and $\cT^3_{i,r}$ for $r \in [s-1]$ satisfying \ref{para:step3}, and suppose the functions $w_{i,r} : U_i \to \mathbb{Z}$ satisfy~\eqref{eqn:ineq_ws} for $r \in [s-1]$. We will now construct $\cT^3_{i,s}$
satisfying~\ref{para:step3} for $r = s$, and
show that~\eqref{eqn:ineq_ws} holds for $r = s$.

Assuming that the sets $C_s,C'_s \subseteq U_i$ of size $h_s$ are chosen, let us first construct $\cT^3_{i,s}$. Let $G^2_{i,s} \coloneqq G^2 - \bigcup_{r=1}^{i-1}\bigcup_{T \in \cT^3_r} E(T) - \bigcup_{r=1}^{s-1}\bigcup_{T \in \cT^3_{i,r}}E(T)$.
Then since $\cT^3_r$ for $r \in [i-1]$ and $\cT^3_{i,r}$ for $r \in [s-1]$ satisfy~\ref{para:step3}, for every $i \in [3]$, $y \in [3] \setminus \{i \}$ and $u \in U_i$, we have\COMMENT{For the second inequality below, note that by our construction of $G^2$, every vertex $u \in U_i$ is complete to $U_y$, so $d_{G^2}(u , U_y) = |U_y|$, and $\Delta(G^2[U_i]) \le 3\eps_2 \eps_3 n$ by \ref{minmaxdegUi}, $s \leq q \le 10 \eps_2 \eps_3 n$ and $|U_y| = \lfloor \eps_3 n \rfloor$. Thus $d_{G^2}(u , U_y) - |\Delta(G^2[U_i])| - 2(s-1) \ge |U_y| - 3 \eps_2 \eps_3 n - 20 \eps_2 \eps_3 n \ge (1-25 \eps_2) |U_y|.$}
\begin{equation}\label{eqn:deg_g3}
	d_{G^2_{i,s}}(u , U_y) \geq d_{G^2}(u , U_y) - \Delta(G^2[U_i]) - 2(s-1) \overset{\ref{minmaxdegUi}}{\geq} (1 - 25 \eps_2)|U_y|,
\end{equation}
since there are at most $\Delta(G^2[U_i])$ $i$-type triangles in $\bigcup_{r=1}^{i-1} \cT^3_r \cup \bigcup_{r=1}^{s-1} \cT^3_{i,r}$
containing $u$ which have a centre vertex in $U_y$, and for every $r \in [s-1]$, the set $\cT^3_{i,r}$ contains at most one $y$-type triangle with centre vertex $u$.
Let us define an auxiliary bipartite graph $H_{i,s}$ with bipartition $\{ C_s, M_s^j \}$ and its random subgraph $H_{i,s}(p)$ such that for $w \in C_s$ and $e = uv \in M_s^j$, $\{w,e\} \in E(H_{i,s})$ if and only if $\{u,v,w\}$ is a triangle in $G^2_{i,s}$, and $\{w,e\} \in E(H_{i,s}(p))$ if and only if $\{u,v,w\}$ is a triangle in $G^2_{i,s}$ also lying in $\cS_{i,s}^3 (p)$ (note that $\cS_{i,s}^3 (p)$ is defined before~\ref{para:step3}). Then $H_{i,s}(p)$ is a random subgraph of $H_{i,s}$ where each edge of $H_{i,s}$ is chosen with probability $p$ independently at random.
Similarly, we define another bipartite graph $H_{i,s}'$ and its random subgraph $H_{i,s}'(p)$, with $C_s'$ and $N_s^k$ playing the role of $C_s$ and $M_s^j$, respectively.
By~\eqref{eqn:deg_g3}, both $\delta(H_{i,s})$ and $\delta(H_{i,s}')$ are at least $h_s-50 \eps_2 |U_y| \geq (1 - 12500\eps_2)h_s$ since $h_s \in \{\ell, \ell+1\}$ and $\ell \ge |U_y|/250$ by \eqref{eqn:size_ell}. 
Thus applying Lemma~\ref{lem:manymatchings} to $H_{i,s}$ and $H_{i,s}'$, with probability at least $1 - n^{-C^{1/4}}$, there exist perfect matchings $M_{i,s}$ and $M'_{i,s}$ in $H_{i,s}(p)$ and $H_{i,s}'(p)$ respectively, exposing the triangles in $\cS_{i,s}^3 (p)$; since $\{ \cS_{i,s}^3 : i \in [3] , s \in [q] \}$ are disjoint subsets of $\cS^3$, we do not expose the same triangle in $\cS^3$ again.
Now let $\cT^3_{i,s} \coloneqq \{ \{u,v,w\} : \{uv, w\} \in M_{i,s} \} \cup \{ \{u',v',w'\} : \{u'v', w'\} \in M'_{i,s}\}$. Then $\cT^3_{i,s}$ is a set of $2h_s$ edge-disjoint triangles satisfying~\ref{para:step3} for $r = s$. 
Since they are triangles in $G_{i,s}^2$, they are edge-disjoint from the triangles in $\bigcup_{r=1}^{i-1} \cT_3^r \cup \bigcup_{r=1}^{s-1} \cT^3_{i,r}$, as desired.

Now we describe how to choose the sets $C_s,C_s' \subseteq U_i$ of size $h_s$ such that~\eqref{eqn:ineq_ws} holds for $r = s$. Let $U_{i,s}^+$, $U_{i,s}^-$, and $U_{i,s}^0$ be the sets of vertices $u \in U_i$ with $w_{i,s-1}(u) > 0$, $w_{i,s-1}(u) < 0$, and $w_{i,s-1}(u) = 0$, respectively. Then we have the following cases.
\begin{enumerate}
	\item\label{cond:case1} If $|U_{i,s}^+|,|U_{i,s}^-| \geq h_s$, then choose any disjoint $h_s$-sets $C_s \subseteq U_{i,s}^+$ and $C_s' \subseteq U_{i,s}^-$. 

	\item\label{cond:case2} If $0 < \min(|U_{i,s}^+|,|U_{i,s}^-|) < h_s$, then since~\eqref{eqn:ineq_ws} holds for $r = s-1$, we have $|U_{i,s}^+|,|U_{i,s}^-| > 0$. Choose any $C_s,C_s'$ of size $h_s$ which satisfy:
	\begin{itemize}
	    \item If $|U_{i,s}^+| \leq |U_{i,s}^-|$, then $C_s \setminus C_s' = U_{i,s}^+$, $C_s' \setminus C_s \subseteq U_{i,s}^-$, and $C_s \cap C_s' \subseteq U_{i,s}^- \cup U_{i,s}^0$.
	    
	    \item If $|U_{i,s}^+| > |U_{i,s}^-|$, then $C_s' \setminus C_s = U_{i,s}^-$, $C_s \setminus C_s' \subseteq U_{i,s}^+$, and $C_s \cap C_s' \subseteq U_{i,s}^+ \cup U_{i,s}^0$.
	\end{itemize}

	\item\label{cond:case3} If $|U_{i,s}^+| = |U_{i,s}^-| = 0$, then choose any $h_s$-set $C_s = C_s' \subseteq U_{i,s}^0$. 
\end{enumerate}

Thus, $w_{i,s}(u) = w_{i,s-1}(u) - 1$ for $u \in C_s \setminus C_s'$, $w_{i,s}(u) = w_{i,s-1}(u) + 1$ for $u \in C_s' \setminus C_s$, and $w_{i,s}(u) = w_{i,s-1}(u)$ for any other vertex $u \in U_i$.\COMMENT{This shows that the first part of \eqref{eqn:ineq_ws} holds for $r = s$ since we assumed \eqref{eqn:ineq_ws} holds for $r = s-1$.}
This, and our choice of $C_s$ and $C'_s$ show that $|w_{i,s}(u)| = |w_{i,s-1}(u)| - 1$ for any $u \in (C_s \setminus C_s') \cup (C_s' \setminus C_s)$ and $|w_{i,s}(u)| = |w_{i,s-1}(u)|$ for any other vertex $u$ of $U_i$. Thus \eqref{eqn:ineq_ws} holds for $r=s$ as desired.
Hence, by induction,~\eqref{eqn:ineq_ws} holds for all $r \in [q]$. 
So if $\norm{w_{i,r}} = 0$ for some $r = s_0$, then $\norm{w_{i,r}} = 0$ for all $r > s_0$. Let $s_0 \in [q]$ be the smallest index such that $\norm{w_{i,s_0}} = 0$. 


\begin{claim}
\label{claim:terminate}
We have $s_0 = O(n^{2/3})$.
\end{claim}
\begin{claimproof}
We first show that \eqref{cond:case1} or \eqref{cond:case2} occurs for at most $O(n^{2/3})$ indices $s \in [q]$. Indeed, whenever~\eqref{cond:case1} occurs for $s \in [q]$, then $\norm{w_{i,s}} = \norm{w_{i,s-1}} - 2h_s$, thus by~\eqref{eqn:size_ell} and by the fact that $h_s \ge \ell$,~\eqref{cond:case1} occurs for at most $\norm{w_{i,0}} / (2\ell) = O(n^{2/3})$ indices $s \in [q]$.
Moreover, whenever~\eqref{cond:case2} occurs for $s \in [q]$, either $\max_{u \in U_{i,s}^+} |w_{i,s}(u)| =  \max_{u \in U_{i,s-1}^+} |w_{i,s-1}(u)| - 1$ 
or $\max_{u \in U_{i,s}^-} |w_{i,s}(u)| = \max_{u \in U_{i,s-1}^-} |w_{i,s-1}(u)| - 1$ holds.
Thus,~\eqref{cond:case2} occurs for at most $n^{2/3}$ indices $s \in [q]$ since $|w_{i,s}(u)| \leq |w_{i,0}(u)| \leq n^{2/3}/2$, as desired. Now since \eqref{cond:case3} occurs if and only if $w_{i,s}(u) = 0$ for every $u \in U_i$ (i.e. $\norm{w_{i,s}} = 0$ for some $s \in [q]$), this proves the claim. 
\end{claimproof}

Claim~\ref{claim:terminate} shows that $\norm{w_{i,q}} = 0$ (i.e., $w_{i,q}(u) = 0$ for every $u \in U_i$) since we assumed $q = \lfloor 10 \eps_2 
\eps_3 n\rfloor > s_0$. So $G^3 \coloneqq G^2_{3,q+1} = G^2 - \bigcup_{T \in \cT^3}E(T)$ satisfies the divisibility condition (\ref{step:step3}:b) as desired.


\begin{pf-step}\label{step:step4}
Finishing the proof.
\end{pf-step}
In this last step, we construct the desired subgraph $H$ by removing some edge-disjoint triangles from $G^3$ which cover all remaining edges of $G$ incident to $W_3$.

Consider a partition $\{ \cS^4_{w} \}_{w \in U_3}$ of $\cS^4$, where $\cS^4_w \coloneqq \{ \{ u,v,w \} : u \in V_1, v \in V_2 \}$. Note that, for every $w \in U_3$,
\begin{equation}
\label{eqn:v1v2div}
    d_{G^3}(w , V_1) = d_{G^3}(w , V_2) \geq \lfloor n/3 \rfloor - 25 \eps_2 \eps_3 n.
\end{equation}
Indeed, since the triangles in $\bigcup_{j=1}^{3}\cS^j$ do not contain any edge of $G$ between $V_1 \setminus U_1$ and $V_2 \setminus U_2$, we have $N_{G^3}(w , V_2 \setminus U_2) = V_2 \setminus U_2$ for every $w \in U_3$. Moreover, since $G^3 = G^2_{3,q+1}$ satisfies  (\ref{step:step3}:b) we have $d_{G^3}(w , V_1) = d_{G^3}(w , V_2)$. Thus, since \eqref{eqn:deg_g3} holds (for $i=3$ and $s = q+1$), we have
$d_{G^3}(w , V_2) = d_{G^3}(w , U_2) + d_{G^3}(w , V_2 \setminus U_2) \ge (1-25 \eps_2) |U_2| + |V_2 \setminus U_2| = |V_2| - 25 \eps_2 |U_2| \ge \lfloor n/3 \rfloor - 25 \eps_2 \eps_3 n$.\COMMENT{because $|V_1| = |V_2| = \lfloor n/3 \rfloor$ and $|U_1| = |U_2| = \lfloor \eps_3 n \rfloor$.}

Thus, using~\eqref{eqn:v1v2div} and the fact that $|U_3| \leq \eps_3 n$, exposing triangles in $\cS_w^4 (p)$,
we apply Lemma~\ref{lem:manymatchings} repeatedly
for every $w \in U_3$ to find a set $\{ M_{w} \}_{w \in U_3}$ of edge-disjoint matchings in $G^3$ with probability at least $1 - n^{-C^{1/3}}$, where $M_{w}$ is a perfect matching\COMMENT{By \eqref{eqn:v1v2div}, $|N_{G^3}(w) \cap V_1| = |N_{G^3}(w) \cap V_2| \ge \lfloor n/3 \rfloor - 25 \eps_2 \eps_3 n$. Moreover, the subgraph of $G^3$ induced by $(N_{G^3}(w) \cap V_1) \cup (N_{G^3}(w) \cap V_2)$ has minimum degree at least $|N_{G^3}(w) \cap V_2| - 25 \eps_2 |U_2| - |U_3|$ after removing some of the matchings $M_w$ for $w \in U_3$. Here we used \eqref{eqn:deg_g3} (for $i=3$ and $s = q+1$) and the fact that any vertex in $(N_{G^3}(w) \cap V_1) \setminus U_1$ (or $(N_{G^3}(w) \cap V_2) \setminus U_2$) is complete to $N_{G^3}(w) \cap V_2$ (or $N_{G^3}(w) \cap V_1$ respectively) in $G^3$. Moreover, $|N_{G^3}(w) \cap V_2| - 25 \eps_2 |U_2| - |U_3| \ge |N_{G^3}(w) \cap V_2|-25 \eps_2 \eps_3 n - \eps_3 n \ge |N_{G^3}(w) \cap V_2| - 26 \eps_3 n \ge (1- 100 \eps_3)|N_{G^3}(w) \cap V_2|$, so we can apply Lemma~\ref{lem:manymatchings}.} between $N_{G^3}(w) \cap V_1$ and $N_{G^3}(w) \cap V_2$ such that $\{u,v,w \} \in \cS^4_w(p)$ for every $uv \in M_w$. Let $\cT^4 \coloneqq \{ \{u,v,w \} : w \in U_3,\:uv \in M_{w} \}$, and recall that $\cT = \bigcup_{i = 1}^{4}\cT^i$. Let $H \coloneqq G^3 - \bigcup_{T \in \cT^4}E(T) = G - \bigcup_{T \in \cT}E(T)$. We will show that $H$ is the desired tripartite graph satisfying \ref{deg_red1} and \ref{reg_red1}.

Recall that we defined (in Step 1) that $W_3 = U_3 \cup \{v^* \}$ if $n = 6t+1$ and $W_3 = U_3$ if $n = 6t+3$ for some integer $t$. By \ref{degVioutsideUi} and \ref{coverremedgesUi} we have  $e(H[V_i]) = 0$ for $i \in [3]$, so $H$ is a tripartite graph with parts $V_1, V_2, V_3$. 
Moreover, note that the triangles in $\cT \subseteq \bigcup_{j=1}^{4} \cS^j$ do not contain any edge of $G$ between $V_3 \setminus W_3$ and $V_1 \cup V_2$, so $N_H(u) = V_1 \cup V_2$ for $u \in V_3 \setminus W_3$ (as $H$ is tripartite). 
Since the triangles in $\cT^1, \cT^3, \cT^4$ remove all edges of $G$ between $W_3$ and $V_1 \cup V_2$, we have $d_H(u) = 0$ for $u \in W_3$ as $H$ is tripartite. 
Thus, $H$ satisfies \ref{deg_red1}.

It remains to show that $H$ satisfies \ref{reg_red1}. Since $G^3$ satisfies the divisibility condition (\ref{step:step3}:b), for any vertex $v \in V_1$, we have $d_{G^3}(v , V_2) = d_{G^3}(v , V_3)$. Since triangles in $\cT^4$ are edge-disjoint, and every triangle containing $v \in V_1$ in $\cT^4$ has exactly one vertex in $V_2$ and in $V_3$ and $H = G^3 - \bigcup_{T \in \cT^4}E(T)$, we have\COMMENT{For every $v \in V_1$, whenever we remove an edge from $v$ to $V_3$, we also remove an edge from $v$ to $V_2$ due to the construction of $\cT^4$, so the $H = G^3 - \cup_{T \in \cT^4} E(T)$ satisfies the divisibility condition as well.} $d_{H}(v , V_2) = d_{H}(v , V_3)$ for every $v \in V_1$. By \ref{deg_red1},  for every $v \in V_1$, $d_{H}(v , V_3) = |V_3| - |W_3|$, so we also have $d_{H}(v , V_2) = |V_3| - |W_3|$. Similarly, $d_{H}(v , V_1)= |V_3| - |W_3|$ for every $v \in V_2$, so $H$ also satisfies \ref{reg_red1}, as desired. 
\end{proof}

\section{1-factorization of $K_{2n}$ reduction: Proof of Theorem~\ref{thm:list-reduction}}\label{sect:list-reduction}


In this section, we prove Theorem~\ref{thm:list-reduction}, but it is convenient for us to instead prove the following result concerning triangle decompositions which easily implies Theorem~\ref{thm:list-reduction}. 

\begin{theorem}\label{thm:red2}
Let $1/n_0 \ll 1/C \ll \eps \ll 1$, and let $n \ge n_0$. Let $\{V_1 , V_2 , C_1 , C_2\}$ be an equitable partition of $[4n-1]$ with $|C_1| = n-1$, and let 
$G$ be the graph with $V(G) \coloneqq V_1 \cup V_2 \cup C_1 \cup C_2$ obtained by adding all edges $xy$ with $x \in V_1 \cup V_2$.

Let $C_2' \subseteq C_2$ be any subset of size $\lfloor \eps n \rfloor$. 
If $p \geq C \log n/n$, then with probability at least $1-1/n$, there exists a collection $\cT \subseteq \randt(G , p)$ of edge-disjoint triangles such that $H \coloneqq G - \bigcup_{T \in \cT} E(T)$ is a tripartite graph with the partition $\{ V_1 , V_2 , C_1 \cup C_2 \}$ satisfying the following properties.
\begin{enumerate}[(\ref*{thm:red2}:a), topsep = 6pt]
    
    \item \label{deg_red2} For every $c \in C_1 \cup C_2'$, $d_H(c) = 0$, and every $c \in C_2 \setminus C_2'$ satisfies $N_H(c)=V_1 \cup V_2$. 
    
    \item \label{reg_red2} $H[V_1 \cup V_2]$ is $(|C_2| - |C_2'|)$-regular.
\end{enumerate}
\end{theorem}

If $G$ is the join of a complete graph on $2n$ vertices and an empty graph on $2n - 1$ vertices (as in Theorem~\ref{thm:red2}), then $1$-factorizations of $K_{2n}$ are in bijection with triangle decompositions of $G$, so we can couple the distributions of a random $(p, 2n - 1)$-list assignment $L$ for the edges of $K_{2n}$ with $\randt(G, p)$ so that an $L$-edge-colouring exists if and only if $\randt(G, p)$ contains a triangle decomposition of $G$.  Using this coupling, it is straightforward to derive Theorem~\ref{thm:list-reduction} from Theorem~\ref{thm:red2}.

\subsubsection*{An overview of the proof of Theorem~\ref{thm:red2}.} 



In Step~\ref{step:step1'}, using the R\"odl nibble, we cover most of the edges in $G[V_1] \cup G[V_2] \cup G[V_1 \cup V_2, C_1]$ using a set $\cT^0$ of edge-disjoint triangles. 
In Step~\ref{step:step1'new}, we cover all remaining edges of $G$ incident to a vertex of $C_1$ using triangles containing an edge of $G[V_1, V_2]$. For this step, we define the `leftover sets' $W_s(c) \coloneqq V_s \setminus \bigcup_{T \in \cT^0(c)} V(T)$ for $s \in [2]$, where $\cT^0(c) \coloneqq \{ T \in \cT^0 : c \in V(T) \}$ for each $c \in C_1$. Then, using Lemma~\ref{lem:coverdown_reservoir}, we find a perfect matching in $G[W_1(c), W_2(c)]$ for each $c \in C_1$ such that these perfect matchings are edge-disjoint; each perfect matching $M_c$ corresponds to a set of triangles in $G$ covering the remaining edges incident to $c$. To find these perfect matchings, we need the leftover sets $W_s(c)$ for $c \in C_1$ to be well-distributed (e.g., they must not overlap very much). To achieve this property of leftover sets, we use Lemma~\ref{prop:pseudomat}, and we find the almost cover $\cT^0$ in Step~\ref{step:step1'} in `batches'. Let $\cT'$ be the triangles found in Step~\ref{step:step1'new}.


In Step~\ref{step:step2'}, we cover all of the remaining edges in $G[V_1] \cup G[V_2]$ using a set $\cT''$ of edge-disjoint triangles of the form $\{u,v,w\}$ where $uv \in E(G[V_1]) \cup E(G[V_2])$ and $w \in C_2'$ satisfying certain properties. Finally, in Step~4, we cover all remaining edges of $G$ incident to vertices in $C_2'$ using a set $\cT'''$ of edge-disjoint triangles. Then we obtain the desired subgraph $H$ obtained by removing all the triangles in $\cT^0 \cup \cT' \cup \cT'' \cup \cT'''$ from $G$.


\subsection{Finding perfect matchings in random subgraphs of almost disjoint, dense graphs}
The following lemma is a randomised bipartite version of \cite[Lemma 3.10]{BGKLMO20}, whose proof is almost identical to the proof given in~\cite{BGKLMO20}. The main difference is that we use Lemma~\ref{lem:manymatchings} instead of finding edge-disjoint matchings one by one greedily.


\begin{lemma}\label{lem:coverdown_reservoir}
Let $1/n_0 \ll 1/C, \eta \ll \rho \ll 1$, and let $n \ge n_0$. Let $1 \leq m \leq n^2$.
Let $V_1 , \dots , V_m , V_1' , \dots , V_m'$ be sets satisfying the following.
\begin{enumerate}[(i)]
    \item For every $i \in [m]$, $|V_i| = |V_i'|$ and $|V_i| \geq \rho^{4/3} n$.
    \item For every $i \in [m]$, there are at least $m - \rho^3 n$ indices $j \in [m]$ such that 
    $|V_i \cap V_j| , |V_i' \cap V_j'| \leq \rho^2 n$.\COMMENT{We use this lemma in the proof of Theorem~\ref{thm:red2}, when dealing with batches, so we need to allow a few exceptional indices from the same block.}
    
    \item \label{assumption1} Every vertex $v \in \bigcup_{i=1}^{m} V_i$ is in at most $2 \rho n$ sets among $V_1 , \dots , V_m$.
    \item \label{assumption2} Every vertex $v \in \bigcup_{i=1}^{m} V_i'$ is in at most $2 \rho n$ sets among $V_1' , \dots , V_m'$.
\end{enumerate}

For every $i \in [m]$, let $H_i$ be a bipartite graph with the bipartition $\{V_i, V_i'\}$ and $\delta(H_i) \geq (1 - \eta)|V_i|$. Let $G_i$ be a random subgraph obtained from $H_i$ by choosing each edge with probability $p \geq C \log n/n$ independently at random, and $G_1 , \dots , G_m$ are mutually independent.
With probability at least $1 - n^{-C^{1/4}}$, for all $i \in [m]$, there exists a perfect matching $M_i$ of $G_i$, such that $M_1 , \dots , M_m$ are edge-disjoint.
\end{lemma}

 
Lemma~\ref{lem:coverdown_reservoir} is used in two different ways. It is used in the proof of Theorem~\ref{thm:red2}
to cover leftover sets of certain colour classes with edge-disjoint matchings. It is also used in the proof of Lemma~\ref{lem:coverdown} to find edge-disjoint matchings in the neighbourhoods of some vertices.

\begin{proof}[Proof of Lemma~\ref{lem:coverdown_reservoir}]
Without loss of generality, we may assume that $p = \frac{C \log n}{n}$. For every $i \in [m]$, let $I_i \subseteq [m]$ be the set of indices $j \in [m] \setminus \{ i \}$ such that $|V_i \cap V_j|, |V_i' \cap V_j'| \leq \rho^2 n$. By our assumption, $|I_i| \geq m - \rho^3 n$. 
For every $i \in [m]$, $j \in I_i$, $v \in V_i \cap V_j$, and $w \in V_i' \cap V_j'$, we have $\mathbb{E}|N_{G_i}(v) \cap V_j'|,\:  \mathbb{E}|N_{G_i}(w) \cap V_j|  \leq  \rho^2 C \log n$.
Thus by Lemma~\ref{lemma:chernoff}, with probability at least $1 - n^{-C^{2/3}}$, 
for all $i \in [m]$, $j \in I_i$, $v \in V_i \cap V_j$, and $w \in V_i' \cap V_j'$, we have
\begin{equation}\label{eqn:cond1'}
    |N_{G_i}(v) \cap V_j'|,\:  |N_{G_i}(w) \cap V_j| \leq 2 \rho^2 C \log n.
\end{equation}

We construct the desired perfect matchings $M_i$ in $G_i$ for $i \in [m]$, one by one. More precisely, for every $i \in [m]$, we aim to choose a perfect matching $M_i$ in $G_i$ such that it is edge-disjoint from the previously chosen matchings $M_1 , \dots , M_{i-1}$ from $G_1 , \dots , G_{i-1}$, respectively. Note that if $\delta(H_i - \bigcup_{k=1}^{i-1} M_k) \geq (1 - 500\rho^{1/6})|V_i|,$ then applying Lemma~\ref{lem:manymatchings} with $H_i - \bigcup_{k=1}^{i-1} M_k$, $|V_i|$, $500\rho^{1/6}$ and $\frac{C|V_i|}{2n}$ playing the roles of $G$, $n$, $\eps$ and $C$, respectively, we find\COMMENT{$\frac{C|V_i|}{2n} \frac{\log|V_i|}{2} \ge \lceil \frac{\rho^{3/2} C \log n }{20} \rceil$} $\ell \coloneqq \lceil \frac{\rho^{3/2} C \log n }{20} \rceil$ edge-disjoint perfect matchings $N_1 , \dots , N_\ell$ in $G_i$ with probability at least $1 - n^{-C^{1/3}}$. Let $M_i \coloneqq N_{t_i}$ where $t_i$ is chosen uniformly at random from $[\ell]$. Then, Claim~\ref{claim:degbound'} (stated below) and ~\eqref{eqn:cond1'} imply that we can find the desired matchings $M_i$ in $G_i$ for all $i \in [m]$ with probability at least $1 - n^{-C^{2/3}} - m \cdot (n^{-C^{1/3}} + n^{-C}) \geq 1 - n^{-C^{1/4}}$, proving our lemma. 

\begin{claim}\label{claim:degbound'}
For every $i \in [m]$, the minimum degree of $H_i - \bigcup_{k=1}^{i-1} M_k$ is at least $(1 - \eta)|V_i| - 200 \rho^{3/2} n \geq (1 - 500\rho^{1/6})|V_i|$ with probability at least $1 - n^{-C}$.
\end{claim}

\begin{claimproof}
Let $i \in [m]$. For every $u \in V_i$ (or $u \in V_i'$), let $\cF_i(u)$ be the set of indices $j \in [i-1] \cap I_i$ such that $u \in V_j$ ($u \in V_j'$, respectively), 
and let $\cG_i(u)$ be the set of indices $j \in [i-1] \setminus I_i$ such that $u \in V_j$ ($u \in V_j'$, respectively).
For every $j \in [i-1]$ and $u \in V_i$ (or $u \in V_i'$), let $X_j(i , u)$ be the indicator random variable for the event that $uv \in M_j$ for some $v \in V_i'$ ($v \in V_i$, respectively). Then for any $u \in V_i$  (or $u \in V_i'$), we have
\begin{equation*}
    d_{H_i - \bigcup_{k=1}^{i-1} M_k}(u) = d_{H_i}(u) - \sum_{j \in \cF_i(u)} X_j (i , u) - \sum_{j \in \cG_i(u)} X_j(i , u).
\end{equation*}

Therefore, since $|\cG_i(u)| \leq m - |I_i| \leq \rho^3 n$, in order to prove the claim, it suffices to show that $X(i,u) \coloneqq \sum_{j \in \cF_i(u)} X_j (i , u) \leq 150 \rho^{3/2} n$ holds with probability at least $1 - n^{-C}$ for all $u \in V_i$  (or $u \in V_i'$).

Without loss of generality we may assume that $u \in V_i$. 
Let $\cF_i(u) = \{ j_1 , \dots , j_{|\cF_i(u)|} \}$, where $j_1 < j_2 < \dots < j_{|\cF_i(u)|}$ is the enumeration of the elements in $\cF_i(u)$ in increasing order. Among the $\ell$ edge-disjoint matchings we have found in  $G_{j_k}$ (while choosing $M_{j_k}$), by~\eqref{eqn:cond1'}, there are at most $|N_{G_{j_k}}(u) \cap V'_{i}| \leq 2 \rho^2 C \log n$ matchings which make $X_{j_k}(i , u) = 1$. 
Thus, 
\begin{equation*}
    \mathbb{P}(X_{j_k}(i , u) = 1 \: | \: X_{j_1}(i,u) , \dots , X_{j_{k-1}}(i,u)) \leq \frac{2 \rho^2 C \log n}{\ell} \leq 60 \rho^{1/2}.
\end{equation*}

Let $B \sim {\rm Bin}(|\cF_i(u)|, 60\rho^{1/2})$. 
Since $X(i,u)$ is stochastically dominated by $B$ and $|\cF_i(u)| \leq 2 \rho n$ by ~\ref{assumption1} and ~\ref{assumption2} of Lemma~\ref{lem:coverdown_reservoir}, we have
$\mathbb{P} ( X(i,u) > 150 \rho^{3/2} n ) \leq \mathbb{P}(B > 150 \rho^{3/2} n) \leq e^{-\Omega(\rho^{3/2} n)}$ by Lemma~\ref{lemma:chernoff}.
This proves the claim by taking  a union bound for all $u \in V_i \cup V_i'$.
\end{claimproof}
\end{proof}

\subsection{Proof of Theorem~\ref{thm:red2}}

\begin{proof}[Proof of Theorem~\ref{thm:red2}]
Choose new constants $\delta,  \gamma$ such that $1/C \ll \delta \ll \gamma \ll \eps$.

\setcounter{pf-step}{0}

We define disjoint subsets of $G^{(3)}$ as follows.
For every $c \in C_1 \cup C_2$, let $\cA_c$ be the set of triangles in $G^{(3)}$ of the form $\{u,v,c \}$ where $uv \in E(G[V_1]) \cup E(G[V_2])$, and let $\cB_c$ be the set of triangles in $G^{(3)}$ of the form $\{u,v,c \}$ where $uv \in E(G[V_1,V_2])$.

To construct the desired set $\cT \subseteq \randt(G , p)$ of edge-disjoint triangles, we will expose triangles in $\cA_c (p)$ for $c \in C_1$ in Step~\ref{step:step1'}, and we will expose triangles in $\cB_c(p)$ for $c \in C_1$ in Step~\ref{step:step1'new}. We will expose triangles in $\cA_c (p)$ for $c \in C'_2$ in Step~\ref{step:step2'}, and we will expose triangles in $\cB_c (p)$ for $c \in C'_2$ in Step~\ref{step:step3'}.

\begin{pf-step}\label{step:step1'}
Covering most of the edges in $G[V_1] \cup G[V_2] \cup G[V_1 \cup V_2, C_1]$.
\end{pf-step}
Let $K \coloneqq \lfloor 1 / \gamma^{4} \rfloor$.
Let $\{G^1 , \dots , G^K \}$ be a decomposition of $G[V_1] \cup G[V_2]$ 
such that $d_{G^i}(v) = (1 \pm n^{-1/3})n/K$ for every $v \in V_1 \cup V_2$ and $i \in [K]$.\COMMENT{Choose each edge of $G[V_1] \cup G[V_2]$ to be in $G^i$ with probability $1/K$ for every $i \in K$. Then by Lemma~\ref{lemma:chernoff}, for a given $v \in V_1 \cup V_2$, $d_{G^i}(v) = (1 \pm n^{-1/3})n/K$ with probability $1 - \exp(\Omega(n^{1/3}/K))$, so with positive probability $d_{G^i}(v) = (1 \pm n^{-1/3})n/K$ for all $v \in V_1 \cup V_2$ and $i \in [K]$.} 
Let $\{ C_1^1 , \dots , C_1^K \}$ be an equitable partition of $C_1$.
Suppose that for some $i \in [K]$, by exposing triangles in $\bigsqcup_{j=1}^{i-1} \bigsqcup_{c \in C_1^j} \cA_c(p)$, we have already obtained the sets $\cT_1 , \dots , \cT_{i-1}$ of edge-disjoint triangles in $G$ satisfying the following statements (\ref{step:step1'}:a)$_{r}$--(\ref{step:step1'}:d)$_{r}$ for $r = i-1$.

\begin{enumerate}[label=(\ref*{step:step1'}:\alph*)$_{r}$, topsep = 6pt]
    \item\label{cond:t1} The triangles in $\bigcup_{j=1}^{r}\cT_j$ are edge-disjoint. Moreover, for $j \in [r]$, every triangle in $\cT_{j}$ contains exactly one vertex in $C_1^j$, and two other vertices which are adjacent in $G^j$ (thus both vertices lie in either $V_1$ or $V_2$).

    \item\label{cond:t3} For every $j \in [r]$ and $c \in C_1^j$, let $\cT_j(c)$ be the set of triangles in $\cT_j$ containing $c$. Then $\cT_j(c) \subseteq \cA_c(p)$.
    Moreover, for $s \in [2]$, if $W_s(c) \coloneqq V_s \setminus \bigcup_{T \in \cT_j(c)}V(T)$, then
    \begin{equation*}
        0.8 \gamma n \leq |W_1(c)| = |W_2(c)| \leq 2 \gamma n.  
    \end{equation*}
    
    \item\label{cond:t4} For all $1 \leq j \ne k \leq r$, $c_j \in C_1^j$, and $c_k \in C_1^k$, we have
    \begin{equation*}
        0.64 \gamma^2 n \leq |W_1(c_j) \cap W_1(c_k)| , |W_2(c_j) \cap W_2(c_k)| \leq 4 \gamma^2 n.
    \end{equation*}
    
    \item\label{cond:t5} For every $s \in [2]$, $j \in [r]$, and every $v \in V_s$, the degree of $v$ in $G^j[V_s] - \bigcup_{T \in \cT_j}E(T)$ lies in $[0.6 \gamma n / K\:,\: 2 \gamma n / K]$.
    Moreover, $v$ is in at most $2 \gamma n / K$ sets in $\{ W_s(c) \}_{c \in C_1^j}$.
\end{enumerate}

Conditioning on the choices of $\cT_1 , \dots , \cT_{i-1}$ satisfying~\ref{cond:t1}--\ref{cond:t5} for $r = i-1$, 
and exposing triangles in $\bigsqcup_{c \in C_1^i} \cA_c (p)$, with probability at least $1 - n^{-C^{2/3}}$, now we show that there exists a set of edge-disjoint triangles $\cT_i$ such that (\ref{step:step1'}:a)$_{r}$--(\ref{step:step1'}:d)$_{r}$ are satisfied for $r = i$.
To that end, let $\cH^i$ be a 3-uniform hypergraph such that 
\begin{align*}
    V(\cH^i) &\coloneqq E(G^i) \cup \{ vc \: : \: v \in V_1 \cup V_2 ,\: c \in C_1^i \},\\
    E(\cH^i) &\coloneqq \{ \{uv,uc,vc\} \: : \: uv \in E(G^i),\: c \in C_1^i\}.
\end{align*}
Note that every vertex in $\cH^i$ has degree $(1 \pm n^{-1/3})\frac{n}{K}$,\COMMENT{$d_{\cH^i}(vc) = d_{G^i}(v) = (1 \pm n^{-1/3})n/K$, and $d_{\cH^i}(uv) = |C_1^i|= \lfloor (n-1)/K \rfloor \pm 1 = (1 \pm n^{-1/3})n/K$} and $|V(\cH^i)| = \Omega(n^2)$. Let $\cH_p^i$ be the random subhypergraph of $\cH^i$ such that every edge $\{uv,uc,vc\} \in E(\cH^i)$ is in $\cH_p^i$ if and only if $\{u,v,c\} \in \cA_c(p)$. 
By Lemma~\ref{lemma:chernoff},\COMMENT{with probability at least $1-\exp(-\frac{\delta^2 C \log n}{10K}) \ge 1-n^{-C^{4/5}}$ a given vertex has degree $(1 \pm \delta)\frac{C \log n}{K}$, so by a union bound, all vertices have degree $(1 \pm \delta)\frac{C \log n}{K}$ with probability at least $1-n^{-C^{2/3}}$.} with probability at least $1 - n^{-C^{2/3}}$, every vertex has degree $(1 \pm \delta)\frac{C \log n}{K}$ in $\cH_p^i$. 

For every $s \in [2]$, $c \in C_1^i$ and all $c' \in C_1^j$ with $j \in [i-1]$, let $E_s(c,c') \coloneqq \{ vc \: : \: v \in W_s(c')\}$, and let $V_s(c) \coloneqq \{ vc \: : \: v \in V_s \}$. 
For every $w \in V_1 \cup V_2$, let $V^i(w)$ be the set of edges in $G^i$ incident to $w$, and $C^i(w) \coloneqq \{ wc : c \in C_1^i \}$. For $c \in C_1^i$, let us define
\begin{align*}
    \cF(c) &\coloneqq \{ V_1(c) , V_2(c) \} \cup \{ E_s(c,c') : s \in [2],\: c' \in C_1^1 \cup \dots \cup C_1^{i-1} \}, \\ 
    \cF_i &\coloneqq \bigcup_{c \in C_1^i}\cF (c) \cup \{ V^i(w) : w \in V_1 \cup V_2 \} \cup \{ C^i(w) : w \in V_1 \cup V_2 \}.
\end{align*}

Applying Lemma~\ref{prop:pseudomat} to $\cH_p^i$ with $\cF_i$ playing the role of $\cF$, we obtain a matching $\cN_i$ in $\cH_p^i$ such that for every $S \in \cF_i$, 
\begin{equation}\label{eqn:size_s}
    0.8 \gamma |S| \leq |S \setminus V(\cN_i)| \leq \gamma |S|.
\end{equation}

Since $\{uv,uc,vc \} \in E(\cH^i)$ naturally corresponds to $\{u,v,c \} \in G^{(3)}$, the matching $\cN_i$ corresponds to a set $\cT_i^* \subseteq \bigsqcup_{c \in C_1^i} \cA_c(p)$ of edge-disjoint triangles in $G$ such that

\begin{enumerate}[label=(\Alph*)]
    \item \label{A_r} Every triangle of $\cT_i^*$ contains an edge of $G^i$ and a vertex of $C_1^i$. Thus the triangles in $\cT_i^*$ are edge-disjoint from the triangles in $\bigcup_{j=1}^{i-1} \cT_j$.  
\end{enumerate}

For every $c \in C_1^i$ and $s \in [2]$, let $\cT_i^*(c) \coloneqq \{ T \in \cT_i^* : c \in V(T) \}$ and let $W^*_s(c) \coloneqq V_s \setminus \bigcup_{T \in \cT_i^*(c)} V(T)$. By \eqref{eqn:size_s}, \ref{cond:t3} for $r = i-1$  and the correspondence between $\cN_i$ and $\cT_i^*$, we have the following for every $c \in C_1^i$ and $s \in [2]$.

\begin{enumerate}[resume, label=(\Alph*)]
    \item \label{B_r} $|W_s^*(c)| \in [0.8 \gamma n , \gamma n]$.\COMMENT{For this, we let $S = V_s(c)$ in \eqref{eqn:size_s}, so $|S| = n$.}
    
    \item \label{C_r} For every $j \in [i-1]$ and $c' \in C_1^j$, we have $$|W_s^*(c) \cap W_s(c')| \in [0.8\gamma |W_s(c')| , \gamma |W_s(c')| ] \subseteq [0.64 \gamma^2 n, 2\gamma^2 n].$$\COMMENT{Letting $S = E_s(c,c')$ in \eqref{eqn:size_s} (so that $|S| = |W_s(c')|$), we have $|W_s^*(c) \cap W_s(c')| \in [4 \gamma |W_s(c')| / 5 , \gamma |W_s(c')|]$. Since \ref{cond:t3} holds for $r = i-1$, for $j \in [i-1]$ and $c' \in C_1^{j}$, we have $W_s(c') \in [4 \gamma n/5, 2 \gamma n]$, so $[4 \gamma |W_s(c')| / 5 , \gamma |W_s(c')|] \subseteq [16\gamma^2 n / 25, 2 \gamma^2 n]$.}
    
    \item \label{D_r} For every $v \in V_s$, the degree of $v$ in $G^i[V_s] - \bigcup_{T \in \cT_i^*} E(T)$ lies in $[\frac{0.6 \gamma n}{K}, \frac{1.2 \gamma n}{K}]$, and $v$ is in at most $1.2 \gamma n / K$ sets in $\{ W_s^*(c) \}_{c \in C_1^i}$.\COMMENT{Letting $S = V^i(v)$ in \eqref{eqn:size_s} (where $|S| = |V^i(v)| = d_{G^i}(v)=(1 \pm n^{-1/3})n/K$), the degree of $v$ in $G^i[V_s] - \bigcup_{T \in \cT_i^*} E(T)$ lies in $[4\gamma|V^i(v)|/5, \gamma |V^i(v)|] \subseteq [3\gamma n/5K, 6\gamma n/5K]$. Moreover, letting $S = C^i(v)$ in \eqref{eqn:size_s} (where $|S| = |C^i(v)| = (1 \pm n^{-1/3})n/K$), we obtain that $v$ is in at most $\gamma |S| \leq 1.2 \gamma n/K$ sets in $\{ W^*_s(c)\}_{c \in C_1^i}$.}
\end{enumerate}

For every $s \in [2]$ and $c \in C_1^i$, the edges in $G^i[V_s]$ covered by $\cT_i^*(c)$ form a matching. 
Thus, we can remove $||W_1^*(c)| - |W_2^*(c)||/2 \le 0.1 \gamma n$ triangles from $\cT_i^*(c)$ independently at random so that the resulting subset $\cT_i \subseteq \cT_i^*$ satisfies $|W_1(c)|=|W_2(c)|$
for every $c \in C_1^i$, where $W_s(c) \coloneqq V_s \setminus \bigcup_{T \in \cT_i(c)} V(T)$ for $s \in [2]$. 
Moreover, since $|W_s^*(c)| \le |W_s(c)| \le |W_s^*(c)|+0.2 \gamma n$\COMMENT{as we removed at most $0.1 \gamma n$ triangles from $\cT_i^*(c)$ and each triangle adds two vertices}, \ref{cond:t3} is satisfied for $r = i$ by ~\ref{B_r}. By ~\ref{A_r} and the fact that $\cT_i \subseteq \cT_i^*$, \ref{cond:t1} is satisfied for $r = i$. 
For every $c \in C_1^i$, we have $|\cT_i^*(c)| = (2n - |W_1^*(c)| - |W_2^*(c)|) / 2$, so by~\ref{B_r}, $|\cT_i^*(c)| \ge n - \gamma n$. 
Thus, for every $T \in \cT_i^*(c)$, 
$\mathbb{P}(T \in \cT_i^*(c) \setminus \cT_i(c)) \leq 0.15 \gamma$, and for every $s \in [2]$ and $w \in V_s \setminus W_s^*(c)$, $\mathbb{P}(w \in W_s(c)) \leq 0.15 \gamma$. By straightforward applications of Lemma~\ref{lemma:chernoff},~\ref{cond:t4} and~\ref{cond:t5} for $r=i$ follow from ~\ref{B_r},~\ref{C_r} and ~\ref{D_r}.
\COMMENT{For each $s \in [2]$ and $S \subseteq V_s \setminus W_s^*(c)$, let $X_S$ be the number of triangles in $\cT_i^*(c) \setminus \cT_i(c)$ intersecting $S$. Since there are at most $|S|$ triangles in $\cT_i^*(c)$ intersecting $S$ and each triangle affects $|S \cap W_s(c)|$ by at most two, $|S \cap W_s(c)| \leq 2X_S$ and $\mathbb{E}X_S \leq 0.15 \gamma |S|$. Applying Lemma~\ref{lemma:chernoff} to $X_S$ with $S = W_s(c') \setminus W_s^*(c)$ for $c' \in \bigsqcup_{j=1}^{i-1} C_1^j$, with high probability, $|S \cap W_s(c)| \leq 0.4 \gamma |S| \le 0.4 \gamma |W_s(c')|$, so
$|W_s(c) \cap W_s(c')| \leq |W_s^*(c) \cap W_s(c')| + 0.4 \gamma |W_s(c')|$. This combined with ~\ref{cond:t3} for $r=i-1$,~\ref{C_r} shows that ~\ref{cond:t4} is satisfied for $r=i$.} 
\COMMENT{For each $w \in V_s$ and $c \in C_1^i$, consider an indicator random variable $X_{w,c}$ which is 1 iff $w \in V_s \setminus W_s^*(c)$ and $wc \in E(G)$ is not covered by a triangle in $\cT_i(c)$. Then the number of sets in $\{W_s(c)\}_{c \in C_1^i}$ containing $w$ is the number of sets in $\{ W_s^*(c) \}_{c \in C_1^i}$ containing $w$ plus $\sum_{c \in C_1^i} X_{w,c}$, where $\mathbb{E}[\sum_{c \in C_1^i} X_{w,c}] \le 0.15 \gamma |C_1^i| \leq 0.2 \gamma n / K$, so we can apply Lemma~\ref{lemma:chernoff} to obtain $\sum_{c \in C_1^i} X_{w,c} \le 0.3 \gamma n/K$. Thus, by ~\ref{D_r}, the second statement of ~\ref{cond:t5} follows. Similarly, for each $v \in N_{G^i}(w) \setminus W_s^*(c)$, 
consider an indicator random variable $X_{wv}$ which is 1 iff $v \in W_s(c)$. Then the degree of $w$ in $G^i[V_s] - \bigcup_{T \in \cT_i}E(T)$ will be the degree of $w$ in $G^i[V_s] - \bigcup_{T \in \cT_i^*}E(T)$ plus $\sum_{v \in N_{G^i}(w)} X_{wv}$, where $\mathbb{E}\sum_{v \in N_{G^i}(w)} X_{wv} \leq 0.2 \gamma n / K$ as before. Thus, we can again apply Lemma~\ref{lemma:chernoff} to show  the first statement of ~\ref{cond:t5}.}

Thus, by induction, we can construct $\cT_1, \cT_2, \ldots, \cT_K$ satisfying (\ref{step:step1'}:a)$_{r}$--(\ref{step:step1'}:d)$_{r}$ for $r=K$ with probability at least $(1 - n^{-C^{2/3}})^K \geq 1 - K n^{-C^{2/3}}$.
Now we fix $\cT_1 , \dots , \cT_K$ satisfying (\ref{step:step1'}:a)$_{K}$--(\ref{step:step1'}:d)$_{K}$.


\begin{pf-step}\label{step:step1'new}
Covering remaining edges incident to a vertex in $C_1$ using triangles containing an edge of $G[V_1, V_2]$.
\end{pf-step}

For each $c \in C_1$, let $G_c$ be a random bipartite subgraph of $G[V_1,V_2]$ with the bipartition $\{ W_1(c) , W_2(c) \}$, where every edge $(u,v) \in W_1(c) \times W_2(c)$ is in $G_c$ if and only if $\{u,v,c\} \in \cB_c (p)$. 
Exposing triangles in $\bigsqcup_{c \in C_1} \cB_c(p)$,
by Lemma~\ref{lem:coverdown_reservoir}\COMMENT{We apply Lemma~\ref{lem:coverdown_reservoir} with $\{ G[W_1(c) , W_2(c)] \}_{c \in C_1}$,  $n-1$, $2 \gamma$ playing the roles of $\{H_i\}_{i \in m}$, $m$, $\rho$ respectively. Note that $\eta$ in Lemma~\ref{lem:coverdown_reservoir} can be taken arbitrarily small as $G[W_1(c) , W_2(c)]$ is a complete bipartite graph for every $c \in C_1$. We now check whether the conditions (i)--(iv) of Lemma~\ref{lem:coverdown_reservoir} are satisfied. By \ref{cond:t3} and \ref{cond:t4} for $r = K$, we know that for $s \in [2]$ and $c \in C_1$,  $|W_s(c)| \ge 0.8 \gamma n \ge (2\gamma)^{4/3}n$, and $|W_s(c) \cap W_s(c')| \le (2 \gamma)^2 n$ for all $c' \in C_1$ which are not in the same block as $c$, where the size of each block is at most $n/K \le 2 \gamma^4 n \leq (2 \gamma)^3 n$. Moreover by \ref{cond:t5} for $r = K$, $v$ is in at most $K (2 \gamma n/K) \le 2(2 \gamma) n$ sets in $W_s(c)$ for $c \in C_1$ and $s \in [2]$.}
, with probability at least $1 - n^{-C^{1/4}}$, for all $c \in C_1$, there exists a perfect matching $M_c$ of $G_c$ such that the matchings in $\{ M_c \}_{c \in C_1}$ are edge-disjoint. Now let $\cT' \coloneqq \{ \{u,v,c\} \: : \: uv \in E(M_c),\:c \in C_1 \} \subseteq \bigsqcup_{c \in C_1}\cB_c(p)$.
Note that\COMMENT{since $M_c$ is a perfect matching in the bipartite graph $G[W_1(c) , W_2(c)]$} the triangles in $\cT'$ are edge-disjoint from the triangles in $\cT_1 \cup \dots \cup \cT_{K}$. 

Thus, the subgraph $G' \coloneqq G - \bigcup_{i=1}^{K}\bigcup_{T \in \cT_i} E(T) - \bigcup_{T \in \cT'}E(T)$ of $G$ satisfies the following with probability at least $1 - n^{-C^{1/5}}$. 
\begin{enumerate}[label=(\ref*{step:step1'}:{\alph*}'), topsep = 6pt]
    \item \label{noedgesC1} $d_{G'}(c) = 0$ for every $c \in C_1$, and $N_{G'}(c) = V_1 \cup V_2$ for every $c \in C_2$.
    
    \item \label{samenumG'Vi} $e(G'[V_1]) = e(G'[V_2]) \in [\gamma n^2 / 4 , \gamma n^2]$ by~\ref{cond:t3} for $r = K$.\COMMENT{Since for each $s \in [2]$ and $c \in C_1$, there are $(n - |W_s(c)|)/2$ triangles of the form $\{u,v,c\}$ with $u,v \in V_s$, we have $e(G[V_s]) - e(G'[V_s]) = \sum_{c \in C_1} (n - |W_s(c)|)/2$, implying that $e(G'[V_s]) = \sum_{c \in C_1}|W_s(c)|/2$. Since for each $c \in C_1$, $|W_1(c)| = |W_2(c)| \in [0.8 \gamma n, 2 \gamma n]$ by~\ref{cond:t3} for $r = K$, this implies that $e(G'[V_1]) = e(G'[V_2]) \in [0.25 \gamma n^2,  \gamma n^2]$.}
    
    \item \label{boundonDegG'[V_s]}For every $j \in [2]$, $\gamma n / 2 \leq \delta(G'[V_j]) \leq \Delta(G'[V_j]) \leq 2 \gamma n$ by~\ref{cond:t5} for $r =K$.
    
    \item \label{cond:mindegG'} For every $j \in [2]$ and $v \in V_j$, $d_{G'}(v, V_{3-j}) \geq (1 - 2\gamma)n$ by~\ref{cond:t5} for $r =K$. \COMMENT{Since $v$ is in at most $2 \gamma n$ sets in $\{ W_s(c) \}_{c \in C_1}$ by~\ref{cond:t5} for $r = K$, it follows that $v$ is contained in at most $2 \gamma n$ matchings among $\{ M_c \}_{c \in C_1}$, thus at most $2 \gamma n$ edges incident to $v$ are removed from $G[V_1, V_2]$ by removing $\bigcup_{T \in \cT'}E(T)$ (for constructing $G'$).}
\end{enumerate}

Now we fix $\cT'$ satisfying  (\ref{step:step1'}:a')--(\ref{step:step1'}:d').

\begin{pf-step}\label{step:step2'}
Covering all remaining edges in $G[V_1] \cup G[V_2]$ using edge-disjoint triangles containing a vertex in $C_2' \subseteq C_2$.
\end{pf-step}

Let $q \coloneqq \lceil \gamma^{2/3} n \rceil$.
By~\ref{boundonDegG'[V_s]} and Vizing's theorem, both $G'[V_1]$ and $G'[V_2]$ admit proper $\lfloor 2 \gamma n + 1 \rfloor$-edge-colourings,
so we can properly (and equitably) colour the edges in $G'[V_1]$ and $G'[V_2]$ with $q$ colours using Lemma~\ref{lem:equitable}.
By~\ref{samenumG'Vi} and~\ref{boundonDegG'[V_s]}, $e(G'[V_1]) = e(G'[V_2]) \in [\gamma n^2/4, \gamma n^2]$, so there exists an integer $\ell$ such that for every $s \in [2]$, every colour class of $G'[V_s]$ has size either $\ell$ or $\ell + 1$, where $\ell \in [\gamma^{1/3} n/6 , \gamma^{1/3} n]$.\COMMENT{as the average size of colour classes is $e(G'[V_s]) / q \in [\gamma^{1/3} n / 5 , \gamma^{1/3} n]$}
For $s \in [2]$, let $M_1^s , \dots , M_q^s$ be the colour classes of $G'[V_s]$ such that $|M_i^1| = |M_i^2|$ for every $i \in [q]$.
Recall that we assumed $C_2' \subseteq C_2$ is a subset of size $\lfloor \eps n \rfloor$, thus $|C_2'| \geq \ell + 1$.

Now we aim to construct a set $\cT''$ of edge-disjoint triangles in $G'$ covering all of the edges in $E(G'[V_1]) \cup E(G'[V_2])$ such that $\cT'' \subseteq \bigsqcup_{c \in C_2'} \cA_c(p)$. We will construct $\cT''$ as the disjoint union of $\cT_r''$ for $r \in [q]$, where the triangles in  $\cT_r''$ cover all of the edges in $M_r^1 \cup M_r^2$ (thus the triangles in $\cT''$ cover all of the edges in $\bigcup_{r=1}^q(M_r^1 \cup M_r^2) = E(G'[V_1]) \cup E(G'[V_2])$). To construct $\cT_r''$ for $r \in [q]$, let $|M_r^1| = |M_r^2| = n_r \in \{\ell, \ell+1\}$, and let $\cS_{r} \coloneqq \{ \{u,v,c\} : uv \in M_r^1 \cup M_r^2, c \in C_2' \}$. Note that $\bigsqcup_{r \in [q]} \cS_{r} \subseteq \bigsqcup_{c \in C_2'} \cA_c$. For every $r \in [q]$, let $C_2'(r) \subseteq C_2'$ be any subset of size $n_r$. 

For every $r \in [q]$, we will inductively construct $\cT_r''$ satisfying the following property with probability at least $1 - n^{-C^{1/4}}$ by exposing triangles in $\cS_r (p)$.

\begin{enumerate}[label=($\mathrm{Q}_r$), topsep = 6pt]
    \item\label{paragraph:step2} For every $uv \in M_r^1$ there exists a unique $w_{uv} \in C_2'(r)$ and for every $u'v' \in M_r^2$, there exists a unique $w'_{u'v'} \in C_2'(r)$, such that the set $\cT_r''$ consists of $2n_r$ edge-disjoint triangles of the form
    \begin{equation*}
        \{ \{u,v, w_{uv}\}, \{u',v', w'_{u'v'}\} : uv \in M_r^1,\:u'v' \in M_r^2 \} \subseteq \cS_{r}(p),
    \end{equation*}
    and they are edge-disjoint from the triangles in $\bigcup_{j=1}^{r-1} \bigcup_{T \in \cT_j''}E(T)$.
\end{enumerate}



To that end, fix $s \in [q]$. Suppose we have already constructed $\cT_r''$ satisfying ~\ref{paragraph:step2} for $r \in [s-1]$. We now construct $\cT_s''$ satisfying ~\ref{paragraph:step2} for $r=s$.

Let $G'_{s} \coloneqq G' - \bigcup_{r=1}^{s-1} \bigcup_{T \in \cT_r''}E(T)$.  Since $\cT_r''$ satisfies ~\ref{paragraph:step2} for $r \le s-1$ and by~\ref{boundonDegG'[V_s]}, the following holds for every $c \in C'_2$, $j \in [2]$, and $v \in V_j$.
\COMMENT{every vertex $v \in V_j$ is in at most $\Delta(G'[V_j])$ sets in $\{ V(M_r^1) \cup V(M_r^2) \}_{r \in [s-1]}$, so $d_{G_{s}'}(v , C_2') \geq |C_2'| - \Delta(G'[V_j])$. Since $|C_2'| = \lfloor \eps n \rfloor$ and $\Delta(G'[V_j]) \leq 2 \gamma n \leq \frac{4 \gamma |C_2'|}{\eps} \le \gamma^{2/3} |C_2'|$, we have $d_{G_{s}'}(v , C_2') \geq |C_2'| - \Delta(G'[V_j]) \geq (1- \gamma^{2/3})|C_2'|$}
\begin{align}\label{eqn:deg_c_i-1}
d_{G_{s}'}(c , V_j) \geq n - 2q \geq (1 - 3 \gamma^{2/3}) n 
\:\:,\:\:
d_{G_{s}'}(v , C_2') \geq |C_2'| - \Delta(G'[V_j]) \geq (1 - \gamma^{2/3}) |C_2'|. \tag{$\mathrm{Q}'_s$}
\end{align}

For every $j \in [2]$, let $G_{\rm aux}^{j,s}$ be an auxiliary bipartite graph such that $V(G_{\rm aux}^{j,s}) \coloneqq C_2'(s) \cup M_s^j$ and
$E(G_{\rm aux}^{j,s}) \coloneqq \{ \{c,uv\} \: : \: c\in C_2'(s),\:uv \in M_s^j,\: uc,vc \in E(G_{s}')\}$.
Then by~\eqref{eqn:deg_c_i-1}, $G_{\rm aux}^{j,s}$ is a bipartite graph such that each part has size $n_s = \Omega(\gamma^{1/3} n)$ and 
$\delta(G_{\rm aux}^{j,s}) \geq n_s - 2q \geq (1 - 14 \gamma^{1/3})n_s$.\COMMENT{since each vertex in $C_2'(s)$ is not adjacent (in $G'_s$) to at most $2q$ vertices in $V(M_s^j)$, and for any $uv \in M_s^j$ either $u$ or $v$ is not adjacent (in $G'_s$) to at most $2 \Delta(G'[V_s]) \leq 4 \gamma n < 2q$ vertices in $C_2'(s)$ by \ref{boundonDegG'[V_s]}, so $\delta(G_{\rm aux}^{j,s}) \geq n_s - 2q$. For the last inequality note that $q/n_s \le (\lceil \gamma^{2/3} n \rceil)/\ell \le (\lceil \gamma^{2/3} n \rceil)/(\gamma^{1/3}n/6) \leq 7 \gamma^{1/3}$.} 
Now for $j \in [2]$, let $H_{\rm aux}^{j,s}$ be the random subgraph of $G_{\rm aux}^{j,s}$ such that every edge $\{c, uv\} \in E(G_{\rm aux}^{j,s})$ is in $H_{\rm aux}^{j,s}$ if and only if $\{u,v,c \} \in \cS_s(p)$. 
By Lemma~\ref{lem:manymatchings}\COMMENT{with $\gamma^{1/3}C, 14 \gamma^{1/3}$ playing the roles of $C, \eps$ respectively}, with probability at least $1 - n^{-C^{1/4}}$, there exist perfect matchings $M_{\rm aux}^{1,s}$ and $M_{\rm aux}^{2,s}$ in $H_{\rm aux}^{1,s}$ and $H_{\rm aux}^{2,s}$ respectively such that for every $j \in [2]$, $M_{\rm aux}^{j,s}$ corresponds to a set $\cT_{\rm aux}^{j,s} \subseteq \cS_s(p)$ of edge-disjoint triangles in $G'_s$.
Let $\cT_s'' \coloneqq \cT_{\rm aux}^{1,s} \cup \cT_{\rm aux}^{2,s}$. Then the triangles in $\cT_s''$ are edge-disjoint from the triangles in $\cT_1'' \cup \dots \cup \cT_{s-1}''$ by the definition of $G_s'$. Hence $\cT_s''$ satisfies~\ref{paragraph:step2} for $r = s$, as desired. Thus, by induction, $\cT_r''$ satisfies ~\ref{paragraph:step2} for all $r \in [q]$.

Let $\cT'' \coloneqq \bigcup_{r=1}^{q} \cT_r''$ and let $G'' \coloneqq G_{q+1}' = G' - \bigcup_{T \in \cT''}E(T)$. Then $G''$ satisfies the following.
\begin{enumerate}[label=(\ref*{step:step2'}:{\alph*}'), topsep = 6pt]
    \item \label{colordegG''} For every $c \in C_1$, $d_{G''}(c) = 0$  and for every $c \in C_2 \setminus C_2'$, $N_{G''}(c) = V_1 \cup V_2$ by ~\ref{noedgesC1}.
    
    \item \label{degC2'toVj} For every $c \in C_2'$ and $j \in [2]$, we have $d_{G''}(c , V_j) \ge (1 - 3 \gamma^{2/3})n$ by~\eqref{eqn:deg_c_i-1} (with $s = q+1$).
    
    \item \label{degG''crossing} $e(G''[V_1]) = e(G''[V_2]) = 0$ (since $\cT''$ covers all of the edges in $G'[V_1]) \cup G'[V_2]$) and $d_{G''}(v , V_{3-j}) \geq (1 - 2 \gamma)n$ by \ref{cond:mindegG'}.
\end{enumerate}

\begin{pf-step}\label{step:step3'}
Covering all remaining edges of $G$ incident to each vertex of $C'_2$ using triangles containing an edge of $G[V_1,V_2]$ and finishing the proof.
\end{pf-step}

In this step we cover the remaining edges of $G$ incident to $C_2'$ with a set $\cT'''$ of edge-disjoint triangles.

Note that for every $c \in C_2'$, we have $d_{G'}(c, V_1) = d_{G'}(c, V_2)$\COMMENT{as $N_{G'}(c) \cap V_2 = V_2$ and $N_{G'}(c) \cap V_1 = V_1$} by ~\ref{noedgesC1}. This combined with the fact that $\cT_r''$ satisfies ~\ref{paragraph:step2} for all $r \in [q]$,\COMMENT{so whenever we remove the triangles in $\cT_r''$ from $G'$ for some $r \in [q]$, each vertex $c \in C_2'$ either loses 2 edges to both $V_1$ and $V_2$ or it loses no edges at all, thus divisibility is maintained.} we have $|N_{G''}(c) \cap V_1| = |N_{G''}(c) \cap V_2|$ for every $c \in C_2'$. By ~\ref{degC2'toVj} and ~\ref{degG''crossing}, we can apply Lemma~\ref{lem:manymatchings}\COMMENT{We find the desired matchings $N_t$ one by one for $t \in C_2'$. Let $c \in C_2'$. By ~\ref{degC2'toVj}, we have $|N_{G''}(c) \cap V_1|, |N_{G''}(c) \cap V_2| \ge (1 - 3 \gamma^{2/3})n$. By ~\ref{degG''crossing}, $G''[N_{G''}(c) \cap V_1, N_{G''}(c) \cap V_2]$ has minimum degree at least $|N_{G''}(c) \cap V_1|-2\gamma n$. Thus, after removing at most $|C_2'|-1 \le \eps n$ matchings $N_t$, $G''[N_{G''}(c) \cap V_1, N_{G''}(c) \cap V_2]$ has minimum degree at least $|N_{G''}(c) \cap V_1|-2\gamma n-\eps n \ge (1 - 4\gamma-2 \eps )|N_{G''}(c) \cap V_1|$, so we can apply Lemma~\ref{lem:manymatchings} with $C/2$ playing the role of $C$ to obtain a perfect matching $N_c$ between $N_{G''}(c) \cap V_1$ and $N_{G''}(c) \cap V_2$ such that $\{u,v,c \} \in \cB_{c}(p)$ for every $uv \in N_{c}$.} repeatedly for every $c \in C_2'$, by exposing triangles in $\cB_{c}(p)$ to obtain a set $\{ N_{c} \}_{c \in C_2'}$ of edge-disjoint matchings in $G''$  with probability at least $1 - n^{-C^{1/3}}$, where $N_{c}$ is a perfect matching between $N_{G''}(c) \cap V_1$ and $N_{G''}(c) \cap V_2$ such that $\{u,v,c \} \in \cB_{c}(p)$ for every $uv \in N_{c}$.

Let $\cT''' \coloneqq \{ \{u,v,c \} : c \in C_2',\:uv \in N_{c} \}$, and let $\cT \coloneqq  \bigcup_{i=1}^{K} \cT_i \cup \cT' \cup \cT'' \cup \cT'''$. Let $H \coloneqq G'' - \bigcup_{T \in \cT'''}E(T) = G - \bigcup_{T \in \cT}E(T)$. We will show that $H$ is the desired tripartite graph satisfying ~\ref{deg_red2} and ~\ref{reg_red2}. 

Since $H = G'' - \bigcup_{T \in \cT'''}E(T)$, by ~\ref{colordegG''} we have $d_{H}(c) = 0$ for every $c \in C_1$,  and $N_{H}(c) = V_1 \cup V_2$  for every $c \in C_2 \setminus C_2'$. Since the triangles in $\cT'''$ cover all remaining edges between $C_2'$ and $V_1 \cup V_2$, $d_{H}(c) = 0$ for every $c \in C_2'$. Moreover, $e(H[V_1]) = e(H[V_2]) = 0$ by ~\ref{degG''crossing}, so $H$ satisfies ~\ref{deg_red2}.

It remains to check that $H$ satisfies ~\ref{reg_red2}. To that end, we claim that the divisibility condition $d_H(v , V_{3-j}) = d_H(v , C_2 \setminus C_2')$ is satisfied for every $j \in [2]$ and $v \in V_j$. Indeed, note that  $d_G(v , C_1 \cup C_2) = d_G(v , V_1 \cup V_2) = 2n-1$ for every $v \in V_1 \cup V_2$, triangles in $\cT$ are edge-disjoint, and every triangle in $\cT$ contains exactly one vertex in $C_1 \cup C_2$ and two other vertices in $V_1 \cup V_2$.\COMMENT{Thus, removing any triangle of $\cT$ from $G$ incident to a vertex $v \in V_1 \cup V_2$ reduces both $d(v , V_1 \cup V_2)$ and $d(v , C_1 \cup C_2)$ by one.} Hence, $d_H(v , C_1 \cup C_2) = d_H(v , V_1 \cup V_2)$. Moreover, by ~\ref{deg_red2}, for every  $i \in [2]$ and $v \in V_i$, $d_{H}(v , C_1 \cup C_2) = d_H(v , C_2 \setminus C_2')  = |C_2|- |C_2'|$ and $d_H (v , V_1 \cup V_2) = d_H(v , V_{3-i})$. Thus, we have $d_H(v , V_{3-i}) = |C_2|- |C_2'|$ for every $i \in [2]$ and $v \in V_i$. Thus $H$ satisfies ~\ref{reg_red2}, as desired.
\end{proof}

\section*{Acknowledgements}
We are grateful to the referees for thoughtful comments and suggestions.

\bibliographystyle{amsabbrv}
\bibliography{cg}

\providecommand{\bysame}{\leavevmode\hbox to3em{\hrulefill}\thinspace}
\providecommand{\MR}{\relax\ifhmode\unskip\space\fi MR }
\providecommand{\MRhref}[2]{%
  \href{http://www.ams.org/mathscinet-getitem?mr=#1}{#2}
}
\providecommand{\href}[2]{#2}
\begin{thebibliography}{10}

\bibitem{AA20}
N.~Alon and S.~Assadi, \emph{Palette sparsification beyond {$(\Delta + 1)$}
  vertex coloring}, Approximation, randomization, and combinatorial
  optimization. {A}lgorithms and techniques, LIPIcs. Leibniz Int. Proc.
  Inform., vol. 176, Schloss Dagstuhl. Leibniz-Zent. Inform., Wadern, 2020,
  Art. No. 6, 22.

\bibitem{AKS1997}
N.~Alon, J.-H. Kim, and J.~Spencer, \emph{Nearly perfect matchings in regular
  simple hypergraphs}, Israel J. Math. \textbf{100} (1997), 171--187.

\bibitem{AS16}
N.~Alon and J.~Spencer, \emph{The probabilistic method}, fourth ed., Wiley
  Series in Discrete Mathematics and Optimization, John Wiley \& Sons, Inc.,
  Hoboken, NJ, 2016.

\bibitem{AY05}
N.~Alon and R.~Yuster, \emph{On a hypergraph matching problem}, Graphs Combin.
  \textbf{21} (2005), 377--384.

\bibitem{ACK19}
S.~Assadi, Y.~Chen, and S.~Khanna, \emph{Sublinear algorithms for
  {$(\Delta+1)$} vertex coloring}, Proceedings of the {T}hirtieth {A}nnual
  {ACM}-{SIAM} {S}ymposium on {D}iscrete {A}lgorithms, SIAM, Philadelphia, PA,
  2019, 767--786.

\bibitem{BGKLMO20}
B.~Barber, S.~Glock, D.~K\"{u}hn, A.~Lo, R.~Montgomery, and D.~Osthus,
  \emph{Minimalist designs}, Random Structures Algorithms \textbf{57} (2020),
  47--63.

\bibitem{barber2017fractional}
B.~Barber, D.~K\"{u}hn, A.~Lo, R.~Montgomery, and D.~Osthus, \emph{Fractional
  clique decompositions of dense graphs and hypergraphs}, J. Combin. Theory
  Ser. B \textbf{127} (2017), 148--186.

\bibitem{Bo01}
B.~Bollob\'{a}s, \emph{Random graphs}, second ed., Cambridge Studies in
  Advanced Mathematics, vol.~73, Cambridge University Press, Cambridge, 2001.

\bibitem{Br73}
L.~M. Br\`{e}gman, \emph{Certain properties of nonnegative matrices and their
  permanents}, Dokl. Akad. Nauk SSSR \textbf{211} (1973), 27--30.

\bibitem{CH16}
C.~J. Casselgren and R.~H{\"a}ggkvist, \emph{Coloring complete and complete
  bipartite graphs from random lists}, Graphs Combin. \textbf{32} (2016),
  533--542.

\bibitem{Eg81}
G.~P. Egorychev, \emph{The solution of van der {W}aerden's problem for
  permanents}, Adv. in Math. \textbf{42} (1981), 299--305.

\bibitem{EGJ20}
S.~Ehard, S.~Glock, and F.~Joos, \emph{Pseudorandom hypergraph matchings},
  Combin. Probab. Comput. \textbf{29} (2020), 868--885.

\bibitem{Erdos81}
P.~Erd\H{o}s, \emph{On the combinatorial problems which {I} would most like to
  see solved}, Combinatorica \textbf{1} (1981), 25--42.

\bibitem{ER64}
P.~Erd\H{o}s and A.~R\'{e}nyi, \emph{On random matrices}, Magyar Tud. Akad.
  Mat. Kutat\'{o} Int. K\"{o}zl. \textbf{8} (1964), 455--461 (1964).

\bibitem{ER66}
\bysame, \emph{On the existence of a factor of degree one of a connected random
  graph}, Acta Math. Acad. Sci. Hungar. \textbf{17} (1966), 359--368.

\bibitem{Fa81}
D.~I. Falikman, \emph{Proof of the van der {W}aerden conjecture on the
  permanent of a doubly stochastic matrix}, Mat. Zametki \textbf{29} (1981),
  931--938, 957.

\bibitem{FJS20}
A.~Ferber, V.~Jain, and B.~Sudakov, \emph{Number of 1-factorizations of regular
  high-degree graphs}, Combinatorica \textbf{40} (2020), 315--344.

\bibitem{FKNP21}
K.~Frankston, J.~Kahn, B.~Narayanan, and J.~Park, \emph{Thresholds versus
  fractional expectation-thresholds}, Ann. of Math. \textbf{194} (2021),
  475--495.

\bibitem{Fr99}
E.~Friedgut, \emph{Sharp thresholds of graph properties, and the {$k$}-sat
  problem}, J. Amer. Math. Soc. \textbf{12} (1999), 1017--1054, With an
  appendix by Jean Bourgain.

\bibitem{Fr05}
\bysame, \emph{Hunting for sharp thresholds}, Random Structures Algorithms
  \textbf{26} (2005), 37--51.

\bibitem{Ga95}
F.~Galvin, \emph{The list chromatic index of a bipartite multigraph}, J.
  Combin. Theory Ser. B \textbf{63} (1995), 153--158.

\bibitem{GKLO16}
S.~Glock, D.~K{\"u}hn, A.~Lo, and D.~Osthus, \emph{The existence of designs via
  iterative absorption: Hypergraph {$F$}-designs for arbitrary {$F$}}, Mem.
  Amer. Math. Soc. (to appear), 95pp.

\bibitem{GKO21}
S.~Glock, D.~K\"{u}hn, and D.~Osthus, \emph{Extremal aspects of graph and
  hypergraph decomposition problems}, Surveys in combinatorics 2021, London
  Math. Soc. Lecture Note Ser., vol. 470, Cambridge Univ. Press, Cambridge,
  2021, 235--265.

\bibitem{HJ97}
R.~H\"{a}ggkvist and J.~Janssen, \emph{New bounds on the list-chromatic index
  of the complete graph and other simple graphs}, Combin. Probab. Comput.
  \textbf{6} (1997), 295--313.

\bibitem{JP2022}
V.~Jain and H.~T. Pham, \emph{Optimal thresholds for {L}atin squares, {S}teiner
  {T}riple {S}ystems, and edge colorings}, arXiv:2212.06109 (2022).

\bibitem{JLR00}
S.~Janson, T.~{\L}uczak, and A.~Rucinski, \emph{Random graphs},
  Wiley-Interscience Series in Discrete Mathematics and Optimization,
  Wiley-Interscience, New York, 2000.

\bibitem{J06}
A.~Johansson, \emph{Triangle factors in random graphs},
  \url{http://citeseerx.ist.psu.edu/viewdoc/download?doi=10.1.1.98.2746&rep=rep1&type=pdf},
  2006, manuscript.

\bibitem{JKV08}
A.~Johansson, J.~Kahn, and V.~Vu, \emph{Factors in random graphs}, Random
  Structures Algorithms \textbf{33} (2008), 1--28.

\bibitem{Ke14}
P.~Keevash, \emph{The existence of designs}, arXiv:1401.3665 (2014).

\bibitem{K2022}
\bysame, \emph{The optimal edge-colouring threshold}, arXiv:2212.04397 (2022).

\bibitem{KKO15}
F.~Knox, D.~K\"{u}hn, and D.~Osthus, \emph{Edge-disjoint {H}amilton cycles in
  random graphs}, Random Structures Algorithms \textbf{46} (2015), 397--445.

\bibitem{KO13}
D.~K\"{u}hn and D.~Osthus, \emph{Hamilton decompositions of regular expanders:
  a proof of {K}elly's conjecture for large tournaments}, Adv. Math.
  \textbf{237} (2013), 62--146.

\bibitem{LL14}
N.~Linial and Z.~Luria, \emph{An upper bound on the number of high-dimensional
  permutations}, Combinatorica \textbf{34} (2014), 471--486.

\bibitem{Lo07}
L.~Lov\'{a}sz, \emph{Combinatorial problems and exercises}, second ed., AMS
  Chelsea Publishing, Providence, RI, 2007.

\bibitem{LS19}
Z.~Luria and M.~Simkin, \emph{On the threshold problem for {L}atin boxes},
  Random Structures Algorithms \textbf{55} (2019), 926--949.

\bibitem{mcdiarmid1972}
C.~J.~H. McDiarmid, \emph{The solution of a timetabling problem}, J. Inst.
  Math. Appl. \textbf{9} (1972), 23--34.

\bibitem{MW08}
B.~D. McKay and I.~M. Wanless, \emph{A census of small {L}atin hypercubes},
  SIAM J. Discrete Math. \textbf{22} (2008), 719--736.

\bibitem{PP22}
J.~Park and H.~T. Pham, \emph{A proof of the {K}ahn--{K}alai conjecture},
  arXiv:2203.17207 (2022).

\bibitem{rodl1985}
V.~R\"{o}dl, \emph{On a packing and covering problem}, European J. Combin.
  \textbf{6} (1985), 69--78.

\bibitem{SSS22}
A.~Sah, M.~Sawhney, and M.~Simkin, \emph{Threshold for {S}teiner triple
  systems}, arXiv:2204.03964 (2022).

\bibitem{Sc14}
U.~Schauz, \emph{Proof of the list edge coloring conjecture for complete graphs
  of prime degree}, Electron. J. Combin. (2014), P3--43.

\bibitem{Si17}
M.~Simkin, \emph{$\left(n, k, k-1\right)$-{S}teiner systems in random
  hypergraphs}, arXiv:1711.01975 (2017).

\bibitem{vizing1965}
V.~G. Vizing, \emph{The chromatic class of a multigraph}, Cybernetics
  \textbf{1} (1965), 32--41.

\bibitem{Yu07}
R.~Yuster, \emph{Combinatorial and computational aspects of graph packing and
  graph decomposition}, Computer Science Review \textbf{1} (2007), 12--26.

\end{thebibliography}

\end{document}